\documentclass[reqno]{amsart}

\usepackage[mathscr]{euscript}
\usepackage{amssymb}
\usepackage{graphicx}
\usepackage[shortlabels]{enumitem}
\usepackage{mathtools}
\usepackage{booktabs}
\usepackage{tikz}
\usepackage{tikz-cd}
\usetikzlibrary{positioning,cd,arrows, patterns}
\usepackage{caption}
\usepackage{subcaption}
\usepackage{ragged2e}
\usepackage{a4wide}
\usepackage{hyperref}
\hypersetup{colorlinks}
\usepackage{amsbsy}
\usepackage{xargs}
\usepackage{array}
\usepackage{float}

\hypersetup{
    unicode=false,          
    pdftoolbar=true,        
    pdfmenubar=true,        
    pdffitwindow=false,     
    pdfnewwindow=true,      
    colorlinks=true,       
    linkcolor=blue,          
    citecolor=red,        
    filecolor=violet,      
    urlcolor=violet       
}

\theoremstyle{plain}
\newtheorem{theorem}{Theorem}[section]
\newtheorem{lemma}[theorem]{Lemma}
\newtheorem{proposition}[theorem]{Proposition}
\newtheorem{corollary}[theorem]{Corollary}

\theoremstyle{remark}
\newtheorem{remark}[theorem]{Remark}

\theoremstyle{definition}
\newtheorem{definition}[theorem]{Definition}
\newtheorem{assumption}[theorem]{Assumption}
\newtheorem{example}[theorem]{Example}

\newcommand\C{\mathbb{C}}

\newcommand\Z{\mathbb{Z}}

\newcommand\bfx{\mathbf{x}}

\newcommand\cA{\mathcal{A}}

\newcommand{\clusterfont}{\mathcal}
\newcommand{\dynkinfont}{\mathsf}

\newcommand{\quiver}{\clusterfont{Q}}
\newcommand{\qbasis}{\clusterfont{B}}

\newcommand{\dynA}{\dynkinfont{A}}
\newcommand{\dynB}{\dynkinfont{B}}
\newcommand{\dynC}{\dynkinfont{C}}
\newcommand{\dynD}{\dynkinfont{D}}
\newcommand{\dynE}{\dynkinfont{E}}
\newcommand{\dynF}{\dynkinfont{F}}
\newcommand{\dynG}{\dynkinfont{G}}
\newcommand{\dynX}{\dynkinfont{X}}
\newcommand{\dynY}{\dynkinfont{Y}}

\newcommand{\exdynA}{\widetilde{\dynA}}
\newcommand{\exdynB}{\widetilde{\dynB}}
\newcommand{\exdynC}{\widetilde{\dynC}}
\newcommand{\exdynD}{\widetilde{\dynD}}
\newcommand{\exdynE}{\widetilde{\dynE}}
\newcommand{\exdynF}{\widetilde{\dynF}}
\newcommand{\exdynG}{\widetilde{\dynG}}
\newcommand{\exdynX}{\widetilde{\dynX}}

\newcommand{\seed}{\Sigma}
\newcommand{\initialseed}{\Sigma_{t_0}}

\newcommand{\mutation}{\mu}

\newcommand{\field}{\mathbb{F}}

\newcommand{\Roots}{\Phi}

\newcommand{\boundellipse}[3]
{(#1) ellipse (#2 and #3)}

\newcommand{\reducedto}{\succ}

\numberwithin{equation}{section}

\tikzset{ynode/.style = {circle, fill = yellow, inner sep = 2pt, opacity = 0.5}}
\tikzset{gnode/.style = {circle, fill=green, inner sep = 2pt, opacity = 0.5}}
\tikzstyle{vertex}=[draw, circle, fill=black, inner sep = 2pt]
\tikzstyle{double line} = [
double distance = 1.5pt, 
double=\pgfkeysvalueof{/tikz/commutative diagrams/background color}
]
\tikzstyle{triple line} = [
double distance = 2pt, 
double=\pgfkeysvalueof{/tikz/commutative diagrams/background color}
]

\makeatletter 
\tikzset{curlybrace/.style={rounded corners=2pt,line cap=round}}%
\pgfkeys{
/curlybrace/.cd,%
tip angle/.code     =  \def\cb@angle{#1},
/curlybrace/.unknown/.code ={\let\searchname=\pgfkeyscurrentname
                              \pgfkeysalso{\searchname/.try=#1,
                              /tikz/\searchname/.retry=#1}}}  
\def\curlybrace{\pgfutil@ifnextchar[{\curly@brace}{\curly@brace[]}}%

\def\curly@brace[#1]#2#3#4{%
\pgfkeys{/curlybrace/.cd,
tip angle = 0.75}%
\pgfqkeys{/curlybrace}{#1}%
\ifnum 1>#4 \def\cbrd{0.05} \else \def\cbrd{0.075} \fi
\draw[/curlybrace/.cd,curlybrace,#1]  (#2:#4-\cbrd) -- (#2:#4) arc (#2:{(#2+#3)/2-\cb@angle}:#4) --({(#2+#3)/2}:#4+\cbrd) coordinate (curlybracetipn);
\draw[/curlybrace/.cd,curlybrace,#1] ({(#2+#3)/2}:#4+\cbrd) -- ({(#2+#3)/2+\cb@angle}:#4) arc ({(#2+#3)/2+\cb@angle} :#3:#4) --(#3:#4-\cbrd)
}
\makeatother

\newlength{\myline}
\newcommandx*{\triplearrow}[4][1=0, 2=1]{
  \draw[double distance=3\myline,#3] #4;
  \draw[shorten <=#1\myline,shorten >=#2\myline,#3] #4
}

\title{On folded cluster patterns of affine type}

\author{Byung Hee An}
\email{anbyhee@knu.ac.kr}
\address{Department of Mathematics Education, Kyungpook National University, Republic of Korea}
\author{Eunjeong Lee}
\email{eunjeong.lee@ibs.re.kr}
\address{Center for Geometry and Physics, Institute for Basic Science (IBS), Pohang 37673, Republic of Korea}

\begin{document}

\begin{abstract}
A cluster algebra is a commutative algebra whose structure is decided by a skew-symmetrizable matrix or 
a quiver.
When a skew-symmetrizable matrix is invariant under an action of a finite group and this action is 
\emph{admissible}, the \emph{folded} cluster algebra is obtained from the 
original one. Any cluster algebra of non-simply-laced affine type can 
be obtained by folding a cluster algebra of simply-laced affine type with a 
specific $G$-action. 
In this paper, we study the combinatorial properties of quivers in the cluster 
algebra of affine type. 
We prove that for any quiver of simply-laced affine type, $G$-invariance 
and $G$-admissibility are equivalent.
This leads us to prove that the set of $G$-invariant seeds forms the folded cluster pattern. 
\end{abstract}

\maketitle

\tableofcontents

\section{Introduction}
Cluster algebras are commutative algebras introduced and studied by 
Fomin--Zelevinsky and Berenstein--Fomin--Zelevinsky in a series of 
articles~\cite{FZ1_2002, FZ2_2003, BFZ3_2005, FZ4_2007}. They were invented in 
the context of total positivity and dual canonical bases in Lie theory; since 
then, connections and applications have been discovered in diverse 
areas of mathematics. A cluster algebra is a commutative algebra with certain 
generators, called 
\emph{cluster variables}, defined recursively. The recursive structure of 
cluster algebra is encoded in the combinatorial datum of an
exchange matrix, which is a skew-symmetrizable integer 
matrix (see Definition~\ref{definition:seed and exchange matrix}).
More precisely, a cluster algebra is defined by a 
bunch of \emph{seeds} and each seed consists of cluster variables and an 
exchange matrix. The structure of this cluster (called the \emph{seed 
pattern}) is decided recursively via an operation (called the \emph{mutations}) 
given by the exchange matrix in each seed (see Section~\ref{section_preliminaries}).

A cluster algebra is said to be of \emph{finite type} if the cluster pattern has 
only a finite number of seeds. Fomin and Zelevinsky~\cite{FZ2_2003} showed that 
the cluster algebras of finite type can be classified in terms of the Dynkin 
diagrams of finite-dimensional simple Lie algebras.
A wider class of cluster algebras consists of cluster algebras of \emph{finite 
mutation type}, which have finitely many exchange matrices but are allowed to 
have infinitely many seeds. Felikson, Shapiro, and 
Tumarkin proved in~\cite{FeliksonShapiroTumarkin12_JEMS} that a 
\emph{skew-symmetric} cluster algebra of 
rank $n$ is finite mutation type if and only if $n \le 2$; or it is of surface 
type; or it is one of $11$ exceptional types:
\[
\underbrace{\dynE_6, \dynE_7, \dynE_8}_{\text{finite type}},\quad
\underbrace{\exdynE_6, \exdynE_7,\exdynE_8}_{\text{affine type}}, \quad
\dynE_6^{(1,1)},\dynE_{7}^{(1,1)},\dynE_8^{(1,1)},\dynX_6,\dynX_7.
\]
Here, we notice that cluster algebras of surface type can be of other  
remaining simply-laced Dynkin type: $\dynA_n,\dynD_n,\exdynA_{p,q},\exdynD_n$ 
as provided in~\cite[Table~1]{FST08}.

A skew-symmetric matrix can be considered as the adjacency matrix of a finite 
directed multigraph that does not have directed cycles of length at most $2$. 
We call such directed graph a \emph{quiver} (see Figure~\ref{fig_quivers} for 
examples/non-examples of quivers).
To study \emph{skew-symmetrizable} cluster algebras of finite mutation type, 
Felikson, Shapiro, and Tumarkin used the folding and unfolding procedures of 
cluster algebras in~\cite{FeliksonShapiroTumarkin12_unfoldings}. 
Indeed, they consider a certain symmetry on the quivers and their quotients
which leads to prove that skew-symmetrizable cluster algebras correspond to the 
non-simply-laced Dynkin diagrams are of finite mutation type by folding 
simply-laced Dynkin diagrams. We present in Table~\ref{table:all possible 
foldings} how simply-laced affine Dynkin diagrams and non-simply-laced affine 
Dynkin diagrams are related (also, see 
figures in~Appendix~\ref{appendix_actions_on_Dynkin_diagrams}). Folding procedure produces all non-simply-laced affine Dynkin diagrams using simply-laced affine Dynkin diagrams.
\newcolumntype{?}{!{\vrule width 0.6pt}}
\begin{table}[ht]
{
\setlength{\tabcolsep}{4pt}
\renewcommand{\arraystretch}{1.5}		
\begin{tabular}{c||c?c?c?c?c?c?c?c?c?c?c}
\toprule
$\dynX$ & $\exdynA_{2,2}$ & $\exdynA_{n,n}$ 
& \multicolumn{2}{c?}{$\exdynD_4$ }
&  \multicolumn{2}{c?}{$\exdynD_n$ } 
&  \multicolumn{2}{c?}{$\exdynD_{2n}$ } 
&  \multicolumn{2}{c?}{$\exdynE_6$ }
& $\exdynE_7$ \\
\hline 
$G$ & $\Z/2\Z$& $\Z/2\Z$  
& $(\Z/2\Z)^2$ & $\Z/3\Z$ 
& $\Z/ 2\Z$ & $\Z/2\Z$
& $\Z/2\Z$  & $(\Z/2\Z)^2$
& $\Z/3\Z$ & $\Z/2\Z$ 
& $\Z/2\Z$ \\
\hline 
$\dynY$ & $\exdynA_1$ & $\dynD_{n+1}^{(2)}$ 
& $\dynA_2^{(2)}$ & $\dynD_4^{(3)}$ 
& $\exdynC_{n-2}$ & $\dynA_{2(n-1)-1}^{(2)}$
& $\exdynB_n$ & $\dynA_{2n-2}^{(2)}$
& $\exdynG_2$ & $\dynE_6^{(2)}$
& $\exdynF_4$ \\
\bottomrule
\end{tabular}}
\caption{Foldings appearing in affine Dynkin diagrams. For 
$(\dynX, G, \dynY)$ in each column, the quiver of type~$\dynX$ is globally 
foldable with respect to $G$, and the corresponding folded cluster pattern is of 
type~$\dynY$.}
\label{table:all possible foldings}
\end{table}

In this paper, we investigate combinatorial properties of quivers in the seed 
pattern of affine type. To state our main theorem, we 
prepare some terminologies. 
We say two quivers are \emph{mutation equivalent} if one can be obtained from 
the other by applying finitely many mutations.
For a simply-laced Dynkin type $\dynX$, a quiver is \emph{of type~$\dynX$} if 
it is mutation equivalent to a quiver whose underlying graph is the Dynkin 
diagram of $\dynX$. 
For a finite group $G$ acting on the set of vertices of a quiver~$\quiver$, the 
quiver $\quiver$ is \emph{$G$-invariant} if for any $g \in G$, the quiver 
$\quiver$ is 
isomorphic to $g \cdot \quiver$ as a directed graph. A $G$-invariant quiver 
$\quiver$ is \emph{$G$-admissible} if for any two vertices $i$ and $i'$ in the same $G$-orbit, there is no arrow connecting $i$ and $i'$ and whenever there is an arrow $i \to j$ (respectively, $j \to i$), we should have $i' \to j$ (respectively, $j \to i'$). 
See 
Section~\ref{section_inv_and_admiss_of_quivers} for more precise definitions.
In general, a $G$-invariant quiver might not be $G$-admissible as explained in 
Example~\ref{example_non_admissible}. Nevertheless, when we concentrate on 
quivers of affine type, the $G$-invariance ensures the $G$-admissibility, which 
is the main result of the paper.
\begin{theorem}[{Theorem~\ref{thm_invaraint_implies_admissible}}]\label{thm_main}
Let $(\dynX, G, \dynY)$ be a triple given by a column of Table~\ref{table:all 
possible foldings}.
Let $\quiver$ be a quiver of type~$\dynX$. 
If $\quiver$ is $G$-invariant, then it is $G$-admissible. Indeed,  
$G$-invariance and $G$-admissibility are equivalent.
\end{theorem}

The proof of the theorem is provided by type-by-type arguments and we make 
great 
use of the fact that the corresponding cluster algebra is of finite mutation 
type. Because the cluster algebras of affine type are of finite mutation type 
as we already mentioned, using the computer program 
\mbox{SageMath}~\cite{sagemath}, one can get the same result as 
Theorem~\ref{thm_main} for quivers of type~$\exdynE$ or type 
$\exdynA_{n,n},\exdynD_n$ for a given $n$, which is an \emph{experimental 
proof}. 
On the other hand, we provide a \emph{combinatorial proof} by observing the 
combinatorics of quivers. 

We also study an application of Theorem~\ref{thm_main} to the \emph{folded} 
cluster pattern. Let $\field$ be the rational function field with $n$ algebraically independent variables over $\C$. 
Suppose that a finite group $G$ acts on $[n] \coloneqq \{1,\dots,n\}$.
Let $\field^G$ be the field of rational functions in $\#([n]/G)$ independent variables and $\psi \colon \field \to \field^G$ be a surjective homomorphism. A seed $\seed = (\mathbf x, \quiver)$, which is a pair of variables $\mathbf x = (x_1,\dots,x_n)$ in $\field$ and a quiver $\quiver$ on $[n]$, is called \emph{$(G,\psi)$-invariant} or \emph{$(G,\psi)$-admissible} if for any indices $i$ and $i'$ in the same $G$-orbit, we have $\psi(x_i) = \psi(x_{i'})$ and $\quiver$ is $G$-invariant or $G$-admissible, respectively.

For a $(G,\psi)$-admissible seed $\seed$, if the admissibility is preserved under orbit mutations, then one can fold the seed $\seed$, which will be denoted by $\seed^G$. Here, an orbit mutation is a composition of mutations for all vertices in the same $G$-orbit. 
A cluster pattern given by the folded seed $\seed^G$ can be identified with the set of $(G,\psi)$-admissible seeds in the original cluster pattern given by $\seed$. 
Indeed, the folded cluster pattern consists of seeds defined recursively via a sequence of orbit mutations.
We prove that the set of $(G,\psi)$-invariant seeds forms the folded cluster pattern.
\begin{corollary}[{Corollary~\ref{cor_invariant_seeds_form_folded_pattern}}]
Let $(\dynX,G,\dynY)$ be a triple given by a column of Table~\ref{table:all 
possible foldings}.
Let $\initialseed = (\mathbf x, \quiver)$ be a seed. Suppose 
that $\quiver$ is of type~$\dynX$. Define $\psi \colon \field  \to \field^G$ so 
that $\initialseed$ is a $(G, \psi)$-admissible seed. Then, any 
$(G,\psi)$-invariant seed can be reached 
by a sequence of orbit mutations from $\initialseed$. Moreover, the set of $(G,\psi)$-invariant seeds forms the `folded' cluster pattern given by $\initialseed^G$ of $\dynY$ via folding. 
\end{corollary}

This provides an 
answer to the question in \cite[Problem~9.5]{Dupont08} which asks whether any 
$(G,\psi)$-invariant seed can be reached by sequences of orbit mutations 
from the initial seed for the case of cluster algebras of affine type (see Remark~\ref{rmk_Dupont}).
Moreover, this result is useful when studying Lagrangian fillings of 
Legendrians of affine type as exhibited in the forthcoming paper 
An--Bae--Lee~\cite{ABL2021c}.

The paper is organized as follows. 
In Section~\ref{section_preliminaries}, we review the definition of cluster 
algebras and mutations. 
In Section~\ref{section_inv_and_admiss_of_quivers}, we consider the 
$G$-invariance and $G$-admissibility of quivers. 
In Section~\ref{section:proof of admissibility}, we present the proof of the 
main theorem by analyzing each type of quiver. 
In Section~\ref{section_connections_with_CA}, we provide an application of the 
main theorem by considering the \emph{folded} version of cluster algebras and 
cluster patterns.
We describe finite group actions on quivers of affine Dynkin type in 
Appendix~\ref{appendix_actions_on_Dynkin_diagrams}.
\subsection*{Acknowledgement}
The authors would like to great thank Youngjin Bae for his careful reading and 
corrections.
B. H. An was supported by the National Research Foundation of Korea (NRF) grant 
funded by the Korea government (MSIT) (No. 2020R1A2C1A0100320).
E. Lee was supported by the Institute for Basic Science (IBS-R003-D1).

\section{Preliminaries: cluster algebras}
\label{section_preliminaries}
Cluster algebras, introduced by Fomin and Zelevinsky~\cite{FZ1_2002}, are 
commutative algebras with specific generators, called \emph{cluster variables}, 
defined recursively. In this section, we recall basic notions in the theory of 
cluster algebras. For more details, we refer the reader to~\cite{FZ1_2002, 
FZ2_2003}.

Throughout this section, we fix $n \in \Z_{>0}$ and we let $\field$ be the 
rational function field with $n$ algebraically independent variables over $\C$.
We also denote the set $\{1,\dots, n\}$ by simply $[n]$.

\begin{definition}[{cf. \cite{FZ1_2002, FZ2_2003}}]\label{definition:seed and exchange matrix}
A \emph{seed} $\seed = (\bfx, \qbasis)$ is a pair of 
\begin{itemize}
\item a tuple $\mathbf x = (x_1,\dots,x_n)$ of algebraically independent 
generators of $\field$, that is, $\field = \C(x_1,\dots,x_n)$;
\item an $n \times n$ skew-symmetrizable integer matrix $\qbasis = 
(b_{i,j})$, that is, there exist positive integers $d_1,\dots,d_n$ such 
that
$$\textrm{diag}(d_1,\dots,d_n) \cdot \qbasis$$ is a 
skew-symmetric matrix.
\end{itemize}
We call elements $x_1,\dots,x_n$ \emph{cluster variables} and call $\qbasis$ 
\emph{exchange matrix}. 
\end{definition}

In general, cluster variables consist of \emph{unfrozen} 
(or \emph{mutable}) variables and \emph{frozen} variables but we assume the following.

\begin{assumption}\label{assumption:all variables are mutable}
Throughout this paper, we assume that all cluster variables are mutable.
\end{assumption} 

A finite directed multigraph $\quiver$ with the set $[n]$ of vertices is called a \emph{quiver} on $[n]$ if it does not have directed cycles of length at 
most $2$.  
In Figure~\ref{fig_quivers}, we provide examples/non-examples of 
quivers. 
The left directed graph in Figure~\ref{fig_quivers_nonexample} is not a quiver because 
of the one-cycle on the vertex $2$; neither is the right one because it has 
a directed two-cycle connecting vertices $1$ and $2$.
\begin{figure}[ht]
\begin{subfigure}[b]{0.45\textwidth}
\begin{tikzpicture}[baseline=-.5ex]
\begin{scope}
\draw[fill] (0,0) node (A1) {} circle (2pt) (210:1) node (A3) {} circle (2pt) 
(90:1) node (A2) {} circle (2pt) (-30:1) node (A4) {} circle (2pt);
\draw[->] (A2) node[above] {$2$} -- (A1) node[above right] {$1$};
\draw[->] (A3) node[left] {$3$} -- (A2);
\draw[->] (A4) node[right] {$4$} -- (A2);
\draw[->] (A1) -- (A3);
\draw[->] (A3) -- (A4);
\draw[->] (A1) -- (A4);
\end{scope}
\begin{scope}[xshift=3cm]
\draw[fill] (0,0) node (A1) {} circle (2pt) (210:1) node (A3) {} circle (2pt) 
(90:1) node (A2) {} circle (2pt) (-30:1) node (A4) {} circle (2pt);
\draw[->] (A1) -- (A3) node[left] {$3$};
\draw[->] (A2) -- (A3);
\draw[-Implies,double distance=2pt] (A4) -- (A1) node[below] {$1$};
\draw[->] (A3) -- (A4) node[right] {$4$};
\draw[<-] (A4) -- (A2) node[above] {$2$};
\draw[-Implies,double distance=2pt] (A1)--(A2);
\end{scope}
\end{tikzpicture}
\caption{Quivers}\label{fig_quivers_example}
\end{subfigure}
\begin{subfigure}[b]{0.45\textwidth}
\begin{tikzpicture}[baseline=-.5ex]
\draw[fill] 
(210:1) node (A3) {} circle (2pt) 
(90:1) node (A2) {} circle (2pt) 
(-30:1) node (A4) {} circle (2pt);
\draw (A2) node[above] {$1$} (A3) node[below left] {$2$} (A4) node[below right] 
{$3$};
\draw[->] (A2)--(A3);
\draw[->] (A3)--(A4);
\draw[->] (A4)--(A2); 

\draw[->] (A3) to [out=90,in=180,looseness=10] (A3);

\begin{scope}[xshift=3cm]
\draw[fill] 
(210:1) node (A3) {} circle (2pt) 
(90:1) node (A2) {} circle (2pt) 
(-30:1) node (A4) {} circle (2pt);
\draw (A2) node[above] {$1$} (A3) node[below left] {$2$} (A4) node[below right] 
{$3$};

\draw[->] (A2) to [bend left] (A3);
\draw[->] (A3) to [bend left] (A2);
\draw[->] (A3)--(A4);
\draw[->] (A4)--(A2); 
\end{scope}
\end{tikzpicture}
\caption{Directed graphs which are not quivers}\label{fig_quivers_nonexample}
\end{subfigure}
\caption{Examples and non-examples of quivers}\label{fig_quivers}
\end{figure}
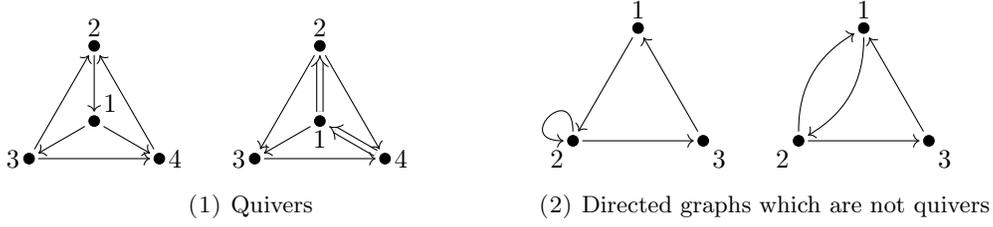

The adjacency matrix of a quiver is always a skew-symmetric matrix. 
To define cluster algebras, we introduce mutations on seeds, exchange matrices, 
and quivers as follows. 

\begin{definition}
The mutation on seeds, exchange matrices, or quivers is defined as follows.
\begin{enumerate}
\item (Mutation on seeds)	For a seed $\seed = (\bfx, \qbasis)$ and an integer 
$k \in [n]$, the \emph{mutation} $\mutation_k(\seed) = (\bfx', \qbasis')$  
is defined as follows:
\begin{equation*}
\begin{split}
x_i' &= \begin{cases}
x_i &\text{ if } i \neq k;\\
\displaystyle 
x_k^{-1}\left( \prod_{b_{j,k} > 0} x_j^{b_{j,k}} + \prod_{b_{j,k} < 
0}x_j^{-b_{j,k}}
\right) & \text{ otherwise}.
\end{cases}\\[1em]
b_{i,j}' &= \begin{cases}
-b_{i,j} & \text{ if } i = k \text{ or } j = k; \\
\displaystyle b_{i,j} + \frac{|b_{i,k}| b_{k,j} + b_{i,k} | b_{k,j}|} {2} & 
\text{ otherwise}.
\end{cases}
\end{split}
\end{equation*}
\item (Mutation on exchange matrices)
We define $\mutation_k(\qbasis) = (b_{i,j}')$, and say that \emph{$\qbasis' 
=(b_{i,j}')$ is the mutation of $\qbasis$ at $k$}.
\item (Mutation on quivers) Let $\quiver$ be a quiver on $[n]$ 
and 
$\qbasis(\quiver)$ its adjacency matrix. For each $k\in[n]$, the mutation
$\mutation_k(\qbasis(\quiver))$ is again the adjacency matrix of a quiver $\quiver'$. 
We define $\mutation_k(\quiver)$ is a quiver satisfying
\begin{equation}\label{eq_mutations_on_quiver_and_exchange_matrix}
\qbasis(\mutation_k(\quiver)) = \mutation_k(\qbasis(\quiver))
\end{equation}
and say that $\mutation_k(\quiver)$ is the \emph{mutation of $\quiver$ at $k$}.
\end{enumerate}
\end{definition}

An immediate check shows that $\mutation_k(\seed)$ is again a seed, and a 
mutation is an involution, that is, its square is the identity. 
For a skew-symmetrizable matrix $\qbasis$ of size $n \times n$, and for $k \in [n]$, we have  
\(
\mutation_k(\qbasis^T) = \mutation_k(\qbasis)^T
\)
by the definition of mutations. Therefore, the mutation preserves the skew-symmetricity.
Because 
of~\eqref{eq_mutations_on_quiver_and_exchange_matrix}, we sometimes denote a 
seed by 
\[
\seed = (\mathbf{x},\quiver) = (\mathbf{x}, \qbasis(\quiver)).
\] 
\begin{example}
Let $\quiver$ be a quiver on the left side of Figure~\ref{fig_quivers_example}. 
The adjacency matrix $\qbasis$ of $\quiver$ and the mutation $\mutation_3(\qbasis)$ 
are given by
\[
\qbasis = \qbasis(\quiver) = \begin{bmatrix}
0 & -1 & 1 & 1 \\
1 & 0 & -1 & -1 \\
-1 & 1 & 0 & 1 \\
-1 & 1 & -1 & 0 
\end{bmatrix},
\qquad 
\mutation_3(\qbasis) = \begin{bmatrix}
0 & 0 & -1 & 2 \\
0 & 0 & 1 & -1 \\
1 & -1 & 0 & -1 \\
-2 & 1 & 1 & 0
\end{bmatrix},
\]
which produces the quiver $\mutation_3(\quiver)$:
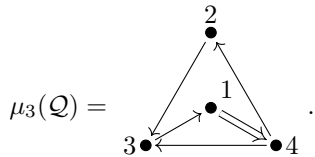
\[
\mutation_3(\quiver)=
\begin{tikzpicture}[baseline=-.5ex]
\draw[fill] (0,0) node (A1) {} circle (2pt) (210:1) node (A3) {} circle (2pt) 
(90:1) node (A2) {} circle (2pt) (-30:1) node (A4) {} circle (2pt);

\draw[->] (A2) node[above] {$2$} -- (A3) node[left] {$3$};
\draw[->] (A3) -- (A1) node[above right] {$1$};
\draw[->] (A4) node[right] {$4$} -- (A3);
\draw[->] (A4) -- (A2);
\draw[-Implies,double distance=2pt] (A1) -- (A4);
\end{tikzpicture}.
\]
One can easily check that $\mutation_4 \mutation_3(\quiver)$ becomes the quiver on the 
right side of Figure~\ref{fig_quivers_example}.
\end{example}

\begin{remark}
Let $k$ be a vertex in a quiver $\quiver$.
The mutation $\mutation_k(\quiver)$ can also be described via a sequence of three steps:
\begin{enumerate}
\item For each directed two-arrow path $i \to k \to j$, add a new arrow $i \to j$.
\item Reverse the direction of all arrows incident to the vertex $k$.
\item Repeatedly remove directed $2$-cycles until unable to do so.
\end{enumerate}
\end{remark}

Let $\mathbb{T}_n$ denote the $n$-regular tree whose edges 
are labeled by $1,\dots,n$. Except for $n = 1$, there are infinitely many 
vertices on the tree $\mathbb{T}_n$. 
A \emph{cluster pattern} (or \emph{seed pattern}) is an assignment
\[
\mathbb{T}_n \to \{\text{seeds in } \field\}, \quad t \mapsto \seed_t = 
(\bfx_t, \qbasis_t)
\]
such that if $\begin{tikzcd} t \arrow[r,dash, "k"] & t' \end{tikzcd}$ in 
$\mathbb{T}_n$, then $\mutation_k(\seed_t) = \seed_{t'}$.

\begin{definition}[{cf. \cite{FZ2_2003}}]
Let $\{ \seed_t = (\bfx_t, \qbasis_t)\}_{t \in \mathbb{T}_n}$ be a cluster 
pattern with $\bfx_t = (x_{1;t},\dots,x_{n;t})$. The \emph{cluster algebra} 
$\cA(\{\seed_t\}_{t \in \mathbb{T}_n})$ is defined to be the 
$\C$-subalgebra of $\field$ generated by all the cluster 
variables $\bigcup_{t \in \mathbb{T}_n} \{x_{1;t},\dots,x_{n;t}\}$.
\end{definition}

If we fix a vertex $t_0 \in \mathbb{T}_n$, then a cluster pattern $\{ \seed_t 
\}_{t \in \mathbb{T}_n}$ is constructed from the seed $\initialseed$. In this 
case, we call $\initialseed$ an \emph{initial seed}. Moreover, up to 
isomorphism on $\field$, a cluster algebra depends only on $\qbasis_{t_0}$ in the initial seed 
$\initialseed$. Because of this reason, we 
simply denote by $\cA(\qbasis_{t_0})$ the cluster algebra given by the cluster 
pattern constructed from the initial seed $\initialseed = (\bfx_{t_0}, 
\qbasis_{t_0})$. Moreover, when $\qbasis_{t_0} = \qbasis(\quiver_{t_0})$ for a 
quiver $\quiver_{t_0}$, we denote by $\cA(\quiver_{t_0})$ the cluster algebra 
$\cA(\qbasis_{t_0})$.

We say that a skew-symmetrizable matrix $\qbasis'$ is \emph{mutation equivalent}
to another \linebreak
skew-symmetrizable matrix $\qbasis$ if they are connected by a 
sequence of mutations
\[
\qbasis' = (\mutation_{j_{\ell}} \cdots \mutation_{j_1})(\qbasis),
\]
and say that $\qbasis$ is \emph{acyclic} if there are no sequences $j_1,j_2,\dots, j_\ell$ with $\ell\ge 3$ such that 
\[
b_{j_1j_2}, b_{j_2j_3},\dots, b_{j_{\ell-1}j_\ell}, b_{j_\ell j_1}>0.
\]

Similarly, we say that a quiver $\quiver'$ is mutation equivalent to another quiver 
$\quiver$ if $\qbasis(\quiver')$ is mutation equivalent to $\qbasis(\quiver)$, and say that $\quiver$ is acyclic if so is $\qbasis(\quiver)$, which is also equivalent to that $\quiver$ has no directed cycles.
 
The \textit{Cartan counterpart} $C(\qbasis) = (c_{i,j})$ of $\qbasis$ is defined by 
\[
c_{i,j} = \begin{cases}
2 & \text{ if } i = j;  \\
-|b_{i,j}| & \text{ if } i \neq j.
\end{cases}
\]

\begin{definition}\label{definition:type}
For a Dynkin type $\dynX$, we define a quiver $\quiver$ or a matrix $\qbasis$ 
\emph{of type~$\dynX$} as follows.
\begin{enumerate}
\item For a quiver $\quiver$, we say that $\quiver$ is \textit{of type~$\dynX$} 
if it is mutation equivalent to an \emph{acyclic} quiver whose underlying graph is isomorphic to the Dynkin 
diagram of type $\dynX$.
\item For a skew-symmetrizable matrix $\qbasis$, we say $\qbasis$ is \textit{of 
type $\dynX$} if it is mutation equivalent to an acyclic skew-symmetrizable matrix whose 
Cartan counterpart $C(\qbasis)$ is isomorphic to the Cartan matrix of type~$\dynX$. 
\end{enumerate}
\end{definition}

Here, we say that two matrices $C_1$ and $C_2$ are \emph{isomorphic} if they are conjugate to each other via a permutation matrix, that is, $C_2 = P^{-1} C_1 P$ for some permutation matrix~$P$.
It is proved in~\cite[Corollary~4]{CalderoKeller06} that if two acyclic skew-symmetrizable matrices are mutation equivalent, then there exists a sequence of mutations from one to other such that intermediate skew-symmetrizable matrices are all acyclic. 
Indeed, if two acyclic skew-symmetrizable matrices are mutation equivalent, then their Cartan counterparts are isomorphic. 
Accordingly, a quiver or a matrix of type $\dynX$ is well-defined. 

On the other hand, all Dynkin diagrams of finite or affine type but $\exdynA_{n-1}$ are acyclic and therefore the acyclicity in Definition~\ref{definition:type}(1) can be omitted for all $\dynX$ but $\exdynA_{n-1}$.
Here is one of the reason why the acyclicity is needed for $\exdynA_{n-1}$ as follows:
Let $\quiver$ be a quiver of $n\ge 3$ vertices whose underlying graph is isomorphic to the $n$-cycle $\exdynA_{n-1}$ and so we have to say that $\quiver$ is of type $\exdynA_{n-1}$ unless the acyclicity.
However, it is known from \cite[Type IV]{Vatne10} that the quiver $\quiver$ is of type $\dynD_n$ when $\quiver$ is a directed $n$-cycle.

Even for acyclic quivers $\quiver$ of $\exdynA_{n-1}$, we have finer separation.
Recall from~\cite[Lemma~6.8]{FST08} that the mutation equivalence class of $\quiver$ does depend on the orientation of the edges in the quiver.
More precisely, let $\quiver$ and $\quiver'$ be of type $\exdynA_{n-1}$. Suppose that in 
$\quiver$, there are $p$ edges of one direction and $q = n - p$ edges of the 
opposite direction. Also, in $\quiver'$, there are $p'$ edges of one direction 
and $q' = n - p'$ edges of the opposite direction. Then two quivers $\quiver$ 
and $\quiver'$ are mutation equivalent if and only if the unordered pairs 
$\{p,q\}$ and $\{p',q'\}$ coincide. 

%
We say that a quiver $\quiver$ is of type~$\exdynA_{p,q}$ if $\quiver$ is mutation equivalent to the quiver of type $\exdynA_{p+q}$ with $p$ edges of one direction and $q$ edges of the opposite direction. We depict some examples for quivers of 
type $\exdynA_{p,q}$ in Figure~\ref{fig_example_Apq}.

\begin{figure}[ht]
\begin{tabular}{cccc}
\begin{tikzpicture}[scale = 0.5]			
\tikzset{every node/.style={scale=0.7}}
\node[vertex] (3) {};
\node[vertex] (1) [below left = 0.6cm and 0.6cm of 3]{};
\node[vertex] (2) [below right = 0.6cm and 0.6cm  of 3] {};

\draw[->] (3)--(1);
\draw[->] (3)--(2);
\draw[->] (1)--(2);
\end{tikzpicture} 
&	\begin{tikzpicture}[scale = 0.5]			
\tikzset{every node/.style={scale=0.7}}
\node[vertex] (1) {};
\node[vertex] (2) [below = of 1] {};
\node[vertex] (3) [right = of 2] {};
\node[vertex] (4) [above = of 3] {};

\draw[->] (4) -- (1);
\draw[->]	(4) -- (3);
\draw[->]	(3) -- (2);
\draw[->]	(2) --(1);
\end{tikzpicture} 
&	\begin{tikzpicture}[scale = 0.5]			
\tikzset{every node/.style={scale=0.7}}
\node[vertex] (1) {};
\node[vertex] (2) [below = of 1] {};
\node[vertex] (3) [right = of 2] {};
\node[vertex] (4) [above = of 3] {};

\draw[->] (4) -- (1);
\draw[->]	(4) -- (3);
\draw[->]	(3) -- (2);
\draw[->]	(1)--(2);
\end{tikzpicture} 
&
\raisebox{4em}{	\begin{tikzpicture}[baseline=-.5ex,scale=0.6]
\tikzset{every node/.style={scale=0.7}}
\foreach \x in {1,...,7, 9,10,11}{
\node[vertex] (\x) at (\x*30:3) {};
}
\node(12) at (12*30:3) {$\vdots$};
\node[rotate=-30] (8) at (8*30:3) {$\cdots$};

\foreach \x [evaluate={\y=int(\x+1);}] in {3,...,6,9}{
\draw[->] (\x)--(\y);
}
\foreach \x [evaluate={\y=int(\x-1);}] in {3,2,11}{
\draw[->] (\x)--(\y);
} 
\curlybrace[]{100}{290}{3.5};
\draw (190:4) node[rotate=0] {$p$};
\curlybrace[]{-50}{80}{3.5};
\draw (15:4) node[rotate=0] {$q$};
\end{tikzpicture}}
\\	
$\exdynA_{1,2}$ & $\exdynA_{1,3}$ & $\exdynA_{2,2}$ & $\exdynA_{p,q}$
\end{tabular}
\caption{Quivers of type~$\exdynA_{p,q}$.}\label{fig_example_Apq}
\end{figure}

In Tables~\ref{Dynkin} and~\ref{table_twisted_affine}, we present lists of standard affine root 
systems and twisted affine root systems, respectively. They are the same as 
presented in Tables Aff 1, Aff2, and Aff 3 of~\cite[Chapter~4]{Kac83}, and we 
denote by $\exdynX = \dynX^{(1)}$.
We notice that the number of vertices of the Dynkin diagram of type $\dynX_{n-1}$ is $n$ while we do not specify the vertex numbering. 

For a Dynkin type $\dynX$, we say that $\dynX$ is \emph{simply-laced} if its Dynkin diagram has only single edges, otherwise, $\dynX$ is \emph{non-simply-laced}.
Recall that the Cartan matrix associated to a Dynkin diagram $\dynX$ can be read directly from the diagram~$\dynX$ as follows:
\begin{center}
\setlength{\tabcolsep}{20pt}
\begin{tabular}{ccccc}
\begin{tikzpicture}[scale =.5, baseline=-.5ex]
\tikzset{every node/.style={scale=0.7}}

\node[vertex, fill=black, label=below:{$i$}] (2) {};
\node[vertex, label=below:{$j$}] (3) [right=of 2] {};

\draw (2)-- (3);

\end{tikzpicture}&
\begin{tikzpicture}[scale =.5, baseline=-.5ex]
\tikzset{every node/.style={scale=0.7}}

\node[vertex, fill=black, label=below:{$i$}] (2) {};
\node[vertex, label=below:{$j$}] (3) [right=of 2] {};

\draw[double line] (2)-- node{\scalebox{1.3}{$>$}} (3);

\end{tikzpicture}&
\begin{tikzpicture}[scale =.5, baseline=-.5ex]
\tikzset{every node/.style={scale=0.7}}

\node[vertex, fill=black, label=below:{$i$}] (2) {};
\node[vertex, label=below:{$j$}] (3) [right=of 2] {};

\draw[triple line] (2)-- node{\scalebox{1.3}{$>$}} (3);
\draw (2)--(3);

\end{tikzpicture}
&
\begin{tikzpicture}[scale=.5, baseline=-.5ex]
\tikzset{every node/.style={scale=0.7}}

\node[vertex, label=below:{$i$}] (1) {};
\node[vertex, label=below:{$j$}] (2) [right = of 1] {};

\draw[double distance = 2.7pt] (1)--(2);
\draw[double distance = 0.9pt] (1)-- node{\scalebox{1.3}{ $>$}} (2);			
\end{tikzpicture} 
&
\begin{tikzpicture}[scale=.5, baseline=-.5ex]
\tikzset{every node/.style={scale=0.7}}
\node[vertex, label=below:{$i$}] (1) {};
\node[vertex, label=below:{$j$}] (2) [right = of 1] {};

\draw[double line] (1)-- node[pos=0.2]{\scalebox{1.3}{ $<$}} 
node[pos=0.7]{\scalebox{1.3}{ $>$}} (2);

\end{tikzpicture}\\
$c_{i,j} = -1$
&$c_{i,j} = -2$
& $c_{i,j} = -3$
& $c_{i,j} = -4$
& $c_{i,j} = -2$ \\
$c_{j,i} = -1$
&$c_{j,i} = -1$
& $c_{j,i} = -1$
& $c_{j,i} = -1$
& $c_{j,i} = -2$
\end{tabular}
\end{center}
For example, the Cartan matrix of the diagram \begin{tikzpicture}[scale =.5, baseline=-.5ex]
\tikzset{every node/.style={scale=0.7}}

\node[vertex, label=below:{$1$}] (1) {};
\node[vertex, label=below:{$2$}] (2) [right = of 1] {};
\node[vertex, label=below:{$3$}] (3) [right=of 2] {};

\draw[triple line] (2)-- node{\scalebox{1.3}{$>$}} (3);
\draw (1)--(2);
\draw (2)--(3);

\end{tikzpicture} of type
$\exdynG_2$ is given by
\begin{equation}\label{eq_Cartan_G2}
\begin{bmatrix}
2 & -1 & 0 \\
-1 & 2 & -3 \\
0 & -1 & 2
\end{bmatrix}.
\end{equation}
Therefore, for each non-simply-laced Dynkin diagram $\dynX$, any exchange matrix $\qbasis$ of type $\dynX$ is \emph{not} skew-symmetric but skew-symmetrizable. Hence it never come from any quiver.

\begin{table}[ht]
\begin{center}
\begin{tabular}{l|l}
\toprule
$\Roots$ & Dynkin diagram \\
\midrule
$\exdynA_1$  &
\begin{tikzpicture}[scale=.5, baseline=-.5ex]
\tikzset{every node/.style={scale=0.7}}
\node[vertex] (1) {};
\node[vertex] (2) [right = of 1] {};

\draw[double line] (1)-- node[pos=0.2]{\scalebox{1.3}{ $<$}} 
node[pos=0.7]{\scalebox{1.3}{ $>$}} (2);

\end{tikzpicture}  \\ 
$\exdynA_{n-1}$ $(n \geq 3)$ &
\begin{tikzpicture}[scale=.5, baseline=-.5ex]
\tikzset{every node/.style={scale=0.7}}
\node[vertex] (1) {};
\node[vertex] (2) [right = of 1] {};
\node[vertex] (3) [right = of 2] {};
\node[vertex] (4) [right =of 3] {};
\node[vertex] (5) [right =of 4] {};			
\node[vertex] (6) [above =of 3] {};			

\draw (1)--(2)--(3)
(4)--(5)
(1)--(6)--(5);
\draw[dotted] (3)--(4);

\end{tikzpicture}  \\ 

$\exdynB_{n-1}$ $(n \geq 4)$ &
\begin{tikzpicture}[scale=.5, baseline=-.5ex]
\tikzset{every node/.style={scale=0.7}}

\node[vertex] (1) {};
\node[vertex] (2) [right = of 1] {};
\node[vertex] (3) [right = of 2] {};
\node[vertex] (4) [right =of 3] {};
\node[vertex] (5) [right =of 4] {};
\node[vertex] (6) [below left = 0.6cm and 0.6cm of 1] {};
\node[vertex] (7) [above left = 0.6cm and 0.6cm of 1] {};			

\draw (1)--(2)
(3)--(4)
(6)--(1)--(7);
\draw [dotted] (2)--(3);
\draw[double line] (4)-- node{\scalebox{1.3}{ $>$}} (5);
\end{tikzpicture}  \\ 		
$\exdynC_{n-1}$ $(n \geq 3)$ & 
\begin{tikzpicture}[scale=.5, baseline=-.5ex]
\tikzset{every node/.style={scale=0.7}}

\node[vertex] (1) {};
\node[vertex] (2) [right = of 1] {};
\node[vertex] (3) [right = of 2] {};
\node[vertex] (4) [right =of 3] {};
\node[vertex] (5) [right =of 4] {};			
\node[vertex] (6) [left =of 1] {};
\draw (1)--(2)
(3)--(4);
\draw [dotted] (2)--(3);
\draw[double line] (4)-- node{\scalebox{1.3}{ $<$}} (5);
\draw[double line] (1)-- node{\scalebox{1.3}{ $>$}} (6);			
\end{tikzpicture}  \\ 	

$\exdynD_{n-1}$ $(n \geq 5)$& 
\begin{tikzpicture}[scale=.5, baseline=-.5ex]
\tikzset{every node/.style={scale=0.7}}

\node[vertex] (1) {};
\node[vertex] (2) [right = of 1] {};
\node[vertex] (3) [right = of 2] {};
\node[vertex] (4) [right =of 3] {};			

\node[vertex] (5) [ below right = 0.6cm and 0.6cm of 4] {};
\node[vertex] (6) [above right=  0.6cm and 0.6cm of 4] {};

\node[vertex] (7) [above left = 0.6cm and 0.6cm of 1] {};
\node[vertex] (8) [below left=0.6cm and 0.6cm of 1] {};

\draw(1)--(2)
(3)--(4)--(5)
(4)--(6)
(7)--(1)--(8);
\draw[dotted] (2)--(3);
\end{tikzpicture}
\\ 
$\exdynE_6$ & 

\begin{tikzpicture}[scale=.5, baseline=-.5ex]
\tikzset{every node/.style={scale=0.7}}

\node[vertex] (1) {};
\node[vertex] (3) [right=of 1] {};
\node[vertex] (4) [right=of 3] {};
\node[vertex] (2) [above=of 4] {};
\node[vertex] (5) [right=of 4] {};
\node[vertex] (6) [right=of 5]{};
\node[vertex] (7) [above=of 2]{};

\draw(1)--(3)--(4)--(5)--(6)
(7)--(2)--(4);
\end{tikzpicture}			
\\
$\exdynE_7$& 
\begin{tikzpicture}[scale=.5, baseline=-.5ex]
\tikzset{every node/.style={scale=0.7}}

\node[vertex] (1) {};
\node[vertex] (8) [left=of 1] {};
\node[vertex] (3) [right=of 1] {};
\node[vertex] (4) [right=of 3] {};
\node[vertex] (2) [above=of 4] {};
\node[vertex] (5) [right=of 4] {};
\node[vertex] (6) [right=of 5]{};
\node[vertex] (7) [right=of 6]{};

\draw (8)--(1)--(3)--(4)--(5)--(6)--(7)
(2)--(4);
\end{tikzpicture}	
\\
$\exdynE_8$ & 
\begin{tikzpicture}[scale=.5, baseline=-.5ex]
\tikzset{every node/.style={scale=0.7}}

\node[vertex] (1) {};
\node[vertex] (3) [right=of 1] {};
\node[vertex] (4) [right=of 3] {};
\node[vertex] (2) [above=of 4] {};
\node[vertex] (5) [right=of 4] {};
\node[vertex] (6) [right=of 5]{};
\node[vertex] (7) [right=of 6]{};
\node[vertex] (8) [right=of 7]{};
\node[vertex] (9) [right=of 8]{};

\draw(1)--(3)--(4)--(5)--(6)--(7)--(8)--(9)
(2)--(4);
\end{tikzpicture}
\\
$\exdynF_4$ 
&
\begin{tikzpicture}[scale = .5, baseline=-.5ex]
\tikzset{every node/.style={scale=0.7}}

\node[vertex] (1) {};
\node[vertex] (2) [right = of 1] {};
\node[vertex] (3) [right = of 2] {};
\node[vertex] (4) [right =of 3] {};
\node[vertex] (5) [right =of 4] {};

\draw (1)--(2)
(2)--(3)
(4)--(5);
\draw[double line] (3)-- node{\scalebox{1.3}{$>$}} (4);
\end{tikzpicture}  \\
$\exdynG_2$
&
\begin{tikzpicture}[scale =.5, baseline=-.5ex]
\tikzset{every node/.style={scale=0.7}}

\node[vertex] (1) {};
\node[vertex] (2) [right = of 1] {};
\node[vertex] (3) [right=of 2] {};

\draw[triple line] (2)-- node{\scalebox{1.3}{$>$}} (3);
\draw (1)--(2);
\draw (2)--(3);

\end{tikzpicture}\\
\bottomrule
\end{tabular}
\end{center}
\caption{Dynkin diagrams of standard affine root systems}\label{Dynkin}
\end{table}

\begin{table}[ht]
\begin{tabular}{l|l}
\toprule
$\Roots$ & Dynkin diagram \\
\midrule
$\dynA_2^{(2)}$ & 
\begin{tikzpicture}[scale=.5, baseline=-.5ex]
\tikzset{every node/.style={scale=0.7}}

\node[vertex] (1) {};
\node[vertex] (2) [right = of 1] {};

\draw[double distance = 2.7pt] (1)--(2);
\draw[double distance = 0.9pt] (1)-- node{\scalebox{1.3}{ $<$}} (2);			
\end{tikzpicture} 
\\
$\dynA_{2(n-1)}^{(2)}$ ($n \ge 3$)
&
\begin{tikzpicture}[scale=.5, baseline=-.5ex]
\tikzset{every node/.style={scale=0.7}}

\node[vertex] (1) {};
\node[vertex] (2) [right = of 1] {};
\node[vertex] (3) [right = of 2] {};
\node[vertex] (4) [right =of 3] {};
\node[vertex] (5) [right =of 4] {};			
\node[vertex] (6) [left =of 1] {};
\draw (1)--(2)
(3)--(4);
\draw [dotted] (2)--(3);
\draw[double line] (4)-- node{\scalebox{1.3}{ $>$}} (5);
\draw[double line] (1)-- node{\scalebox{1.3}{ $>$}} (6);			
\end{tikzpicture}  \\ 
	$\dynA_{2(n-1)-1}^{(2)}$ ($n \ge 4$) & 
\begin{tikzpicture}[scale=.5, baseline=-.5ex]
\tikzset{every node/.style={scale=0.7}}

\node[vertex] (1) {};
\node[vertex] (2) [right = of 1] {};
\node[vertex] (3) [right = of 2] {};
\node[vertex] (4) [right =of 3] {};
\node[vertex] (5) [right =of 4] {};
\node[vertex] (6) [below left = 0.6cm and 0.6cm of 1] {};
\node[vertex] (7) [above left = 0.6cm and 0.6cm of 1] {};			

\draw (1)--(2)
(3)--(4)
(6)--(1)--(7);
\draw [dotted] (2)--(3);
\draw[double line] (4)-- node{\scalebox{1.3}{ $<$}} (5);
\end{tikzpicture}  \\ 	
$\dynD_{n}^{(2)}$ ($n \ge 3$) & 
\begin{tikzpicture}[scale=.5, baseline=-.5ex]
\tikzset{every node/.style={scale=0.7}}

\node[vertex] (1) {};
\node[vertex] (2) [right = of 1] {};
\node[vertex] (3) [right = of 2] {};
\node[vertex] (4) [right =of 3] {};
\node[vertex] (5) [right =of 4] {};			
\node[vertex] (6) [left =of 1] {};
\draw (1)--(2)
(3)--(4);
\draw [dotted] (2)--(3);
\draw[double line] (4)-- node{\scalebox{1.3}{ $>$}} (5);
\draw[double line] (1)-- node{\scalebox{1.3}{ $<$}} (6);			
\end{tikzpicture}  \\ 
$\dynE_6^{(2)}$  &
\begin{tikzpicture}[scale = .5, baseline=-.5ex]
\tikzset{every node/.style={scale=0.7}}

\node[vertex] (1) {};
\node[vertex] (2) [right = of 1] {};
\node[vertex] (3) [right = of 2] {};
\node[vertex] (4) [right =of 3] {};
\node[vertex] (5) [right =of 4] {};

\draw (1)--(2)
(2)--(3)
(4)--(5);
\draw[double line] (3)-- node{\scalebox{1.3}{$<$}} (4);
\end{tikzpicture}  \\
$\dynD_4^{(3)}$  &
\begin{tikzpicture}[scale =.5, baseline=-.5ex]
\tikzset{every node/.style={scale=0.7}}

\node[vertex] (1) {};
\node[vertex] (2) [right = of 1] {};
\node[vertex] (3) [right=of 2] {};

\draw[triple line] (2)-- node{\scalebox{1.3}{$<$}} (3);
\draw (1)--(2);
\draw (2)--(3);

\end{tikzpicture}\\
\bottomrule
\end{tabular}
\caption{Dynkin diagrams of twisted affine root 
systems}\label{table_twisted_affine}
\end{table}

One of the beauties of a skew-symmetrizable matrix of Dynkin type is 
that they are used to classify cluster algebras of finite type or 
finite mutation type.
A cluster algebra is said to be of \emph{finite type} if the cluster pattern has 
only a finite number of seeds.  
A wider class of cluster algebras consists of cluster algebras of \emph{finite 
mutation type}, which have finitely many exchange matrices but are allowed to 
have infinitely many seeds. 

\begin{theorem}[{\cite{FZ2_2003} for finite Dynkin type; 
\cite{FST08,FeliksonShapiroTumarkin12_JEMS} for affine Dynkin type}]
\label{thm_finite_type_classification}
Let $\initialseed = (\mathbf x_{t_0}, \qbasis_{t_0})$ be an initial seed. 
\begin{enumerate}
\item The exchange matrix $\qbasis_{t_0}$  is of finite Dynkin type if and only if there 
are only finitely many seeds in the cluster pattern $\{\seed_t\}_{t \in 
\mathbb{T}_n}$. 
\item If the exchange matrix $\qbasis_{t_0}$ is of affine Dynkin type, then there are only finitely 
many exchange matrices in the cluster pattern $\{\seed_t\}_{t \in 
\mathbb{T}_n}$ while there might be infinitely many seeds. 
\end{enumerate}
\end{theorem}

\begin{assumption}
Throughout this paper, we consider the case where $\qbasis$ is an acyclic 
matrix of \textit{affine} type.
\end{assumption}

\begin{remark}\label{rmk_transpose}
Let $\qbasis$ be a skew-symmetrizable matrix of size $n \times n$.
We have already seen that
\(
\mutation_k(\qbasis^T) = \mutation_k(\qbasis)^T
\)
for $k \in [n]$. 
Accordingly, for a skew-symmetrizable matrix $\qbasis$ of affine type, there is a bijective correspondence between the set of exchange matrices in the cluster pattern $\{ \seed_t \}_{t \in \mathbb{T}_n}$ given by $\initialseed = (\mathbf{x}_{t_0}, \qbasis)$ and that in the cluster pattern $\{\seed_t'\}_{t \in \mathbb{T}_n}$ given by $\initialseed' = (\mathbf{x}_{t_0}', \qbasis^T)$. 
We present pairs $(\dynX, \dynX')$ of affine Dynkin type whose Cartan matrices are transposed to each other.
\begin{gather*} 
(\exdynA_{n-1},\exdynA_{n-1}), \quad
(\exdynB_{n-1}, \dynA_{2(n-1)-1}^{(2)}),\quad
(\exdynC_{n-1}, \dynD_n^{(2)}), \quad
(\exdynD_{n-1}, \exdynD_{n-1}),\\
(\exdynE_n, \exdynE_n) ~~\text{ for }n=6,7,8, \quad
(\exdynF_4, \dynE_6^{(2)}),\quad
(\exdynG_2, \dynD_4^{(3)}), \quad
(\dynA_{2(n-1)}^{(2)}, \dynA_{2(n-1)}^{(2)}).
\end{gather*}
\end{remark}

\begin{remark}
The number of isomorphism classes of exchange matrices in the cluster pattern $\{\seed_t\}_{t \in 
\mathbb{T}_n}$ of type $\exdynA_{p,q}$ is provided in~\cite{BPRS11}:
\[
\begin{cases}
\displaystyle \frac{1}{2}
\sum_{k|p, k|q} \frac{\phi(k)}{p+q} {\binom{2p/k}{p/k}}{\binom{2q/k}{q/k}} & 
\text{ if } p \neq q;\\
\displaystyle \frac{1}{2}\left(
\frac{1}{2} {\binom{2p}{p}} + \sum_{k|p}\frac{\phi(k)}{4p}{\binom{2p/k}{p/k}}^2
\right)
& \text{ if } p = q,
\end{cases}
\]
where $\phi(k)$ is the number of $1 \le d \le k$ coprime to $k$, called 
Euler's totient function. 
Moreover, the following table is obtained 
from~\cite[Theorem~4.15]{MusikerStump2011}.
\begin{center} 
\begin{tabular}{l|ccccc} 
\toprule
Dynkin type $\dynX$ & $\exdynE_6$ & $\exdynE_7$ & $\exdynE_8$ & $\exdynF_4$ or $\dynE_6^{(2)}$ & 
$\exdynG_2$ or $\dynD_4^{(3)}$ \\
\midrule
\# of exchange matrices & $130$ & $1080$ & $7660$ & $60$ & $6$ \\
\bottomrule
\end{tabular}
\end{center}
The number of exchange matrices in the cluster pattern of other affine Dynkin type 
is conjectured in~\cite[Conjecture~4.14]{MusikerStump2011} which is still an 
open problem to the authors' knowledge. 
\end{remark}

For the sake of convenience, we define the restriction of a seed $\seed = (\mathbf x, \quiver)$ as follows.
Suppose that $\quiver$ is a quiver on~$[n]$. 
For each subset $I=\{i_1,\dots, i_k\}\subset[n]$ with $i_1<i_2<\dots<i_k$, the 
restriction $\seed|_I$ is defined as the pair
\[
\seed|_I = (\bfx|_I, \quiver|_I),
\]
where $\bfx|_I=(x_{i_1},\dots, x_{i_k})$ and $\quiver|_I$ is the induced 
subquiver with the set $I$ of vertices. Also, we denote by $\qbasis|_I$ the 
submatrix of $\qbasis$ obtained by considering the columns and rows in $I$ 
simultaneously, that is, $\qbasis|_I = (b_{i,j})_{i,j \in I}$. The following 
is an example of restriction of $\quiver$ of type~$\exdynD_{2n}$,
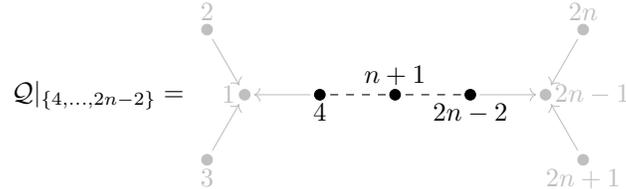
\begin{align*}
\quiver|_{\{4,\dots,2n-2\}} = \begin{tikzpicture}[baseline=-.5ex]
\draw[fill, lightgray]
(0,0) circle(2pt) node (A1) {}
(120:1) circle(2pt) node (A2) {}
(240:1) circle(2pt) node (A3) {}
(4,0) circle (2pt) node (A2n-1) {}
(4,0) ++(60:1) circle(2pt) node (A2n) {}
(4,0) ++(-60:1) circle(2pt) node (A2n+1) {};
\draw[fill]
(1,0) circle(2pt) node (A4) {}
(2,0) circle(2pt) node (An+1) {}
(3,0) circle (2pt) node (A2n-2) {};
\draw[->,lightgray] (A2) node[above] {$2$} -- (A1) node[left] 
{$1$};
\draw[->,lightgray]  (A3) node[below] {$3$} -- (A1);
\draw[->,lightgray] (A4) -- (A1);
\draw[dashed] (An+1) node[above] {$n+1$} -- (A4) node[below] {$4$};
\draw[dashed] (A2n-2) node[below] {$2n-2$} -- (An+1);
\draw[->,lightgray] (A2n-2) -- (A2n-1) node[right] {$2n-1$};
\draw[->,lightgray] (A2n) node[above] {$2n$} -- (A2n-1);
\draw[->,lightgray] (A2n+1) node[below] {$2n+1$} -- (A2n-1);
\end{tikzpicture}
\end{align*}
which is of type~$\dynA_{2n-5}$.

We enclose this section by recalling the following result for later use.
\begin{lemma}[{\cite{ReadingStella20} and also 
see~\cite[Theorem~2.12]{ABL2021c}}]\label{lemma_on_facets}
Let $\initialseed = (\mathbf x_{t_0}, \quiver_{t_0})$ be an initial seed with a quiver $\quiver_{t_0}$ on~$[n]$. 
Suppose that $\quiver_{t_0}$ is an acyclic quiver of affine type. 
Then for any seed $\seed = (\mathbf x, \quiver)$ in the cluster pattern given by 
the initial seed $\initialseed$, there exists
an index $i \in [n]$ such that the quiver $\quiver$ is obtained from 
$\quiver_{t_0}$ by applying mutations on vertices $[n] \setminus \{i\}$. 
\end{lemma}

\section{Invariance and admissibility of quivers}
\label{section_inv_and_admiss_of_quivers}
Under certain condition, one can fold quivers to produce new ones. 
This procedure is used to study quivers of non-simply-laced affine type from 
those of simply-laced affine type. 
In this section, we recall from~\cite{FWZ_chapter45} the invariance and 
admissibility of a finite group action on the quiver. 
We also refer the reader to~\cite{Dupont08}.

Let $\quiver$ be a quiver on $[n]$
and let $G$ be a finite group acting on the set $[n]$.
For $i, i' \in [n]$, the notation $i \sim i'$ will mean that $i$ and $i'$ lie in the same $G$-orbit. To study folding of exchange matrices or quivers, we prepare some terminologies.

For each $g\in G$, let $\quiver'=g\cdot\quiver$ be the quiver whose adjacency matrix $\qbasis(\quiver')$ is given by
\[
\qbasis(\quiver')=(b'_{ij}),\qquad b_{ij}=b_{g(i),g(j)}.
\] 

\begin{definition}[{cf.~\cite[\S4.4]{FWZ_chapter45} and~\cite[\S 
3]{Dupont08}}] \label{definition:admissible quiver}
Let $\quiver$ be a quiver on $[n]$ and 
let $G$ be a finite group acting on the set $[n]$.
\begin{enumerate} 
\item A quiver $\quiver$ is \emph{$G$-invariant} if $g\cdot\quiver=\quiver$ for each $g\in G$.
\item A $G$-invariant quiver $\quiver$ is \emph{$G$-admissible} if for any $i\sim i'$,
\begin{enumerate}[ref = (\alph*)]
\item $b_{i,i'} = 0$; 
\label{def_admissible_3}
\item $b_{i,j} b_{i',j} \geq 0$ for any $j$,
\label{def_admissible_4}
\end{enumerate}
where $\qbasis(\quiver)=(b_{i,j})$.
\end{enumerate}        
\begin{remark}
The $G$-admissibility can also be defined for an exchange matrix and a seed, and furthermore those with frozen vertices.
Note that this definition is simplified due to Assumption~\ref{assumption:all variables are mutable}.
\end{remark} 

\end{definition}
For a $G$-admissible quiver $\quiver$, we define the matrix $\qbasis^G = 
\qbasis(\quiver)^G = (b_{I,J}^G)$ whose rows and columns are 
labeled by the $G$-orbits by
\[
b_{I,J}^G = \sum_{i \in I} b_{i,j}
\]
where $I$ and $J$ are $G$-orbits and $j$ is an arbitrary index in $J$. We then say $\qbasis^G$ is obtained 
from $\qbasis$ (or from the quiver $\quiver$) by \textit{folding} with respect 
to the given $G$-action.

\begin{example}\label{example_A22_to_A1}
Let $\quiver$ be a quiver of type $\exdynA_{2,2}$ given as follows:
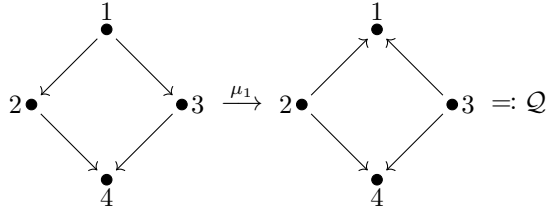
\[
\begin{tikzpicture}[baseline=-.5ex]
\foreach \x in {1,...,4}{
\draw[fill] (90*\x:1) node (A\x) {} circle (2pt);
}

\draw (A1) node[above] {$1$} 
(A2) node[left] {$2$}
(A3) node[below] {$4$}
(A4) node[right] {$3$};

\draw[->] (A1) to (A2);
\draw[->] (A2) to (A3);
\draw[->] (A4) to (A3);
\draw[->] (A1) to (A4);
\end{tikzpicture}
\stackrel{\mu_1}{\longrightarrow}
\begin{tikzpicture}[baseline=-.5ex]
\foreach \x in {1,...,4}{
\draw[fill] (90*\x:1) node (A\x) {} circle (2pt);
}

\draw (A1) node[above] {$1$} 
(A2) node[left] {$2$}
(A3) node[below] {$4$}
(A4) node[right] {$3$};

\draw[<-] (A1) to (A2);
\draw[->] (A2) to (A3);
\draw[->] (A4) to (A3);
\draw[<-] (A1) to (A4);
\end{tikzpicture} \eqqcolon \quiver 
\]
The adjacency matrix $\qbasis(\quiver)$ of $\quiver$ is  
\[
\qbasis(\quiver) = 
\begin{bmatrix}
0 & -1 & -1 & 0 \\
1 & 0 & 0 & 1 \\
1 & 0 & 0 & 1 \\
0 & -1 & -1 & 0
\end{bmatrix}.
\]
Suppose that the finite group $G = \Z/2\Z$ acts on the set $[4] = \{1,\dots,4\}$ such that the generator sends $1 \mapsto 4 \mapsto 1$ and $2 \mapsto 3 \mapsto 2$. Then, the quiver $\quiver$ is $G$-admissible, and by setting $I_1 = \{1,4\}$ and $I_2 = \{2,3\}$, we obtain
\[
\begin{split}
b_{I_1,I_2}^G &= \sum_{i \in I_1} b_{i,2} = b_{1,2} + b_{4,2} = -2, \\
b_{I_2,I_1}^G &= \sum_{i \in I_2} b_{i,1} = b_{2,1} + b_{3,1} = 2. 
\end{split}
\]
This provides
\[
\qbasis^G = \begin{bmatrix}
0 & -2 \\ 2 & 0
\end{bmatrix},
\]
whose Cartan counterpart is the Cartan matrix of type $\exdynA_1$, and moreover, it is the adjacency matrix of the quiver
\(
\begin{tikzpicture}[baseline=-.5ex]
\draw[fill] (0,0) node (A1) {} circle (2pt);
\draw[fill] (1,0) node (A2) {} circle (2pt);
\draw (A1) node[above] {$1$} ;
\draw (A2) node[above] {$2$} ;
\draw[-Implies,double distance=2pt] (A2)--(A1);
\end{tikzpicture}
\).
\end{example}
In Example~\ref{example_A22_to_A1}, the folded matrix $\qbasis^G$ is again skew-symmetric. However, as we will see in the example below, the folded matrix is not skew-symmetric but skew-symmetrizable in general. 
\begin{example}\label{example_D4_to_G2}
Let $\quiver$ be a quiver of type~$\exdynE_6$ whose adjacency matrix 
$\qbasis(\quiver)$ is given as follows.
\[
\qbasis(\quiver) = \begin{bmatrix}
0 & 1 & 0 & 1 & 0 & 1 & 0\\
-1 & 0 & -1 & 0 & 0 & 0 & 0 \\
0 & 1 & 0 & 0 & 0 & 0 & 0 \\
-1 & 0 & 0 & 0 & -1 & 0 & 0 \\
0 & 0 & 0  & 1 & 0 & 0 &  0\\
-1 & 0 & 0 & 0 & 0 & 0 & -1 \\
0 & 0 & 0 & 0 & 0 & 1 & 0 
\end{bmatrix}.
\]
Suppose that the finite group $G = \Z / 3 \Z$ acts on the set $[7] = \{1,\dots,7\}$ as depicted in 
\ref{fig_E6_Z3} of Appendix~\ref{appendix_actions_on_Dynkin_diagrams}. 
One may check that the quiver $\quiver$ is 
$G$-admissible. By setting $I_1 = \{1\}$, $I_2 = \{2, 4, 6\}$, and $I_3 = \{3, 
5, 7\}$, we obtain
\[
\begin{split}
b_{I_1,I_2}^G &= \sum_{i \in I_1} b_{i,2} = b_{1,2} = 1, \\
b_{I_1, I_3}^G &= \sum_{i \in I_1} b_{i,3} = b_{1,3} = 0, \\
b_{I_2, I_3}^G &= \sum_{i \in I_2} b_{i,3} = b_{2,3} + b_{4,3} + b_{6,3} = -1, 
\\
b_{I_2, I_1}^G &= \sum_{i \in I_2} b_{i,1} = b_{2,1} + b_{4,1} + b_{6,1} = -3, 
\\
b_{I_3, I_1}^G &= \sum_{i \in I_3} b_{i,1} = b_{3,1} + b_{5,1} + b_{7,1} = 0, \\
b_{I_3,I_2}^G &= \sum_{i \in I_3} b_{i, 2} = b_{3,2} + b_{5,2} + b_{7,2} = 1.
\end{split}
\]
Accordingly, we obtain the matrix 
\[
\qbasis^G = \begin{bmatrix} 
0 & 1 & 0 \\ -3 & 0 & -1 \\ 0 & 1 & 0 
\end{bmatrix}
\] 
whose Cartan counterpart is isomorphic to the Cartan matrix of type~$\exdynG_2$ (cf.~\eqref{eq_Cartan_G2}).
\end{example}

For a $G$-admissible quiver $\quiver$ and a $G$-orbit $I$, we consider 
a composition of mutations given by
\[
\mutation_I = \prod_{i \in I} \mutation_i
\]
which is well-defined because of the definition of admissible quivers. We call 
$\mutation_I$ an \textit{orbit mutation}. If $\mutation_I(\quiver)$ is again 
$G$-admissible, then we have that
\begin{equation*}
(\mutation_I(\qbasis))^G = \mutation_I(\qbasis^G).
\end{equation*}
We notice that the quiver $\mutation_I(\quiver)$ may \textit{not} be 
$G$-admissible in general (cf. Example~\ref{example_non_admissible}). 
Therefore, we present the following definition.
\begin{definition}
Let $G$ be a group acting on the vertex set of a quiver $\quiver$. We say that 
$\quiver$ is \emph{globally foldable} with respect to $G$ if $\quiver$ is 
$G$-admissible, and moreover, for any sequence of $G$-orbits 
$I_1,\dots,I_\ell$, the quiver $(\mutation_{I_\ell} \dots 
\mutation_{I_1})(\quiver)$ is $G$-admissible.
\end{definition}

\begin{example}\label{example_non_admissible}
Let $\quiver$ be a quiver with $6$ vertices given as follows.
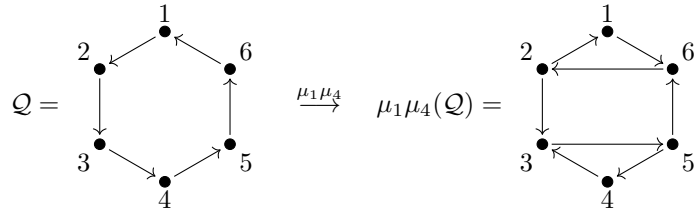
\begin{center}
$\quiver = $
\begin{tikzpicture}[baseline=-.5ex]
\foreach \x in {1,...,6}{
\draw[fill] (30+60*\x:1) node (A\x) {} circle (2pt);
}

\draw (A1) node[above] {$1$} 
(A2) node[above left] {$2$}
(A3) node[below left] {$3$}
(A4) node[below] {$4$}
(A5) node[below right] {$5$}
(A6) node[above right] {$6$};

\draw[->] (A1) to (A2);
\draw[->] (A2) to (A3);
\draw[->] (A3) to (A4);
\draw[->] (A4) to (A5);
\draw[->] (A5) to (A6);
\draw[->] (A6) to (A1);

\end{tikzpicture}
$\quad \stackrel{\mutation_1\mutation_4}{\longrightarrow} \quad \mutation_1\mutation_4(\quiver) = $
\begin{tikzpicture}[baseline=-.5ex]
\foreach \x in {1,...,6}{
\draw[fill] (30+60*\x:1) node (A\x) {} circle (2pt);
}

\draw (A1) node[above] {$1$} 
(A2) node[above left] {$2$}
(A3) node[below left] {$3$}
(A4) node[below] {$4$}
(A5) node[below right] {$5$}
(A6) node[above right] {$6$};

\draw[->] (A1) to (A6);
\draw[->] (A2) to (A1);
\draw[->] (A2) to (A3);
\draw[->] (A3) to (A5);
\draw[->] (A4) to (A3);
\draw[->] (A5) to (A4);
\draw[->] (A5) to (A6);
\draw[->] (A6) to (A2);
\end{tikzpicture}
\end{center}
Consider an action of $G = \Z/2\Z$ such that $1 \sim 4$, $2 \sim 5$, and $3 
\sim 6$. One can easily see that the quiver $\quiver$ is $G$-invariant, and 
moreover, is $G$-admissible. However, by considering mutations on vertices $1$ 
and $4$, we obtain the quiver $\mutation_1\mutation_4(\quiver)$ which is $G$-invariant but 
not $G$-admissible. This is because for indices $2 \sim 5$ and $3$, we have
\[
b_{2,3} b_{5,3} = 1 \cdot (-1) = -1 \not\ge 0
\]
which violates the condition \ref{def_admissible_4} in 
Definition~\ref{definition:admissible quiver}(2). 
Accordingly, the quiver $\quiver$ is $G$-admissible but not globally foldable 
with respect to $G$. 
\end{example}

As we saw in Example~\ref{example_non_admissible}, a $G$-invariant quiver may 
not be $G$-admissible. The following theorem says that the converse holds when 
we consider the foldings presented in Table~\ref{table:all possible foldings}. 
In Appendix~\ref{appendix_actions_on_Dynkin_diagrams}, we describe the finite 
group action explicitly for each triple $(\dynX,G,\dynY)$. 
The proof of the following theorem will be given in Section~\ref{section:proof 
of admissibility}.
\begin{theorem}\label{thm_invaraint_implies_admissible}
Let $(\dynX, G, \dynY)$ be a triple given by a column of Table~\ref{table:all 
possible foldings}.
Let $\quiver$ be a quiver of type~$\dynX$. 
If $\quiver$ is $G$-invariant, then it is $G$-admissible. Indeed,  
$G$-invariance and $G$-admissibility are equivalent.
\end{theorem}
As a direct consequence of the above theorem, we obtain the following.
\begin{corollary}
Let $(\dynX, G, \dynY)$ be a triple given by a column of Table~\ref{table:all 
possible foldings}.
Then any quiver of type~$\dynX$ is globally foldable with respect to $G$. 
\end{corollary}

\begin{remark}
Because the cluster algebras of affine type is finite mutation type (see 
Theorem~\ref{thm_finite_type_classification}(2)), using the computer program 
SageMath~\cite{sagemath}, one can get the same result as 
Theorem~\ref{thm_invaraint_implies_admissible} for quivers of type $\exdynE$ or type 
$\exdynA_{n,n},\exdynD_n$ for a given $n$.
More precisely, the command \verb!mutation_class! produces all quivers which 
are mutation equivalent to a given one. 
For more details, we refer the reader to~\cite[\S 4.4]{MusikerStump2011}. 
Using this command, one may provide an \emph{experimental proof} while we provide a \emph{combinatorial proof} by observing the 
combinatorics of quivers. 
\end{remark}

\section{Type-by-type arguments for admissibility}
\label{section:proof of admissibility}
In this section, we will prove Theorem~\ref{thm_invaraint_implies_admissible}. 
We say that a quiver $\quiver$ is of \emph{finite mutation type} (or is \emph{mutation-finite}) if there is only finitely many quivers mutation equivalent to $\quiver$. Otherwise, we say that $\quiver$ is of \emph{infinite mutation type} (or is \emph{mutation-infinite}).
As we have already seen in Theorem~\ref{thm_finite_type_classification}, a quiver of affine or finite Dynkin type is mutation-finite. 

Before providing a proof, we study some mutation-infinite quivers. 
We first recall the following lemma:
\begin{lemma}[{\cite[Lemma~6.4]{FeliksonShapiroTumarkin12_JEMS}}]
\label{lemma:induced subquiver}
Let $\quiver$ be a quiver on $[n]$ and let
$\quiver_0=\quiver|_I$ be the restriction onto a subset $I\subset[n]$. Then, 
for any quiver $\quiver_1$ mutation equivalent to $\quiver_0$, there exists 
$\quiver'$ mutation equivalent to $\quiver$ such that $\quiver'|_I = \quiver_1$.
\end{lemma}

We say that a quiver $\quiver$ is \emph{reduced to} $\quiver'$ if $\quiver'$ is 
obtained by applying a sequence of mutations on $\quiver$ and restrictions, 
denoted by 
\[
\quiver\reducedto \quiver'.
\]
Then the lemma below is the direct consequence of the definition and 
Lemma~\ref{lemma:induced subquiver}.
\begin{lemma}\label{lemma:mutation_finite_cannot_have_mutation_infinite_subquiver}
Let $\quiver$ and $\quiver'$ be quivers with $\quiver\reducedto \quiver'$.
Then, we obtain the following.
\begin{enumerate}
\item If $\quiver$ is mutation-finite, then so is $\quiver'$.
\item If $\quiver'$ is mutation-infinite, then so is $\quiver$.
\end{enumerate}
\end{lemma}

Obviously, any quiver with two vertices is mutation-finite and here are known mutation-infinite quivers of three vertices 
(cf.~\cite[Section~2.3]{ABBS08}):  
\begin{align*}
&\begin{tikzpicture}[baseline=-.5ex]
\draw[fill] (-1,0) node (A1) {} circle (2pt) (0,0) node (A2) {} circle (2pt) 
(1,0) node (A3) {} circle (2pt);
\draw (A1) -- (A2) node[midway, above] {$a$};
\draw (A2) -- (A3) node[midway, above] {$b$};
\end{tikzpicture}\quad ab\ge2,&
&\begin{tikzpicture}[baseline=-.5ex]
\draw[fill] (210:1) node (A1) {} circle (2pt) (90:1) node (A2) {} circle (2pt) 
(-30:1) node (A3) {} circle (2pt);
\draw[->] (A1) -- (A2) node[midway, left] {$a$};
\draw[->] (A2) -- (A3) node[midway, right] {$b$};
\draw[->] (A3) -- (A1) node[midway, below] {$c$};
\end{tikzpicture}\quad ab>2c \ge 2,&
&\begin{tikzpicture}[baseline=-.5ex]
\draw[fill] (210:1) node (A1) {} circle (2pt) (90:1) node (A2) {} circle (2pt) 
(-30:1) node (A3) {} circle (2pt);
\draw[->] (A1) -- (A2) node[midway, left] {$a$};
\draw[->] (A2) -- (A3) node[midway, right] {$b$};
\draw[->] (A1) -- (A3) node[midway, below] {$c$};
\end{tikzpicture}\quad abc\ge 2
\end{align*}
where $a,b,c$ represent the number of arrows in the shown direction. We call 
these quivers the \emph{linear quiver} of type~$(a,b)$, the \emph{cyclic 
triangle} of type~$(a,b,c)$ and the \emph{acyclic triangle} of type~$(a,b,c)$, 
respectively.

\begin{lemma}\label{lemma:complete 4-graph}
Let $\quiver$ be a quiver on $[4]=\{1,2,3,4\}$ such that every pair of distinct vertices is 
connected. 
Then $\quiver$ is mutation-infinite unless $\quiver$ is the following quiver:
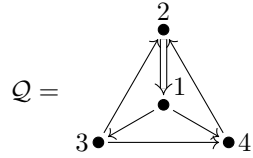
\[
\quiver=
\begin{tikzpicture}[baseline=.5ex]
\draw[fill] (0,0) node (A1) {} circle (2pt) (210:1) node (A3) {} circle (2pt) 
(90:1) node (A2) {} circle (2pt) (-30:1) node (A4) {} circle (2pt);
\draw[-Implies,double distance=2pt] (A2) node[above] {$2$} -- (A1) node[above 
right] {$1$};
\draw[->] (A3) node[left] {$3$} -- (A2);
\draw[->] (A4) node[right] {$4$} -- (A2);
\draw[->] (A1) -- (A3);
\draw[->] (A3) -- (A4);
\draw[->] (A1) -- (A4);
\end{tikzpicture}
\]
\end{lemma}

\begin{proof} 
According to the number of \emph{sources}---vertices only with outward edges--- and \emph{sinks}---vertices only with inward edges--- in $\quiver$, there are only four 
quivers up to isomorphisms as depicted in Figure~\ref{figure:complete 4-graph}. 
Then it is easy to check that a quiver with one sink is mutation equivalent to 
a quiver with one source.

\begin{figure}[ht]
\begin{subfigure}[t]{0.28\textwidth}
\centering
\begin{tikzpicture}[baseline=.5ex]
\draw[fill] (0,0) node (A1) {} circle (2pt) (210:1) node (A3) {} circle (2pt) 
(90:1) node (A2) {} circle (2pt) (-30:1) node (A4) {} circle (2pt);
\draw[->] (A2) node[above] {$2$} -- (A1) node[above right] {$1$};
\draw[->] (A3) node[left] {$3$} -- (A2);
\draw[->] (A4) node[right] {$4$} -- (A2);
\draw[->] (A1) -- (A3);
\draw[->] (A3) -- (A4);
\draw[->] (A1) -- (A4);
\end{tikzpicture}
\caption{A quiver without source or sink}
\end{subfigure}
\begin{subfigure}[t]{0.4\textwidth}
\centering
\begin{tikzpicture}[baseline=.5ex]
\draw[fill] (0,0) node (A1) {} circle (2pt) (210:1) node (A3) {} circle (2pt) 
(90:1) node (A2) {} circle (2pt) (-30:1) node (A4) {} circle (2pt);
\draw[->] (A1) node[above right] {$1$} -- (A2) node[above] {$2$};
\draw[->] (A3) node[left] {$3$} -- (A1);
\draw[->] (A4) node[right] {$4$} -- (A1);
\draw[->] (A2) -- (A3);
\draw[->] (A4) -- (A3);
\draw[->] (A4) -- (A2);
\end{tikzpicture}$\!\!\stackrel{\mutation_4}{\longleftrightarrow}\!\!$
\begin{tikzpicture}[baseline=.5ex]
\draw[fill] (0,0) node (A1) {} circle (2pt) (210:1) node (A3) {} circle (2pt) 
(90:1) node (A2) {} circle (2pt) (-30:1) node (A4) {} circle (2pt);
\draw[->] (A1) node[above right] {$1$} -- (A2) node[above] {$2$};
\draw[->] (A3) node[left] {$3$} -- (A1);
\draw[->] (A1) -- (A4) node[right] {$4$};
\draw[->] (A2) -- (A3);
\draw[->] (A3) -- (A4);
\draw[->] (A2) -- (A4);
\end{tikzpicture}
\caption{Quivers with source or sink}
\end{subfigure}
\begin{subfigure}[t]{0.28\textwidth}
\centering
\begin{tikzpicture}[baseline=.5ex]
\draw[fill] (0,0) node (A1) {} circle (2pt) (210:1) node (A3) {} circle (2pt) 
(90:1) node (A2) {} circle (2pt) (-30:1) node (A4) {} circle (2pt);
\draw[->] (A1) node[above right] {$1$} -- (A2) node[above] {$2$};
\draw[->] (A3) node[left] {$3$} -- (A1);
\draw[->] (A4) node[right] {$4$} -- (A1);
\draw[->] (A3) -- (A2);
\draw[->] (A4) -- (A3);
\draw[->] (A4) -- (A2);
\end{tikzpicture}
\caption{A quiver with both source and sink}
\end{subfigure}
\caption{Quivers of the complete 4-graph}
\label{figure:complete 4-graph}
\end{figure}
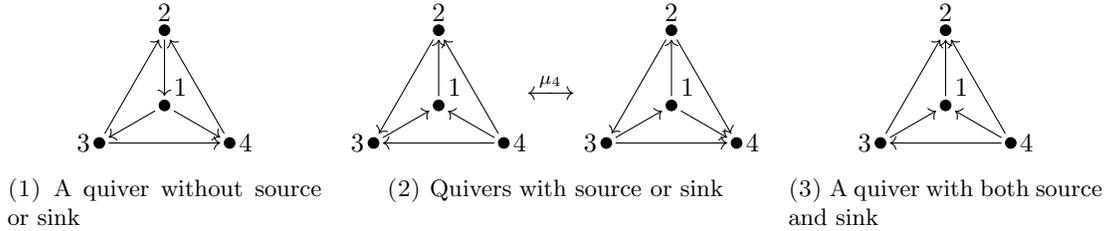

Suppose that $\quiver$ is mutation-finite. Then it can not be reduced to the 
acyclic triangle of type~$(a,b,c)$ with $abc\ge 2$ and therefore any acyclic 
triangle is of type~$(1,1,1)$.

Since every edge of the quiver $\quiver$ with a sink or a source is a part of 
an 
acyclic triangle, all edges of $\quiver$ is \emph{simple}, that is, it consists 
of a 
single arrow.  
For a quiver without sink and source, the only edge from $2$ to $1$ may have 
multiple edges.

Since a cyclic triangle of type~$(1,1,b)$ with $b\ge 3$ is mutation-infinite, 
$|b_{1,2}|\le 2$. Hence either,
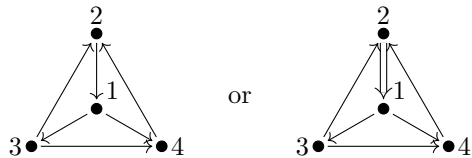
\[
\begin{tikzpicture}[baseline=.5ex]
\draw[fill] (0,0) node (A1) {} circle (2pt) (210:1) node (A3) {} circle (2pt) 
(90:1) node (A2) {} circle (2pt) (-30:1) node (A4) {} circle (2pt);
\draw[->] (A2) node[above] {$2$} -- (A1) node[above right] {$1$};
\draw[->] (A3) node[left] {$3$} -- (A2);
\draw[->] (A4) node[right] {$4$} -- (A2);
\draw[->] (A1) -- (A3);
\draw[->] (A3) -- (A4);
\draw[->] (A1) -- (A4);
\end{tikzpicture}\quad\text{ or }\quad
\begin{tikzpicture}[baseline=.5ex]
\draw[fill] (0,0) node (A1) {} circle (2pt) (210:1) node (A3) {} circle (2pt) 
(90:1) node (A2) {} circle (2pt) (-30:1) node (A4) {} circle (2pt);
\draw[-Implies,double distance=2pt] (A2) node[above] {$2$} -- (A1) node[above 
right] {$1$};
\draw[->] (A3) node[left] {$3$} -- (A2);
\draw[->] (A4) node[right] {$4$} -- (A2);
\draw[->] (A1) -- (A3);
\draw[->] (A3) -- (A4);
\draw[->] (A1) -- (A4);
\end{tikzpicture}
\]
but we exclude the latter by assumption.

Then the following can be checked directly:
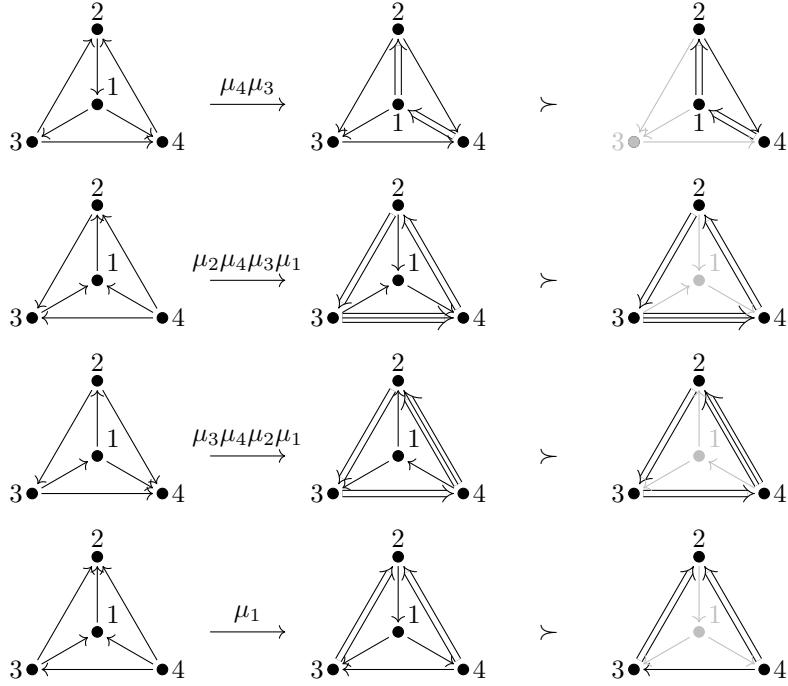
\begin{align*}
&\begin{tikzpicture}[baseline=.5ex]
\begin{scope}
\draw[fill] (0,0) node (A1) {} circle (2pt) (210:1) node (A3) {} circle (2pt) 
(90:1) node (A2) {} circle (2pt) (-30:1) node (A4) {} circle (2pt);
\draw[->] (A2) node[above] {$2$} -- (A1) node[above right] {$1$};
\draw[->] (A3) node[left] {$3$} -- (A2);
\draw[->] (A4) node[right] {$4$} -- (A2);
\draw[->] (A1) -- (A3);
\draw[->] (A3) -- (A4);
\draw[->] (A1) -- (A4);
\end{scope}
\begin{scope}[xshift=4cm]
\draw[fill] (0,0) node (A1) {} circle (2pt) (210:1) node (A3) {} circle (2pt) 
(90:1) node (A2) {} circle (2pt) (-30:1) node (A4) {} circle (2pt);
\draw[->] (A1) -- (A3) node[left] {$3$};
\draw[->] (A2) -- (A3);
\draw[-Implies,double distance=2pt] (A4) -- (A1) node[below] {$1$};
\draw[->] (A3) -- (A4) node[right] {$4$};
\draw[<-] (A4) -- (A2) node[above] {$2$};
\draw[-Implies,double distance=2pt] (A1)--(A2);
\end{scope}
\begin{scope}[xshift=8cm]
\draw[fill] (0,0) node (A1) {} circle (2pt) (210:1) node (A3) {} circle (2pt) 
(90:1) 
node (A2) {} circle (2pt) (-30:1) node (A4) {} circle (2pt);
\draw[-Implies,double distance=2pt] (A4) -- (A1) node[below] {$1$};
\draw[<-] (A4) node[right] {$4$} -- (A2) node[above] {$2$};
\draw[-Implies,double distance=2pt] (A1)--(A2);
\draw[->, lightgray] (A2) -- (A3);
\draw[->, lightgray] (A1) -- (A3) node[left] {$3$};
\draw[->, lightgray] (A3) -- (A4) ;
\draw[fill, lightgray] (A3) {} circle (2pt);
\end{scope}
\draw[->](1.5,0)--(2.5,0) node[midway,above] {$\mutation_4\mutation_3$};
\draw (6,0) node {$\succ$};
\end{tikzpicture}\\
&\begin{tikzpicture}[baseline=.5ex]
\begin{scope}
\draw[fill] (0,0) node (A1) {} circle (2pt) (210:1) node (A3) {} circle (2pt) 
(90:1) node (A2) {} circle (2pt) (-30:1) node (A4) {} circle (2pt);
\draw[->] (A1) node[above right] {$1$} -- (A2) node[above] {$2$};
\draw[->] (A3) node[left] {$3$} -- (A1);
\draw[->] (A4) node[right] {$4$} -- (A1);
\draw[->] (A2) -- (A3);
\draw[->] (A4) -- (A3);
\draw[->] (A4) -- (A2);
\end{scope}
\begin{scope}[xshift=4cm]
\draw[fill] (0,0) node (A1) {} circle (2pt) (210:1) circle (2pt) node (A3) {} 
(90:1) node (A2) {} circle (2pt) (-30:1) node (A4) {} circle (2pt);
\draw[->] (A2) -- (A1);
\draw[->] (A3) node[left] {$3$} -- (A1) node[above right] {$1$};
\draw[->] (A1) -- (A4);
\draw[-Implies,double distance=2pt] (A2) -- (A3);
\triplearrow{arrows={-Implies}}{(A3) -- (A4)};
\draw[-Implies,double distance=2pt] (A4) node[right] {$4$} -- (A2) node[above] 
{$2$};
\end{scope}
\begin{scope}[xshift=8cm]
\draw[fill] (210:1) node (A3) {} circle (2pt) (90:1) node (A2) {} circle (2pt) 
(-30:1) node (A4) {} circle (2pt);
\draw[-Implies,double distance=2pt] (A2) -- (A3) node[left] {$3$};
\triplearrow{arrows={-Implies}}{(A3) -- (A4)};
\draw[-Implies,double distance=2pt] (A4) node[right] {$4$} -- (A2) node[above] 
{$2$};
\draw[fill,lightgray] (0,0) node (A1) {} circle (2pt);
\draw[->,lightgray] (A2) -- (A1) node[above right] {$1$};
\draw[->,lightgray] (A3) -- (A1);
\draw[->,lightgray] (A1) -- (A4);
\end{scope}
\draw[->](1.5,0)--(2.5,0) node[midway,above] {$\mutation_2\mutation_4\mutation_3\mutation_1$};
\draw (6,0) node {$\succ$};
\end{tikzpicture}\\
&\begin{tikzpicture}[baseline=.5ex]
\begin{scope}
\draw[fill] (0,0) node (A1) {} circle (2pt) (210:1) node (A3) {} circle (2pt) 
(90:1) node (A2) {} circle (2pt) (-30:1) node (A4) {} circle (2pt);
\draw[->] (A1) node[above right] {$1$} -- (A2) node[above] {$2$};
\draw[->] (A3) node[left] {$3$} -- (A1);
\draw[->] (A1) -- (A4) node[right] {$4$};
\draw[->] (A2) -- (A3);
\draw[->] (A3) -- (A4);
\draw[->] (A2) -- (A4);
\end{scope}
\begin{scope}[xshift=4cm]
\draw[fill] (0,0) node (A1) {} circle (2pt) (210:1) node (A3) {} circle (2pt) 
(90:1) node (A2) {} circle (2pt) (-30:1) node (A4) {} circle (2pt);
\draw[->] (A1) -- (A2) node[above] {$2$};
\draw[->] (A1) -- (A3) node[left] {$3$};
\draw[->] (A4) -- (A1) node[above right] {$1$};
\draw[-Implies,double distance=2pt] (A2) -- (A3);
\draw[-Implies,double distance=2pt] (A3) -- (A4) node[right] {$4$};
\triplearrow{arrows={-Implies}}{(A4) -- (A2)};
\end{scope}
\begin{scope}[xshift=8cm]
\draw[fill] (210:1) node (A3) {} circle (2pt) (90:1) node (A2) {} circle (2pt) 
(-30:1) node (A4) {} circle (2pt);
\draw[-Implies,double distance=2pt] (A2) node[above] {$2$} -- (A3) node[left] 
{$3$};
\draw[-Implies,double distance=2pt] (A3) -- (A4) node[right] {$4$};
\triplearrow{arrows={-Implies}}{(A4) -- (A2)};
\draw[fill,lightgray] (0,0) node (A1) {} circle (2pt);
\draw[->,lightgray] (A1) node[above right] {$1$} -- (A2);
\draw[->,lightgray] (A1) -- (A3);
\draw[->,lightgray] (A4) -- (A1);
\end{scope}
\draw[->](1.5,0)--(2.5,0) node[midway,above] {$\mutation_3\mutation_4\mutation_2\mutation_1$};
\draw (6,0) node {$\succ$};
\end{tikzpicture}\\
&\begin{tikzpicture}[baseline=.5ex]
\begin{scope}
\draw[fill] (0,0) node (A1) {} circle (2pt) (210:1) node (A3) {} circle (2pt) 
(90:1) node (A2) {} circle (2pt) (-30:1) node (A4) {} circle (2pt);
\draw[->] (A1) node[above right] {$1$} -- (A2) node[above] {$2$};
\draw[->] (A3) node[left] {$3$} -- (A1);
\draw[->] (A4) node[right] {$4$} -- (A1);
\draw[->] (A3) -- (A2);
\draw[->] (A4) -- (A3);
\draw[->] (A4) -- (A2);
\end{scope}
\begin{scope}[xshift=4cm]
\draw[fill] (0,0) node (A1) {} circle (2pt) (210:1) node (A3) {} circle (2pt) 
(90:1) node (A2) {} circle (2pt) (-30:1) node (A4) {} circle (2pt);
\draw[->] (A2) -- (A1) node[above right] {$1$};
\draw[->] (A1) -- (A3) node[left] {$3$};
\draw[->] (A1) -- (A4) node[right] {$4$};
\draw[-Implies,double distance=2pt] (A3) -- (A2) node[above] {$2$};
\draw[->] (A4) -- (A3);
\draw[-Implies,double distance=2pt] (A4) -- (A2);
\end{scope}
\begin{scope}[xshift=8cm]
\draw[fill] (210:1) node (A3) {} circle (2pt) (90:1) node (A2) {} circle (2pt) 
(-30:1) node (A4) {} circle (2pt);
\draw[-Implies,double distance=2pt] (A3) node[left] {$3$} -- (A2) node[above] 
{$2$};
\draw[->] (A4) -- (A3);
\draw[-Implies,double distance=2pt] (A4) node[right] {$4$} -- (A2);
\draw[fill,lightgray] (0,0) node (A1) {} circle (2pt);
\draw[->,lightgray] (A2) -- (A1) node[above right] {$1$};
\draw[->,lightgray] (A1) -- (A3);
\draw[->,lightgray] (A1) -- (A4);
\end{scope}
\draw[->](1.5,0)--(2.5,0) node[midway,above] {$\mutation_1$};
\draw (6,0) node {$\succ$};
\end{tikzpicture}
\end{align*}
\end{proof}

\begin{remark}
Indeed, the quiver that we exclude in the assumption of the lemma above is 
\emph{block-decomposable} in the sense of \cite{FST08, 
FeliksonShapiroTumarkin12_JEMS} and so mutation-finite.
\end{remark}

\begin{corollary}\label{corollary:complete 4-graph}
Let $\quiver$ be a quiver on $[4]$ such that every pair of distinct vertices is 
connected.  
If $\quiver$ contains a cyclic triangle of type different from~$(1,1,2)$, then 
it is 
mutation-infinite.
\end{corollary}

\begin{corollary}\label{corollary:seed restriction}
Let $\quiver$ be a quiver on $[n]$ of standard affine type. Then $\quiver$ can not be 
reduced to any of the following: 
\begin{enumerate}
\item mutation-infinite quivers;
\item $\quiver(2\exdynA_1)$,
\end{enumerate}
where
\[
\quiver(2\exdynA_1) = \quiver(\exdynA_1)\coprod \quiver(\exdynA_1)=
\begin{tikzpicture}[baseline=-0.5ex]
\draw[fill] (-0.5,0.5) node (A2) {} circle (2pt) (0.5,-0.5) node (A4) {} circle 
(2pt) (-0.5,-0.5) node (A6) {} circle (2pt) (0.5,0.5) node (A7) {} circle (2pt);
\draw[-Implies,double distance=2pt] (A2) -- (A6);
\draw[-Implies,double distance=2pt] (A4) -- (A7);
\end{tikzpicture}
\]
\end{corollary} 
\begin{proof}
\noindent (1) This is obvious since any standard affine type quiver is 
mutation-finite by Theorem~\ref{thm_finite_type_classification}(2) and by 
Lemma~\ref{lemma:mutation_finite_cannot_have_mutation_infinite_subquiver}. 

\noindent (2) Assume on the contrary that $\quiver\succ \quiver(2\exdynA_1)$. Since 
$2\exdynA_1$ is not of standard affine type, it must be a proper restriction of 
a quiver $\quiver'$, which is mutation equivalent to $\quiver$. We denote by $I 
\subsetneq [n]$ the subset satisfying $\quiver'|_{I} = \quiver(2\exdynA_1)$.
Consider a cluster pattern $\{\seed_t\}_{t \in \mathbb{T}_n}$ with the initial 
seed $\Sigma' = (\bfx', \quiver')$. Since the quiver $\quiver'$ is of standard 
affine type, by Lemma~\ref{lemma_on_facets}, the number of seeds in the 
cluster pattern $\{\seed_t\}_{t \in \mathbb{T}_n}$  obtained from the initial seed 
$\Sigma'$ by applying mutations on vertices $[n] \setminus \{i\}$  is finite 
for any $i \in [n]$. This implies that the number of seeds obtained from the 
restriction $\Sigma'|_{I}$ by mutations on vertices $I \setminus \{i\}$ is 
finite as well. This is impossible because of 
Theorem~\ref{thm_finite_type_classification}(1) and therefore $\quiver$ can 
not be reduced to $\quiver(2\exdynA_1)$ as claimed. 
\end{proof}

By using these lemmas, we will prove 
Theorem~\ref{thm_invaraint_implies_admissible}. We note that for any triple 
$(\dynX, G, \dynY)$ in Table~\ref{table:all possible foldings}, we have
\[
\dynX = \exdynA_{n,n}, \quad\exdynD_n,\quad \text{ or }\quad\exdynE.
\]
We will prove the theorem using type-by-type arguments. 

Before presenting the proof, we explain our strategy.
In order to show the $G$-admissibility 
of a $G$-invariant quiver $\quiver$, it is enough to show the conditions \ref{def_admissible_3} and 
\ref{def_admissible_4} in 
Definition~\ref{definition:admissible quiver}(2) since $\quiver$ is $G$-invariant by the assumption.
Moreover, if any element $g \in G$ is of order $2$, the 
condition~\ref{def_admissible_3} holds:
\begin{lemma}\label{lemma:condition_3_holds_of_order_2}
Let $(\dynX, G, \dynY)$ be a triple given by a column of 
Table~\ref{table:all possible foldings}. 
Let $\quiver$ be a quiver on $[n]$ of type~$\dynX$. 
Suppose that every element $g \in G$ is of order $2$. 
Then, for each $i\in[n]$ and $g\in G$, we obtain
\[
b_{i,g(i)}=0.
\]
\end{lemma}
\begin{proof}
Since every $g\in G$ is of order 2, we have
\[
b_{i,g(i)} = b_{g(i), g^2(i)} = b_{g(i), i} = -b_{i,g(i)}.
\]
Therefore $b_{i,g(i)} = 0$.
\end{proof}

\subsection{\texorpdfstring{Admissibility of quivers of 
type~$\exdynA_{n,n}$}{Admissibility of quivers of affine An,n}}

\subsubsection{\texorpdfstring{$(\dynX, G, \dynY) = (\exdynA_{2,2}, \Z/2\Z, \exdynA_1)$}{(X, G, Y)=(affine A22, Z/2Z, affine A1)}
}
Let $\quiver$ be the quiver of type~$\exdynA_{2,2}$:
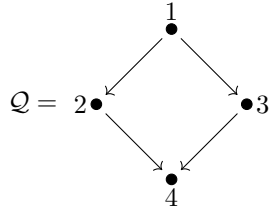
\[
\quiver = 
\begin{tikzpicture}[baseline=-.5ex]
\foreach \x in {1,...,4}{
\draw[fill] (90*\x:1) node (A\x) {} circle (2pt);
}

\draw (A1) node[above] {$1$} 
(A2) node[left] {$2$}
(A3) node[below] {$4$}
(A4) node[right] {$3$};

\draw[->] (A1) to (A2);
\draw[->] (A2) to (A3);
\draw[->] (A4) to (A3);
\draw[->] (A1) to (A4);


\end{tikzpicture}
\]
There are four quivers mutation equivalent to $\quiver$ up to isomorphisms. We present all of them:

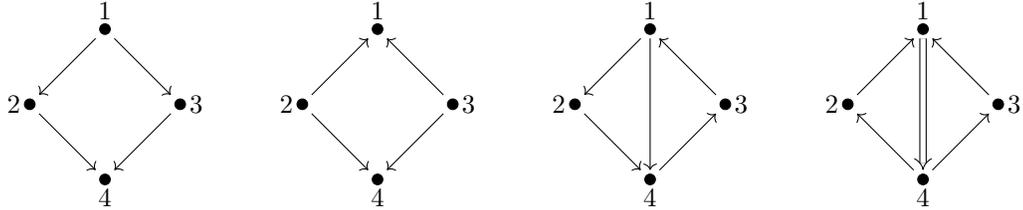
\hfill
\begin{tikzpicture}[baseline=-.5ex]
\foreach \x in {1,...,4}{
\draw[fill] (90*\x:1) node (A\x) {} circle (2pt);
}

\draw (A1) node[above] {1} 
(A2) node[left] {$2$}
(A3) node[below] {$4$}
(A4) node[right] {$3$};

\draw[->] (A1) to (A2);
\draw[->] (A2) to (A3);
\draw[->] (A4) to (A3);
\draw[->] (A1) to (A4);
\end{tikzpicture}
\hfill
\begin{tikzpicture}[baseline=-.5ex]
\foreach \x in {1,...,4}{
\draw[fill] (90*\x:1) node (A\x) {} circle (2pt);
}

\draw (A1) node[above] {$1$} 
(A2) node[left] {$2$}
(A3) node[below] {$4$}
(A4) node[right] {$3$};

\draw[<-] (A1) to (A2);
\draw[->] (A2) to (A3);
\draw[->] (A4) to (A3);
\draw[<-] (A1) to (A4);
\end{tikzpicture}
\hfill
\begin{tikzpicture}[baseline=-.5ex]
\foreach \x in {1,...,4}{
\draw[fill] (90*\x:1) node (A\x) {} circle (2pt);
}

\draw (A1) node[above] {$1$} 
(A2) node[left] {$2$}
(A3) node[below] {$4$}
(A4) node[right] {$3$};

\draw[->] (A1) to (A2);
\draw[->] (A2) to (A3);
\draw[->] (A3) to (A4);
\draw[->] (A4) to (A1);
\draw[->] (A1) to (A3);
\end{tikzpicture}
\hfill 
\begin{tikzpicture}[baseline=-.5ex]
\foreach \x in {1,...,4}{
\draw[fill] (90*\x:1) node (A\x) {} circle (2pt);
}

\draw (A1) node[above] {$1$} 
(A2) node[left] {$2$}
(A3) node[below] {$4$}
(A4) node[right] {$3$};

\draw[<-] (A1) to (A2);
\draw[<-] (A2) to (A3);
\draw[<-] (A4) to (A3);
\draw[<-] (A1) to (A4);
\draw[-Implies,double distance=2pt] (A1) -- (A3);
\end{tikzpicture} 
\hfill

For the $\Z/2\Z$-action defined by 
\[
\tau(1) = 4,\quad \tau(2) =3,\quad \tau(3) = 2,\quad \tau(4) =1,
\]
the second quiver is the only $\Z/2\Z$-invariant quiver, which is $\Z/2\Z$-admissible. See~\ref{fig_A22_A1} in Appendix~\ref{appendix_actions_on_Dynkin_diagrams}. This proves the following lemma:
\begin{lemma}\label{lemma:A22}
Let $\quiver$ be the $\Z/2\Z$-invariant quiver of type $\exdynA_{2,2}$. Then, $\quiver$ is $\Z/2\Z$-admissible.
\end{lemma}

\subsubsection{\texorpdfstring{$(\dynX, G, \dynY) = (\exdynA_{n,n}, \Z/2\Z, \dynD_{n+1}^{(2)})$}{(X, G, Y)=(affine Ann, Z/2Z, affine Dn+1 2)}}
\label{section_Ann}

Let $\quiver$ be a $\Z/2\Z$-invariant quiver on $[2n]$ of type~$\exdynA_{n,n}$ and 
$\qbasis=\qbasis(\quiver)$.  See~\ref{fig_Ann_Dn+1_2} in Appendix~\ref{appendix_actions_on_Dynkin_diagrams} for the $\Z/2\Z$-action.
To show the admissibility, it is enough to check the 
condition~\ref{def_admissible_4} because of 
Lemma~\ref{lemma:condition_3_holds_of_order_2}. 
\begin{lemma}\label{lemma:no bigons in affine A}
For any $i,j\in[2n]$, we obtain
\[
b_{i,j}b_{i,\tau(j)}\ge 0.
\]
\end{lemma}
\begin{proof}
If one of $i$ or $j$ is $1$ or $2n$, then we are done since $\tau(i)=i$ or 
$\tau(j)=j$.
Assume on the contrary that $b_{i,j}b_{i,\tau(j)}<0$ for some $i,j\neq 1, 2n$. 
We may assume that $1<i,j\le n$ and
\[
b_{\tau(i),\tau(j)}=b_{i,j}<0<b_{i,\tau(j)}=b_{\tau(i),j}.
\]
Then we have a directed cycle $\quiver|_{\{i,j,\tau(i),\tau(j)\}}$
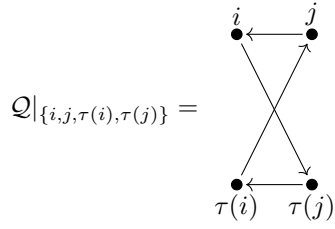
\[
\quiver|_{\{i,j,\tau(i),\tau(j)\}}=\begin{tikzpicture}[baseline=-.5ex]
\draw[fill]
(0,1) circle (2pt) node (Ai) {}
(1,1) circle (2pt) node (Aj) {}
(0,-1) circle (2pt) node (Ataui) {}
(1,-1) circle (2pt) node (Atauj) {};
\draw[->] (Ai) -- (Atauj) node[below] {$\tau(j)$};
\draw[->] (Atauj) -- (Ataui) node[below] {$\tau(i)$};
\draw[->] (Ataui) -- (Aj) node[above] {$j$};
\draw[->] (Aj) -- (Ai) node[above] {$i$};
\end{tikzpicture}
\]

On the other hand, $\quiver$ is obtained from the initial quiver 
$\quiver(\exdynA_{n,n})$ via a sequence of mutations 
\[
\quiver = 
(\mutation^{\exdynA_{n,n}}_{j_L}\cdots\mutation^{\exdynA_{n,n}}_{j_1})(\quiver(\exdynA_{n,n})),
\]
where the sequence $j_1,\dots, j_L$ misses at least index $\ell\in[2n]$ by Lemma~\ref{lemma_on_facets}. 

If $\ell\not\in\{i,j,\tau(i),\tau(j)\}$, then the restriction 
\[
\quiver|_{[2n]\setminus\{\ell\}}=
(\mutation^{\dynA_{2n-1}}_{j_L}\cdots\mutation^{\dynA_{2n-1}}_{j_1})
(\quiver(\dynA_{2n-1})))
\]
is of type~$\dynA_{2n-1}$. However, this yields a contradiction as any quiver mutation equivalent to $\dynA_{2n-1}$ never have a directed cycle of length $4$, as shown in ~\cite[Proposition~2.4]{BuanVatne08}.

If $\ell\in\{i,j,\tau(i),\tau(j)\}$, then the sequence $j_1,\dots, j_L$ misses 
$\tau(\ell)$ as well by Lemma~\ref{lemma_on_facets}.
Hence the directed cycle $\quiver|_{\{i,j,\tau(i),\tau(j)\}}$ should be 
contained in one of two restrictions $\quiver|_{R}$ and $\quiver|_{S}$ of 
type~$\dynA_{2\ell-1}$ and $\dynA_{2(n-\ell)-3}$, where
\begin{align*}
R&=\{1,\dots,\ell, n+1,\dots,n+\ell-1\},&
S&=\{\ell,\dots,n, n+\ell-1,\dots,2n\}.
\end{align*}

On the other hand, two restrictions of $\quiver(\exdynA_{n,n})$ are of 
type~$\dynA_{2\ell-1}$ and $\dynA_{2(n-\ell)+3}$, respectively.
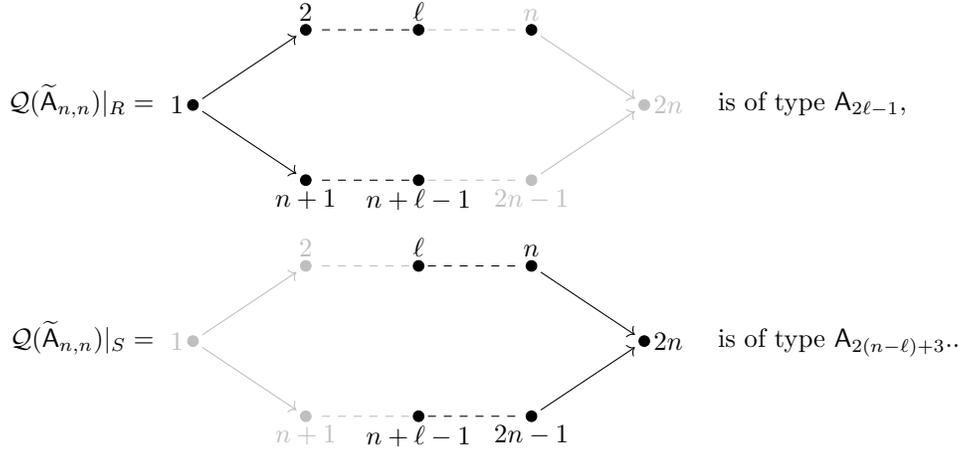
\begin{align*}
\quiver(\exdynA_{n,n})|_{R}&=\begin{tikzpicture}[baseline=-.5ex]
\draw[fill] 
(-3,0) circle(2pt) node (A1) {}
(-1.5,1) circle(2pt) node (A2) {}
(0,1) circle(2pt) node (A3) {}
(1.5,1) circle(2pt) node (An) {}
(-1.5,-1) circle(2pt) node (An+1) {}
(0,-1) circle(2pt) node (An+2) {};
\draw[fill,lightgray] 
(1.5,-1) circle(2pt) node (A2n-1) {}
(3,0) circle(2pt) node (A2n) {};
\draw[->] (A1) node[left] {$1$} -- (A2);
\draw[->] (A1) -- (An+1);
\draw[dashed] (A3) node[above] {$\ell$} -- (A2) node[above] {$2$};
\draw[dashed] (An+2) node[below] {$n+\ell-1$} -- (An+1) node[below] {$n+1$};
\draw[dashed, lightgray] (A3) -- (An) node[above] {$n$};
\draw[dashed, lightgray] (An+2) -- (A2n-1) node[below] {$2n-1$};
\draw[->, lightgray] (An) -- (A2n) node[right] {$2n$};
\draw[->, lightgray] (A2n-1) -- (A2n);
\end{tikzpicture} \quad \text{is of type~$\dynA_{2\ell-1}$,}\\
\quiver(\exdynA_{n,n})|_{S}&=\begin{tikzpicture}[baseline=-.5ex]
\draw[fill, lightgray] 
(-3,0) circle(2pt) node (A1) {}
(-1.5,1) circle(2pt) node (A2) {}
(-1.5,-1) circle(2pt) node (An+1) {};
\draw[fill] 
(0,1) circle(2pt) node (A3) {}
(1.5,1) circle(2pt) node (An) {}
(0,-1) circle(2pt) node (An+2) {}
(1.5,-1) circle(2pt) node (A2n-1) {}
(3,0) circle(2pt) node (A2n) {};
\draw[->, lightgray] (A1) node[left] {$1$} -- (A2);
\draw[->, lightgray] (A1) -- (An+1);
\draw[dashed, lightgray] (A3) -- (A2) node[above] {$2$};
\draw[dashed, lightgray] (An+2) -- (An+1) node[below] {$n+1$};
\draw[dashed] (A3) node[above] {$\ell$} -- (An) node[above] {$n$};
\draw[dashed] (An+2) node[below] {$n+\ell-1$} -- (A2n-1) node[below] {$2n-1$};
\draw[->] (An) -- (A2n) node[right] {$2n$};
\draw[->] (A2n-1) -- (A2n);
\end{tikzpicture} \quad \text{is of type~$\dynA_{2(n-\ell)+3}$.}.
\end{align*}      
Hence, there is a sequence of mutations either from $\quiver(\exdynA_{n,n})|_R$ 
to $\quiver|_R$ or $\quiver(\exdynA_{n,n})|_S$ to $\quiver|_S$, which yields a 
contradiction again.
\end{proof}

\begin{proof}[Proof of Theorem~\ref{thm_invaraint_implies_admissible} for 
$\dynX = \exdynA_{n,n}$]
For a $\Z/2\Z$-invariant quiver $\quiver$ of type~$\exdynA_{n,n}$, the 
conditions~\ref{def_admissible_3} and \ref{def_admissible_4} in 
Definition~\ref{definition:admissible quiver}(2) 
follows from Lemmas~\ref{lemma:condition_3_holds_of_order_2} and \ref{lemma:no 
bigons in affine A}. Combining this with Lemma~\ref{lemma:A22}, the quiver $\quiver$ is $\Z/2\Z$-admissible as claimed.
\end{proof}

\subsection{\texorpdfstring{Admissibility of quivers of 
type~$\exdynE$}{Admissibility of quivers of affine E}}

\subsubsection{\texorpdfstring{$(\dynX, G, \dynY) = (\exdynE_6, \Z/3\Z, \exdynG_2)$}{(X, G, Y)=(affine E6, Z/3Z, affine G2)}}
\label{section_E6_G2}

Let $\quiver$ be a $\Z/3\Z$-invariant quiver on $[7]$ of type $\exdynE_6$
and $\qbasis=\qbasis(\quiver)$. See~\ref{fig_E6_Z3} in 
Appendix~\ref{appendix_actions_on_Dynkin_diagrams}.

\begin{lemma}\label{lemma:no monogons in affine E6 with Z/3Z-action}
We have
\[
b_{2,4}=b_{4,6}=b_{6,2}=0 \quad\text{ and }\quad
b_{3,5}=b_{5,7}=b_{7,3}=0.
\]
\end{lemma}
\begin{proof}
Suppose that the assertion does not hold. Then by relabelling vertices if 
necessary, we may assume that $b_{2,4}=b_{4,6}=b_{6,2}=n\ge 1$ and so 
$\quiver|_{\{2,4,6\}}$ is a cyclic triangle of type~$(n,n,n)$.
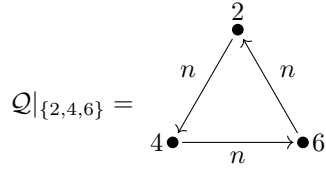
\[
\quiver|_{\{2,4,6\}}=
\begin{tikzpicture}[baseline=-.5ex]
\draw[fill] (210:1) node (A3) {} circle (2pt) (90:1) node (A2) {} circle (2pt) 
(-30:1) node (A4) {} circle (2pt);
\draw (A2) node[above] {$2$} (A3) node[left] {$4$} (A4) node[right] {$6$};
\draw[->] (A2) -- (A3) node[midway, above left] {$n$};
\draw[->] (A3) -- (A4) node[midway, below] {$n$};
\draw[->] (A4) -- (A2) node[midway, above right] {$n$};
\end{tikzpicture}
\]

If $b_{1,2}=b_{1,4}=b_{1,6}\neq 0$, then every pair of distinct vertices 
in $\quiver|_{\{1,2,4,6\}}$ is connected and 
it contains a cyclic triangle of type~$(n,n,n)$. This restriction is 
mutation-infinite by Corollary~\ref{corollary:complete 4-graph} and so is 
$\quiver$. Since $\exdynE_6$ is mutation-finite, this is a contradiction. 
Hence, 
we obtain 
\[
b_{1,2}=b_{1,4}=b_{1,6}=0.
\]
Similarly, we have
\[
b_{3,5} = b_{5,7} = b_{7,3} = 0,
\]
otherwise, the restriction $\quiver|_{\{1,3,5,7\}}$ is mutation-infinite. 
However, since $\quiver$ is connected, we have
\begin{align*}
b_{1,3}=b_{1,5}=b_{1,7}=m\neq 0
\end{align*}
and at least one of 
\begin{align*}
x&=b_{3,2}=b_{5,4}=b_{7,6},&
y&=b_{3,4}=b_{5,6}=b_{7,2},&
z&=b_{3,6}=b_{5,2}=b_{7,4}
\end{align*}
is non-zero.

If none of $x,y$ and $z$ is zero, then the restriction $\quiver|_{\{2,3,4,6\}}$ 
is 
again mutation-infinite by Corollary~\ref{corollary:complete 4-graph} since 
every pair of distinct vertices in $\quiver|_{\{2,3,4,6\}}$ is connected  
and it contains a cyclic triangle of type~$(n,n,n)$. Hence, we may assume that 
either
\[
\begin{array}{ccc}
\text{(i) } x\neq 0, y=z=0& \text{ or }&
\text{(ii) } x\neq 0, y\neq 0, z=0\\
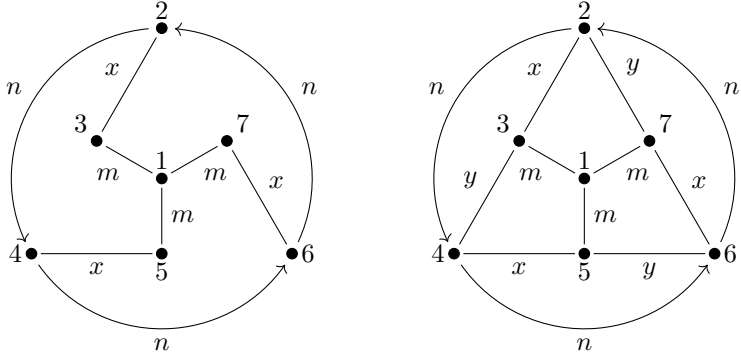
\begin{tikzpicture}[baseline=-.5ex]
\draw[fill] (0,0) node (A1) {} circle (2pt) (210:2) node (A3) {} circle (2pt) 
(90:2) node (A2) {} circle (2pt) (-30:2) node (A4) {} circle (2pt) (150:1) node 
(A5) {} circle (2pt) (270:1) node (A6) {} circle (2pt) (30:1) node (A7) {} 
circle (2pt);
\draw (A1) node[above] {$1$} (A5) node[above left] {$3$} (A6) node[below] {$5$} 
(A7) node[above right] {$7$};
\draw (A2) node[above] {$2$} (A3) node[left] {$4$} (A4) node[right] {$6$};
\draw[->] (95:2) arc (95:205:2) node[midway, above left] {$n$};
\draw[->] (215:2) arc (215:325:2) node[midway, below] {$n$};
\draw[->] (335:2) arc (335:445:2) node[midway, above right] {$n$};
\draw (A1) -- (A5) node[midway,below left] {$m$};
\draw (A1) -- (A6) node[midway,right] {$m$};
\draw (A1) -- (A7) node[midway,below right] {$m$};
\draw (A5) -- (A2) node[midway, above left] {$x$};
\draw (A6) -- (A3) node[midway, below] {$x$};
\draw (A7) -- (A4) node[midway, above right] {$x$};
\end{tikzpicture}& 
&\begin{tikzpicture}[baseline=-.5ex]
\draw[fill] (0,0) node (A1) {} circle (2pt) (210:2) node (A3) {} circle (2pt) 
(90:2) node (A2) {} circle (2pt) (-30:2) node (A4) {} circle (2pt) (150:1) node 
(A5) {} circle (2pt) (270:1) node (A6) {} circle (2pt) (30:1) node (A7) {} 
circle (2pt);
\draw (A1) node[above] {$1$} (A5) node[above left] {$3$} (A6) node[below] {$5$} 
(A7) node[above right] {$7$};
\draw (A2) node[above] {$2$} (A3) node[left] {$4$} (A4) node[right] {$6$};
\draw[->] (95:2) arc (95:205:2) node[midway, above left] {$n$};
\draw[->] (215:2) arc (215:325:2) node[midway, below] {$n$};
\draw[->] (335:2) arc (335:445:2) node[midway, above right] {$n$};
\draw (A1) -- (A5) node[midway,below left] {$m$};
\draw (A1) -- (A6) node[midway,right] {$m$};
\draw (A1) -- (A7) node[midway,below right] {$m$};
\draw (A5) -- (A2) node[midway, above left] {$x$};
\draw (A6) -- (A3) node[midway, below] {$x$};
\draw (A7) -- (A4) node[midway, above right] {$x$};
\draw (A5) -- (A3) node[midway, above left] {$y$};
\draw (A6) -- (A4) node[midway, below] {$y$};
\draw (A7) -- (A2) node[midway, above right] {$y$};
\end{tikzpicture}
\end{array}
\]

\noindent  \textbf{Case (i)} Suppose that $x\neq0$ but $y=z=0$. By 
taking a mutation $\mutation_1$ 
if necessary, and mutations $\mutation_3,\mutation_5,\mutation_7$, we obtain a new quiver 
$\quiver'$ with $mx$ edges from $1$ to $2$. That is, for 
$\qbasis'=\qbasis(\quiver')=(b'_{i,j})$
\[
b'_{1,2} = b'_{1,4}=b'_{1,6} = mx\neq 0.
\]
Hence $\quiver\succ\quiver'|_{\{1,2,4,6\}}$, which is mutation-infinite and 
therefore this is a contradiction.
\smallskip

\noindent \textbf{Case (ii)} Suppose that $x\neq0, y\neq 0, z=0$.
Similarly, if $xy>0$ (equivalently, the signs of $x$ and $y$ are same), then by 
taking a mutation 
$\mutation_1$ if necessary, and mutations $\mutation_3,\mutation_5,\mutation_7$, we obtain $\quiver'$ 
so that
\[
b'_{1,2}=b'_{1,4}=b'_{1,6} = m(x+y)\neq 0.
\]
Even if $xy<0$, unless $\quiver|_{\{2,3,4\}}$ is cyclic and $n+xy=0$, the same 
argument still holds. 

Suppose that $y<0<x$, $n+xy=0$, $m>0$, and quivers $\quiver|_{\{2,3,4\}}$, 
$\quiver|_{\{4,5,6\}}$, and $\quiver|_{\{2,6,7\}}$ are cyclic. Then the quiver 
$\quiver'$ obtained by applying the mutations $\mutation_3,\mutation_5$, and $\mutation_7$ on 
$\quiver$ looks like
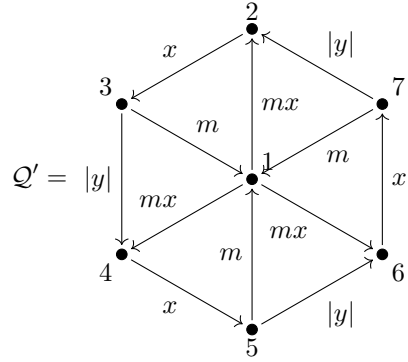
\[
\quiver'=\begin{tikzpicture}[baseline=-.5ex]
\draw[fill] (0,0) node (A1) {} circle (2pt) (210:2) node (A3) {} circle (2pt) 
(90:2) node (A2) {} circle (2pt) (-30:2) node (A4) {} circle (2pt) (150:2) node 
(A5) {} circle (2pt) (270:2) node (A6) {} circle (2pt) (30:2) node (A7) {} 
circle (2pt);
\draw (A1) node[above right] {$1$} (A5) node[above left] {$3$} (A6) node[below] 
{$5$} (A7) node[above right] {$7$};
\draw (A2) node[above] {$2$} (A3) node[below left] {$4$} (A4) node[below right] 
{$6$};
\draw[->] (A5) -- (A1) node[midway,above right] {$m$};
\draw[->] (A6) -- (A1) node[midway,left] {$m$};
\draw[->] (A7) -- (A1) node[midway,below right] {$m$};
\draw[->] (A2) -- (A5) node[midway, above left] {$x$};
\draw[->] (A3) -- (A6) node[midway, below left] {$x$};
\draw[->] (A4) -- (A7) node[midway, right] {$x$};
\draw[->] (A5) -- (A3) node[midway, left] {$|y|$};
\draw[->] (A6) -- (A4) node[midway, below right] {$|y|$};
\draw[->] (A7) -- (A2) node[midway, above right] {$|y|$};
\draw[->] (A1) -- (A2) node[midway, right] {$mx$};
\draw[->] (A1) -- (A3) node[midway, above left] {$mx$};
\draw[->] (A1) -- (A4) node[midway, below left] {$mx$};
\end{tikzpicture}
\]
Then the restriction $\quiver'|_{\{1,3,4\}}$ is an acyclic triangle of 
type~$(m, mx, |y|)$, and so it is mutation-infinite unless $m \cdot mx \cdot 
|y|=1$. 
That is, $x=1, m=1$, and $|y|=1$. Then due to the classification of 
mutation-finite quivers, this is not mutation-finite. Indeed, after the 
mutations $\mutation_1, \mutation_2, \mutation_4$ and $\mutation_6$, the quiver will be reduced to a 
cyclic triangle of type~$(1,2,2)$.
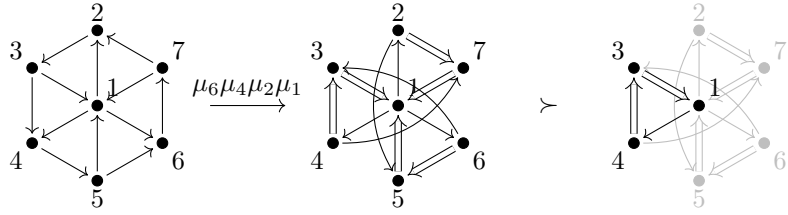
\[
\begin{tikzpicture}[baseline=-.5ex]
\draw[fill] (0,0) node (A1) {} circle (2pt) (210:1) node (A3) {} circle (2pt) 
(90:1) node (A2) {} circle (2pt) (-30:1) node (A4) {} circle (2pt) (150:1) node 
(A5) {} circle (2pt) (270:1) node (A6) {} circle (2pt) (30:1) node (A7) {} 
circle (2pt);
\draw (A1) node[above right] {$1$} (A5) node[above left] {$3$} (A6) node[below] 
{$5$} (A7) node[above right] {$7$};
\draw (A2) node[above] {$2$} (A3) node[below left] {$4$} (A4) node[below right] 
{$6$};
\draw[->] (A5) -- (A1);
\draw[->] (A6) -- (A1);
\draw[->] (A7) -- (A1);
\draw[->] (A2) -- (A5);
\draw[->] (A3) -- (A6);
\draw[->] (A4) -- (A7);
\draw[->] (A5) -- (A3);
\draw[->] (A6) -- (A4);
\draw[->] (A7) -- (A2);
\draw[->] (A1) -- (A2);
\draw[->] (A1) -- (A3);
\draw[->] (A1) -- (A4);
\begin{scope}[xshift=4cm]
\draw[fill] (0,0) node (A1) {} circle (2pt) (210:1) node (A3) {} circle (2pt) 
(90:1) node (A2) {} circle (2pt) (-30:1) node (A4) {} circle (2pt) (150:1) node 
(A5) {} circle (2pt) (270:1) node (A6) {} circle (2pt) (30:1) node (A7) {} 
circle (2pt);
\draw (A1) node[above right] {$1$} (A5) node[above left] {$3$} (A6) node[below] 
{$5$} (A7) node[above right] {$7$};
\draw (A2) node[above] {$2$} (A3) node[below left] {$4$} (A4) node[below right] 
{$6$};
\draw[->] (A1) -- (A2);
\draw[->] (A1) -- (A3);
\draw[->] (A1) -- (A4);
\draw[-Implies,double distance=2pt] (A5) -- (A1);
\draw[-Implies,double distance=2pt] (A6) -- (A1);
\draw[-Implies,double distance=2pt] (A7) -- (A1);
\draw[-Implies,double distance=2pt] (A2) -- (A7);
\draw[-Implies,double distance=2pt] (A3) -- (A5);
\draw[-Implies,double distance=2pt] (A4) -- (A6);
\draw[->] (A2) to[out=-120,in=120] (A6);
\draw[->] (A3) to[out=0,in=240] (A7);
\draw[->] (A4) to[out=120,in=0] (A5);
\end{scope}
\draw[->] (1.5,0) -- (2.5,0) node[midway,above] {$\mutation_6\mutation_4\mutation_2\mutation_1$};
\draw (6,0) node {$\succ$};
\begin{scope}[xshift=8cm]
\draw[fill] (0,0) node (A1) {} circle (2pt) (210:1) node (A3) {} circle (2pt) 
(150:1) node (A5) {} circle (2pt);
\begin{scope}[lightgray]
\draw[fill] (90:1) node (A2) {} circle (2pt) (-30:1) node (A4) {} circle (2pt) 
(270:1) node (A6) {} circle (2pt) (30:1) node (A7) {} circle (2pt);
\draw (A6) node[below] {$5$} (A7) node[above right] {$7$}
(A2) node[above] {$2$} (A4) node[below right] {$6$};
\draw[->] (A1) -- (A2);
\draw[->] (A1) -- (A4);
\draw[-Implies,double distance=2pt] (A6) -- (A1);
\draw[-Implies,double distance=2pt] (A7) -- (A1);
\draw[-Implies,double distance=2pt] (A2) -- (A7);
\draw[-Implies,double distance=2pt] (A4) -- (A6);
\draw[->] (A2) to[out=-120,in=120] (A6);
\draw[->] (A3) to[out=0,in=240] (A7);
\draw[->] (A4) to[out=120,in=0] (A5);
\end{scope}
\draw (A1) node[above right] {$1$} (A5) node[above left] {$3$} (A3) node[below 
left] {$4$} ;

\draw[->] (A1) -- (A3);
\draw[-Implies,double distance=2pt] (A5) -- (A1);
\draw[-Implies,double distance=2pt] (A3) -- (A5);
\end{scope}
\end{tikzpicture}
\]
This yields a contradiction and we are done.
\end{proof}

\begin{lemma}\label{lemma:no bigons in affine E6 with Z/3Z-action}
For any $i,j\in[7]$, we have
\[
b_{i,j}b_{i,\tau(j)}\ge 0.
\]
\end{lemma}
\begin{proof}
For $i=1$ or $j=1$, there is nothing to prove.
Assume on the contrary that $b_{i,j}b_{i,\tau(j)}<0$ for some $i,j\ge 2$. By 
relabeling if necessary, we may assume that
\[
b_{2,3}<0<b_{2,5}.
\]
Since there is no edges between $3$ and $5$ by Lemma~\ref{lemma:no monogons in 
affine E6 with Z/3Z-action}, the restriction $\quiver|_{\{2,3,5\}}$ is 
mutation-finite if and only if $|b_{2,3}|=b_{2,5}=1$.

Suppose that $b_{2,7}\neq 0$. Then by considering the restriction 
$\quiver|_{\{2,5,7\}}$ or $\quiver|_{\{2,3,7\}}$, we have $|b_{2,7}|=1$. Up to 
relabelling, we may assume that $b_{7,2}=1$. Then the restriction 
$\quiver|_{[7]\setminus\{1\}}=\quiver|_{\{2,3,4,5,6,7\}}$ is reduced to 
$2\exdynA_1$ as follows: 
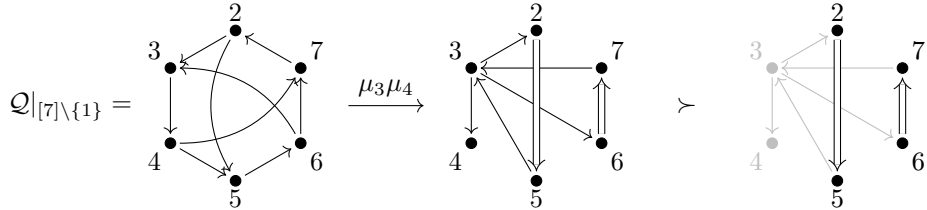
\[
\quiver|_{[7]\setminus\{1\}}=\begin{tikzpicture}[baseline=-.5ex]
\draw[fill] (210:1) node (A3) {} circle (2pt) (90:1) node (A2) {} circle (2pt) 
(-30:1) node (A4) {} circle (2pt) (150:1) node (A5) {} circle (2pt) (270:1) 
node (A6) {} circle (2pt) (30:1) node (A7) {} circle (2pt);
\draw (A5) node[above left] {$3$} (A6) node[below] {$5$} (A7) node[above right] 
{$7$};
\draw (A2) node[above] {$2$} (A3) node[below left] {$4$} (A4) node[below right] 
{$6$};
\draw[->] (A2) -- (A5);
\draw[->] (A3) -- (A6);
\draw[->] (A4) -- (A7);
\draw[->] (A5) -- (A3);
\draw[->] (A6) -- (A4);
\draw[->] (A7) -- (A2);
\draw[->] (A2) to[out=-120,in=120] (A6);
\draw[->] (A3) to[out=0,in=240] (A7);
\draw[->] (A4) to[out=120,in=0] (A5);
\begin{scope}[xshift=4cm]
\draw[fill] (210:1) node (A3) {} circle (2pt) (90:1) node (A2) {} circle (2pt) 
(-30:1) node (A4) {} circle (2pt) (150:1) node (A5) {} circle (2pt) (270:1) 
node (A6) {} circle (2pt) (30:1) node (A7) {} circle (2pt);
\draw (A5) node[above left] {$3$} (A6) node[below] {$5$} (A7) node[above right] 
{$7$};
\draw (A2) node[above] {$2$} (A3) node[below left] {$4$} (A4) node[below right] 
{$6$};
\draw[->] (A5) -- (A2);
\draw[->] (A5) -- (A3);
\draw[->] (A5) -- (A4);
\draw[->] (A6) -- (A5);
\draw[->] (A7) -- (A5);
\draw[-Implies,double distance=2pt] (A2) -- (A6);
\draw[-Implies,double distance=2pt] (A4) -- (A7);
\end{scope}
\draw[->] (1.5,0) -- (2.5,0) node[midway,above] {$\mutation_3\mutation_4$};
\draw (6,0) node {$\succ$};
\begin{scope}[xshift=8cm]
\draw[fill] (90:1) node (A2) {} circle (2pt) (-30:1) node (A4) {} circle (2pt) 
(270:1) node (A6) {} circle (2pt) (30:1) node (A7) {} circle (2pt);
\begin{scope}[lightgray]
\draw[fill] (210:1) node (A3) {} circle (2pt) (150:1) node (A5) {} circle (2pt);
\draw (A5) node[above left] {$3$} (A6) node[below] {$5$} (A7) node[above right] 
{$7$};
\draw (A2) node[above] {$2$} (A3) node[below left] {$4$} (A4) node[below right] 
{$6$};
\draw[->] (A5) -- (A2);
\draw[->] (A5) -- (A3);
\draw[->] (A5) -- (A4);
\draw[->] (A6) -- (A5);
\draw[->] (A7) -- (A5);
\end{scope}
\draw  (A6) node[below] {$5$} (A7) node[above right] {$7$};
\draw (A2) node[above] {$2$} (A4) node[below right] {$6$};
\draw[-Implies,double distance=2pt] (A2) -- (A6);
\draw[-Implies,double distance=2pt] (A4) -- (A7);
\end{scope}
\end{tikzpicture}
\]
However, by Corollary~\ref{corollary:seed restriction}, $\quiver$ can not be 
reduced to $2\exdynA_1$. Therefore, we obtain $b_{2,7}=0$.

On the other hand, we have $|b_{1,j}|\le 1$, otherwise, the restriction 
$\quiver|_{\{1,j,\tau(j)\}}$ is mutation-infinite. If $b_{1,2} b_{1,3}<0$, then 
the quiver $\quiver$ is the same as the last quiver in the proof of 
Lemma~\ref{lemma:no monogons in affine E6 with Z/3Z-action}. Hence, up to the 
mutation $\mutation_1$, we may assume that $b_{1,i}\ge 0$ and so the quiver $\quiver$ 
is one of the following, both are reduced to mutation-infinite quivers:
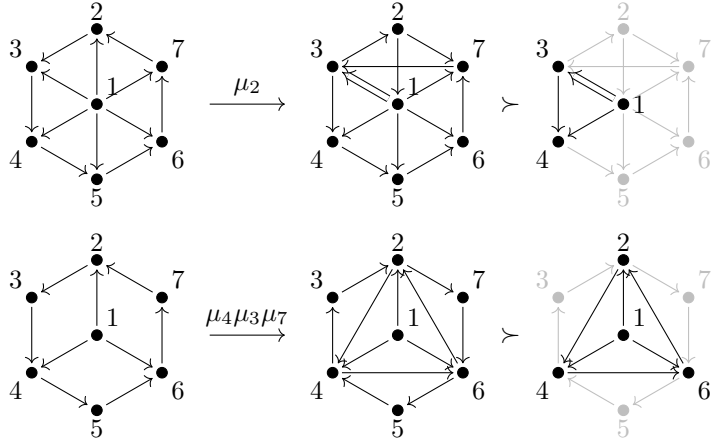
\begin{align*}
&\begin{tikzpicture}[baseline=-.5ex]
\draw[fill] (0,0) node (A1) {} circle (2pt) (210:1) node (A3) {} circle (2pt) 
(90:1) node (A2) {} circle (2pt) (-30:1) node (A4) {} circle (2pt) (150:1) node 
(A5) {} circle (2pt) (270:1) node (A6) {} circle (2pt) (30:1) node (A7) {} 
circle (2pt);
\draw (A1) node[above right] {$1$};
\draw (A5) node[above left] {$3$} (A6) node[below] {$5$} (A7) node[above right] 
{$7$};
\draw (A2) node[above] {$2$} (A3) node[below left] {$4$} (A4) node[below right] 
{$6$};
\draw[->] (A1) -- (A2);
\draw[->] (A1) -- (A3);
\draw[->] (A1) -- (A4);
\draw[->] (A1) -- (A5);
\draw[->] (A1) -- (A6);
\draw[->] (A1) -- (A7);
\draw[->] (A2) -- (A5);
\draw[->] (A3) -- (A6);
\draw[->] (A4) -- (A7);
\draw[->] (A5) -- (A3);
\draw[->] (A6) -- (A4);
\draw[->] (A7) -- (A2);
\begin{scope}[xshift=4cm]
\draw[fill] (0,0) node (A1) {} circle (2pt) (210:1) node (A3) {} circle (2pt) 
(90:1) node (A2) {} circle (2pt) (-30:1) node (A4) {} circle (2pt) (150:1) node 
(A5) {} circle (2pt) (270:1) node (A6) {} circle (2pt) (30:1) node (A7) {} 
circle (2pt);
\draw (A1) node[above right] {$1$}; 
\draw (A5) node[above left] {$3$} (A6) node[below] {$5$} (A7) node[above right] 
{$7$};
\draw (A2) node[above] {$2$} (A3) node[below left] {$4$} (A4) node[below right] 
{$6$};
\draw[->] (A2) -- (A1);
\draw[->] (A1) -- (A3);
\draw[->] (A1) -- (A4);
\draw[-Implies,double distance=2pt] (A1) -- (A5);
\draw[->] (A1) -- (A6);
\draw[->] (A1) -- (A7);
\draw[->] (A5) -- (A2);
\draw[->] (A3) -- (A6);
\draw[->] (A4) -- (A7);
\draw[->] (A5) -- (A3);
\draw[->] (A6) -- (A4);
\draw[->] (A2) -- (A7);
\draw[->] (A7) -- (A5);
\end{scope}
\draw[->] (1.5,0) -- (2.5,0) node[midway,above] {$\mutation_2$};
\draw (5.5,0) node {$\succ$};
\begin{scope}[xshift=7cm]
\draw[fill] (0,0) node (A1) {} circle (2pt) (210:1) node (A3) {} circle (2pt) 
(150:1) node (A5) {} circle (2pt);
\begin{scope}[lightgray]
\draw[fill] (90:1) node (A2) {} circle (2pt) (-30:1) node (A4) {} circle (2pt) 
(270:1) node (A6) {} circle (2pt) (30:1) node (A7) {} circle (2pt);
\draw (A6) node[below] {$5$} (A7) node[above right] {$7$};
\draw (A2) node[above] {$2$} (A4) node[below right] {$6$};
\draw[->] (A2) -- (A1);
\draw[->] (A1) -- (A4);
\draw[->] (A1) -- (A6);
\draw[->] (A1) -- (A7);
\draw[->] (A5) -- (A2);
\draw[->] (A3) -- (A6);
\draw[->] (A4) -- (A7);
\draw[->] (A6) -- (A4);
\draw[->] (A2) -- (A7);
\draw[->] (A7) -- (A5);
\end{scope}
\draw (A1) node[right] {$1$} ;
\draw (A5) node[above left] {$3$} (A3) node[below left] {$4$};
\draw[->] (A1) -- (A3);
\draw[->] (A5) -- (A3);
\draw[-Implies,double distance=2pt] (A1) -- (A5);
\end{scope}
\end{tikzpicture}\\
&\begin{tikzpicture}[baseline=-.5ex]
\draw[fill] (0,0) node (A1) {} circle (2pt) (210:1) node (A3) {} circle (2pt) 
(90:1) node (A2) {} circle (2pt) (-30:1) node (A4) {} circle (2pt) (150:1) node 
(A5) {} circle (2pt) (270:1) node (A6) {} circle (2pt) (30:1) node (A7) {} 
circle (2pt);
\draw (A1) node[above right] {$1$}; 
\draw (A5) node[above left] {$3$} (A6) node[below] {$5$} (A7) node[above right] 
{$7$};
\draw (A2) node[above] {$2$} (A3) node[below left] {$4$} (A4) node[below right] 
{$6$};
\draw[->] (A1) -- (A2);
\draw[->] (A1) -- (A3);
\draw[->] (A1) -- (A4);
\draw[->] (A2) -- (A5);
\draw[->] (A3) -- (A6);
\draw[->] (A4) -- (A7);
\draw[->] (A5) -- (A3);
\draw[->] (A6) -- (A4);
\draw[->] (A7) -- (A2);
\begin{scope}[xshift=4cm]
\draw[fill] (0,0) node (A1) {} circle (2pt) (210:1) node (A3) {} circle (2pt) 
(90:1) node (A2) {} circle (2pt) (-30:1) node (A4) {} circle (2pt) (150:1) node 
(A5) {} circle (2pt) (270:1) node (A6) {} circle (2pt) (30:1) node (A7) {} 
circle (2pt);
\draw (A1) node[above right] {$1$}; 
\draw (A5) node[above left] {$3$} (A6) node[below] {$5$} (A7) node[above right] 
{$7$};
\draw (A2) node[above] {$2$} (A3) node[below left] {$4$} (A4) node[below right] 
{$6$};
\draw[->] (A1) -- (A2);
\draw[->] (A1) -- (A3);
\draw[->] (A1) -- (A4);
\draw[->] (A2) -- (A3);
\draw[->] (A3) -- (A4);
\draw[->] (A4) -- (A2);
\draw[->] (A2) -- (A7);
\draw[->] (A7) -- (A4);
\draw[->] (A4) -- (A6);
\draw[->] (A6) -- (A3);
\draw[->] (A3) -- (A5);
\draw[->] (A5) -- (A2);
\end{scope}
\draw[->] (1.5,0) -- (2.5,0) node[midway,above] {$\mutation_4\mutation_3\mutation_7$};
\draw (5.5,0) node {$\succ$};
\begin{scope}[xshift=7cm]
\draw[fill] (0,0) node (A1) {} circle (2pt) (210:1) node (A3) {} circle (2pt) 
(90:1) node (A2) {} circle (2pt) (-30:1) node (A4) {} circle (2pt);
\begin{scope}[lightgray]
\draw[fill] (150:1) node (A5) {} circle (2pt) (270:1) node (A6) {} circle (2pt) 
(30:1) node (A7) {} circle (2pt);
\draw (A5) node[above left] {$3$} (A6) node[below] {$5$} (A7) node[above right] 
{$7$};
\draw[->] (A2) -- (A7);
\draw[->] (A7) -- (A4);
\draw[->] (A4) -- (A6);
\draw[->] (A6) -- (A3);
\draw[->] (A3) -- (A5);
\draw[->] (A5) -- (A2);
\end{scope}
\draw (A1) node[above right] {$1$} (A2) node[above] {$2$} (A3) node[below left] 
{$4$} (A4) node[below right] {$6$};
\draw[->] (A1) -- (A2);
\draw[->] (A1) -- (A3);
\draw[->] (A1) -- (A4);
\draw[->] (A2) -- (A3);
\draw[->] (A3) -- (A4);
\draw[->] (A4) -- (A2);
\end{scope}
\end{tikzpicture}
\end{align*}
This yields a contradiction and we are done.
\end{proof}

\subsubsection{\texorpdfstring{$(\dynX, G,\dynY) = (\exdynE_6, \Z/2\Z, \dynE_6^{(2)})$}{(X, G, Y)=(affine E6, Z/2Z, affine E6 2)}}
\label{section_E6_E6_2}

Let $\quiver$ be a $\Z/2\Z$-invariant quiver on $[7]$ of type $\exdynE_6$ and $\qbasis=\qbasis(\quiver)$. 
See~\ref{fig_E6_E6_2} in 
Appendix~\ref{appendix_actions_on_Dynkin_diagrams}.
To show the admissibility, it is enough to check the 
condition~\ref{def_admissible_4} because of 
Lemma~\ref{lemma:condition_3_holds_of_order_2}. 

\begin{lemma}\label{lemma:no bigons in affine E6 with Z/2Z-action}
For any $i,j\in[7]$,
\[
b_{i,j}b_{i,\tau(j)}\ge 0.
\]
\end{lemma}
\begin{proof}
If $i\le 3$ or $j\le 3$, then $b_{i,j}b_{i,\tau(j)}=b_{i,j}^2\ge 0$ since 
$\tau(i)=i$ or $\tau(j)=j$.

Suppose that $b_{i,j}b_{i,\tau(j)}<0$ for some $i,j\ge 4$. Then the only 
possibility is up to relabelling,
\[
b_{7,6}=b_{5,4}<0<b_{5,6}=b_{7,4}.
\]
As seen in the previous lemma, $b_{4,6}=0$ and therefore by considering the 
restriction $\quiver|_{\{4,5,6\}}$, we have 
$b_{4,5}=b_{5,6}=b_{6,7}=b_{7,4}=1$. Therefore $\quiver|_{\{4,5,6,7\}}$ is a 
cyclic graph as follows:
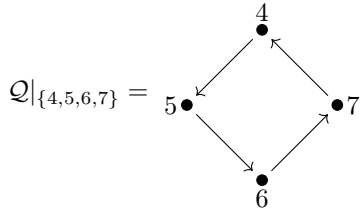
\begin{align*}
\quiver|_{\{4,5,6,7\}}&=
\begin{tikzpicture}[baseline=.5ex]
\draw[fill] (0,1) node (A4) {} circle (2pt)
(-1,0) node (A5) {} circle (2pt) 
(0,-1) node (A6) {} circle (2pt) 
(1,0) node (A7) {} circle (2pt);
\draw (A4) node[above] {$4$};
\draw (A5) node[left] {$5$};
\draw (A6) node[below] {$6$};
\draw (A7) node[right] {$7$};
\draw[->] (A4) -- (A5);
\draw[->] (A5) -- (A6);
\draw[->] (A6) -- (A7);
\draw[->] (A7) -- (A4);
\end{tikzpicture}
\end{align*}

Since $\quiver$ is connected, $b_{i,j}\neq 0$ for some $i\le 3<j$. Let us 
assume that $b_{1,j}\neq0$. Then by $\Z/2\Z$-invariance, the restriction 
$\quiver|_{\{1,4,5,6,7\}}$ is one of the following: up to relabelling and 
mutation $\mutation_1$,
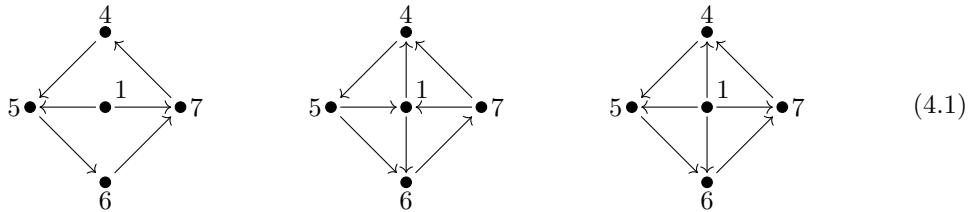
\begin{align}\label{equation:squares}
\begin{tikzpicture}[baseline=-.5ex]
\begin{scope}
\draw[fill] (0,0) node (A1) {} circle (2pt)
(0,1) node (A4) {} circle (2pt)
(-1,0) node (A5) {} circle (2pt) 
(0,-1) node (A6) {} circle (2pt) 
(1,0) node (A7) {} circle (2pt);
\draw[->] (A1) node[above right] {$1$} -- (A5);
\draw[->] (A1) -- (A7);
\draw[->] (A4) -- (A5) node[left] {$5$};
\draw[->] (A5) -- (A6) node[below] {$6$};
\draw[->] (A6) -- (A7) node[right] {$7$};
\draw[->] (A7) -- (A4) node[above] {$4$};
\end{scope}
\begin{scope}[xshift=4cm]
\draw[fill] (0,0) node (A1) {} circle (2pt)
(0,1) node (A4) {} circle (2pt)
(-1,0) node (A5) {} circle (2pt) 
(0,-1) node (A6) {} circle (2pt) 
(1,0) node (A7) {} circle (2pt);
\draw[->] (A1) node[above right] {$1$} -- (A4);
\draw[->] (A1) -- (A6);
\draw[->] (A5) -- (A1);
\draw[->] (A7) -- (A1);
\draw[->] (A4) -- (A5) node[left] {$5$};
\draw[->] (A5) -- (A6) node[below] {$6$};
\draw[->] (A6) -- (A7) node[right] {$7$};
\draw[->] (A7) -- (A4) node[above] {$4$};
\end{scope}
\begin{scope}[xshift=8cm]
\draw[fill] (0,0) node (A1) {} circle (2pt)
(0,1) node (A4) {} circle (2pt)
(-1,0) node (A5) {} circle (2pt) 
(0,-1) node (A6) {} circle (2pt) 
(1,0) node (A7) {} circle (2pt);
\draw[->] (A1) node[above right] {$1$} -- (A4);
\draw[->] (A1) -- (A5);
\draw[->] (A1) -- (A6);
\draw[->] (A1) -- (A7);
\draw[->] (A4) -- (A5) node[left] {$5$};
\draw[->] (A5) -- (A6) node[below] {$6$};
\draw[->] (A6) -- (A7) node[right] {$7$};
\draw[->] (A7) -- (A4) node[above] {$4$};
\end{scope}
\end{tikzpicture}
\end{align}
Then indeed, first two quivers are mutation equivalent via 
$\mutation_5\mutation_7\mutation_4\mutation_6$.
\begin{align}\label{equation:square with one diagonal}
\begin{tikzpicture}[baseline=-.5ex]
\begin{scope}
\draw[fill] (0,0) node (A1) {} circle (2pt)
(0,1) node (A4) {} circle (2pt)
(-1,0) node (A5) {} circle (2pt) 
(0,-1) node (A6) {} circle (2pt) 
(1,0) node (A7) {} circle (2pt);
\draw[->] (A1) node[above right] {$1$} -- (A5);
\draw[->] (A1) -- (A7);
\draw[->] (A4) -- (A5) node[left] {$5$};
\draw[->] (A5) -- (A6) node[below] {$6$};
\draw[->] (A6) -- (A7) node[right] {$7$};
\draw[->] (A7) -- (A4) node[above] {$4$};
\end{scope}
\begin{scope}[xshift=6cm]
\draw[fill] (0,0) node (A1) {} circle (2pt)
(0,1) node (A4) {} circle (2pt)
(-1,0) node (A5) {} circle (2pt) 
(0,-1) node (A6) {} circle (2pt) 
(1,0) node (A7) {} circle (2pt);
\draw[->] (A1) node[above right] {$1$} -- (A4);
\draw[->] (A1) -- (A6);
\draw[->] (A5) -- (A1);
\draw[->] (A7) -- (A1);
\draw[->] (A4) -- (A5) node[left] {$5$};
\draw[->] (A5) -- (A6) node[below] {$6$};
\draw[->] (A6) -- (A7) node[right] {$7$};
\draw[->] (A7) -- (A4) node[above] {$4$};
\end{scope}
\draw[->] (2,0) -- (4,0) node[midway, above] {$\mutation_5\mutation_7\mutation_4\mutation_6$}; 
\end{tikzpicture}
\end{align}

Finally, the last two out of the above three can be reduced as follows:
\begin{align}
&\begin{tikzpicture}[baseline=-.5ex]
\begin{scope}
\draw[fill] (0,0) node (A1) {} circle (2pt)
(0,1) node (A4) {} circle (2pt)
(-1,0) node (A5) {} circle (2pt) 
(0,-1) node (A6) {} circle (2pt) 
(1,0) node (A7) {} circle (2pt);
\draw[->] (A1) node[above right] {$1$} -- (A4);
\draw[->] (A1) -- (A6) node[below] {$6$};
\draw[->] (A5) -- (A1);
\draw[->] (A7) -- (A1);
\draw[->] (A4) -- (A5) node[left] {$5$};
\draw[->] (A5) -- (A6);
\draw[->] (A6) -- (A7) node[right] {$7$};
\draw[->] (A7) -- (A4) node[above] {$4$};
\end{scope}
\begin{scope}[xshift=5cm]
\draw[fill] (0,0) node (A1) {} circle (2pt)
(0,1) node (A4) {} circle (2pt)
(-1,0) node (A5) {} circle (2pt) 
(0,-1) node (A6) {} circle (2pt) 
(1,0) node (A7) {} circle (2pt);
\draw[->] (A4) node[above] {$4$} -- (A1) node[above right] {$1$};
\draw[->] (A1) -- (A5) node[left] {$5$};
\draw[->] (A6) node[below] {$6$} -- (A1);
\draw[->] (A1) -- (A7) node[right] {$7$};
\draw[-Implies,double distance=2pt] (A5) -- (A6);
\draw[-Implies,double distance=2pt] (A7) -- (A4);
\end{scope}
\draw[->] (2,0) -- (3,0) node[midway, above] {$\mutation_1$}; 
\draw (7,0) node {$\succ$};
\begin{scope}[xshift=9cm]
\draw[fill] (-1,0) node (A5) {} circle (2pt) (0,1) node (A4) {} circle (2pt) 
(0,-1) node (A6) {} circle (2pt) (1,0) node (A7) {} circle (2pt);
\begin{scope}[lightgray]
\draw[fill] (0,0) node (A1) {} circle (2pt);
\draw[->] (A4) -- (A1) node[above right] {$1$};
\draw[->] (A1) -- (A5);
\draw[->] (A6) -- (A1);
\draw[->] (A1) -- (A7);
\end{scope}
\draw[-Implies,double distance=2pt] (A5) node[left] {$5$} -- (A6) node[below] 
{$6$};
\draw[-Implies,double distance=2pt] (A7) node[right] {$7$} -- (A4) node[above] 
{$4$};
\end{scope}
\end{tikzpicture}\label{equation:square with two diagonals}
\\
&\begin{tikzpicture}[baseline=-.5ex]
\begin{scope}
\draw[fill] (0,0) node (A1) {} circle (2pt)
(0,1) node (A4) {} circle (2pt)
(-1,0) node (A5) {} circle (2pt) 
(0,-1) node (A6) {} circle (2pt) 
(1,0) node (A7) {} circle (2pt);
\draw[->] (A1) node[above right] {$1$} -- (A4) node[above] {$4$};
\draw[->] (A1) -- (A5);
\draw[->] (A1) -- (A6);
\draw[->] (A1) -- (A7);
\draw[->] (A4) -- (A5) node[left] {$5$};
\draw[->] (A5) -- (A6) node[below] {$6$};
\draw[->] (A6) -- (A7) node[right] {$7$};
\draw[->] (A7) -- (A4);
\end{scope}
\begin{scope}[xshift=5cm]
\draw[fill] (0,0) node (A1) {} circle (2pt)
(0,1) node (A4) {} circle (2pt)
(-1,0) node (A5) {} circle (2pt) 
(0,-1) node (A6) {} circle (2pt) 
(1,0) node (A7) {} circle (2pt);
\draw[-Implies,double distance=2pt] (A1) -- (A4);
\draw[->] (A5) -- (A1) node[above right] {$1$};
\draw[-Implies,double distance=2pt] (A1) -- (A6);
\draw[->] (A7) -- (A1);
\draw[->] (A5) -- (A4) node[above] {$4$};
\draw[->] (A6) -- (A5) node[left] {$5$};
\draw[->] (A7) -- (A6) node[below] {$6$};
\draw[->] (A4) -- (A7) node[right] {$7$};
\end{scope}
\draw[->] (2,0) -- (3,0) node[midway, above] {$\mutation_5\mutation_7$}; 
\draw (7,0) node {$\succ$};
\begin{scope}[xshift=9cm]
\draw[fill] (0,0) node (A1) {} circle (2pt) (0,1) node (A4) {} circle (2pt) 
(0,-1) node (A6) {} circle (2pt);
\begin{scope}[lightgray]
\draw[fill] (-1,0) node (A5) {} circle (2pt) (1,0) node (A7) {} circle (2pt);
\draw[->] (A5) -- (A1);
\draw[->] (A7) -- (A1);
\draw[->] (A5) -- (A4);
\draw[->] (A6) -- (A5) node[left] {$5$};
\draw[->] (A7) -- (A6);
\draw[->] (A4) -- (A7) node[right] {$7$};
\end{scope}
\draw[-Implies,double distance=2pt] (A1) -- (A4) node[above] {$4$};
\draw[-Implies,double distance=2pt] (A1) node[above right] {$1$} -- (A6) 
node[below] {$6$};
\end{scope}
\end{tikzpicture}\label{equation:square with two diagonals 2}
\end{align}
This yields a contradiction so we are done.
\end{proof}

\subsubsection{\texorpdfstring{$(\dynX,G,\dynY)=(\exdynE_7, \Z/2\Z, \exdynF_4)$}{(X, G, Y)=(affine E7, Z/2Z, affine F4)}}
\label{section_E7_F4}

Let $\quiver$ be a $\Z/2\Z$-invariant quiver on $[8]$ of type $\exdynE_7$ and $\qbasis=\qbasis(\quiver)$.  
See~\ref{fig_E7_F4} in Appendix~\ref{appendix_actions_on_Dynkin_diagrams}. 
To show the admissibility, it is enough to check the 
condition~\ref{def_admissible_4} because of 
Lemma~\ref{lemma:condition_3_holds_of_order_2}. 

\begin{lemma}\label{lemma:no bigons in affine E7 with Z/2Z-action}
For any $i,j\in[8]$,
\[
b_{i,j}b_{i,\tau(j)}\ge 0.
\]
\end{lemma}
\begin{proof}
If $i\le 2$ or $j\le 2$, then $b_{i,j}b_{i,\tau(j)}=b_{i,j}^2\ge 0$ since 
$\tau(i)=i$ or $\tau(j)=j$.
Suppose that $b_{i,j}b_{i,\tau(j)}<0$ for some $i,j\ge 3$. Then up to 
relabelling, we may assume that
\[
b_{8,7}=b_{5,4}<0<b_{5,7}=b_{8,4}.
\]
As before, $\quiver|_{\{4,5,7,8\}}$ is a cyclic graph and so we must have 
$b_{4,5}=b_{5,7}=b_{7,8}=b_{8,4}=1$.
\begin{align*}
\quiver|_{\{4,5,7,8\}}&=
\begin{tikzpicture}[baseline=.5ex]
\draw[fill] (0,1) node (A4) {} circle (2pt)
(-1,0) node (A5) {} circle (2pt) 
(0,-1) node (A7) {} circle (2pt) 
(1,0) node (A8) {} circle (2pt);
\draw (A4) node[above] {$4$};
\draw (A5) node[left] {$5$};
\draw (A7) node[below] {$7$};
\draw (A8) node[right] {$8$};
\draw[->] (A4) -- (A5);
\draw[->] (A5) -- (A7);
\draw[->] (A7) -- (A8);
\draw[->] (A8) -- (A4);
\end{tikzpicture}
\end{align*}

Suppose that $b_{i, j}\neq 0$ for $i\le 2$ and $j\le 7$. Then the restriction 
$\quiver|_{\{i,4,5,7,8\}}$ will be reduced to a mutation-infinite quiver by 
Lemma~\ref{lemma:no bigons in affine E6 with Z/2Z-action}. Hence we may assume 
that for $i\le 2$,
\[
b_{i,4}=b_{i,5}=b_{i,7}=b_{i,8}=0,
\]
and since $\quiver$ is connected, $b_{i,3}=b_{i,6}\neq0$ for some $i\le 2$, say 
$i=1$. Then the mutation-finiteness of the restriction $\quiver|_{\{1,3,6\}}$ 
implies that $b_{1,3}=b_{1,6}=\pm1$.

If $b_{3,j}b_{3,\tau(j)}<0$ for some $j=4,5,7,8$, then the restriction 
$\quiver|_{\{1,3,6,j,\tau(j)\}}$ is mutation-infinite as before. Therefore, 
$b_{3,j}b_{3,\tau(j)}\ge 0$ for all $j\in[8]$ since $b_{3,j}b_{3,\tau(j)}\ge 0$ 
for any $j=1,2,3,6$.

Again, since $\quiver$ is connected, one of $b_{3,4}, b_{3,5}, b_{3,7}$ and 
$b_{3,8}$ is nonzero.
Suppose that none of $b_{3,j}$ for $j=4,5,7,8$ is 
zero. Then two restrictions $\quiver|_{\{3,4,7\}}$ and $\quiver|_{\{3,5,8\}}$ 
force us to have 
\[
b_{3,4}=b_{3,7}=\pm1\quad\text{ and }\quad b_{3,5}=b_{3,7}=\pm1.
\]
Then the restriction $\quiver|_{\{3,4,5,7,8\}}$ is up to mutation $\mutation_3$ one 
of the following:
\begin{align*}
\begin{tikzpicture}[baseline=-.5ex]
\draw[fill] (0,0) node (A1) {} circle (2pt)
(0,1) node (A4) {} circle (2pt)
(-1,0) node (A5) {} circle (2pt) 
(0,-1) node (A6) {} circle (2pt) 
(1,0) node (A7) {} circle (2pt);
\draw[->] (A4) -- (A1) node[above right] {$3$};
\draw[->] (A6) -- (A1);
\draw[->] (A1) -- (A5);
\draw[->] (A1) -- (A7);
\draw[->] (A4) -- (A5) node[left] {$5$};
\draw[->] (A5) -- (A6) node[below] {$7$};
\draw[->] (A6) -- (A7) node[right] {$8$};
\draw[->] (A7) -- (A4) node[above] {$4$};
\begin{scope}[xshift=4cm]
\draw[fill] (0,0) node (A1) {} circle (2pt)
(0,1) node (A4) {} circle (2pt)
(-1,0) node (A5) {} circle (2pt) 
(0,-1) node (A6) {} circle (2pt) 
(1,0) node (A7) {} circle (2pt);
\draw[->] (A1) node[above right] {$3$} -- (A4);
\draw[->] (A1) -- (A6);
\draw[->] (A5) -- (A1);
\draw[->] (A7) -- (A1);
\draw[->] (A4) -- (A5) node[left] {$5$};
\draw[->] (A5) -- (A6) node[below] {$7$};
\draw[->] (A6) -- (A7) node[right] {$8$};
\draw[->] (A7) -- (A4) node[above] {$4$};
\end{scope}
\begin{scope}[xshift=8cm]
\draw[fill] (0,0) node (A1) {} circle (2pt)
(0,1) node (A4) {} circle (2pt)
(-1,0) node (A5) {} circle (2pt) 
(0,-1) node (A6) {} circle (2pt) 
(1,0) node (A7) {} circle (2pt);
\draw[->] (A1) node[above right] {$3$} -- (A4);
\draw[->] (A1) -- (A5);
\draw[->] (A1) -- (A6);
\draw[->] (A1) -- (A7);
\draw[->] (A4) -- (A5) node[left] {$5$};
\draw[->] (A5) -- (A6) node[below] {$7$};
\draw[->] (A6) -- (A7) node[right] {$8$};
\draw[->] (A7) -- (A4) node[above] {$4$};
\end{scope}
\end{tikzpicture}
\end{align*}
However, as seen in \eqref{equation:square with two diagonals} and 
\eqref{equation:square with two diagonals 2}, these three are reduced to either 
a mutation-infinite quiver or $\quiver(2\exdynA_1)$, which are impossible. Therefore, at 
least one of $b_{3,4},b_{3,5}, b_{3,7}$ and $b_{3,8}$ is zero.
\smallskip

\noindent \textbf{Case (i)} Suppose that $b_{3,8}=0$ but $b_{3,4}, b_{3,7}, 
b_{3,5}\neq 0$. 
Then the restriction $\quiver|_{\{3,4,5,7,8\}}$ is reduced to a 
mutation-infinite quiver as follows:
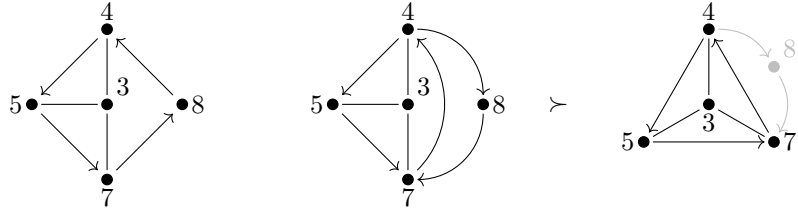
\[
\begin{tikzpicture}[baseline=-.5ex]
\begin{scope}
\draw[fill] (0,0) node (A1) {} circle (2pt)
(0,1) node (A4) {} circle (2pt)
(-1,0) node (A5) {} circle (2pt) 
(0,-1) node (A6) {} circle (2pt) 
(1,0) node (A7) {} circle (2pt);
\draw (A4) -- (A1) node[above right] {$3$};
\draw (A6) -- (A1);
\draw (A1) -- (A5);
\draw[->] (A4) -- (A5) node[left] {$5$};
\draw[->] (A5) -- (A6) node[below] {$7$};
\draw[->] (A6) -- (A7) node[right] {$8$};
\draw[->] (A7) -- (A4) node[above] {$4$};
\end{scope}
\begin{scope}[xshift=4cm]
\draw[fill] (0,0) node (A1) {} circle (2pt)
(0,1) node (A4) {} circle (2pt)
(-1,0) node (A5) {} circle (2pt) 
(0,-1) node (A6) {} circle (2pt) 
(1,0) node (A7) {} circle (2pt);
\draw (A4) -- (A1) node[above right] {$3$};
\draw (A6) -- (A1);
\draw (A1) -- (A5);
\draw[->] (A4) -- (A5) node[left] {$5$};
\draw[->] (A5) -- (A6) node[below] {$7$};
\draw[->] (A7) node[right] {$8$} to[out=-90,in=0] (A6);
\draw[->] (A4) node[above] {$4$} to[out=0,in=90] (A7);
\draw[->] (A6) to[out=45,in=-45] (A4);
\end{scope}
\begin{scope}[xshift=8cm]
\draw[fill] (0,0) node (A1) {} circle (2pt) (210:1) node (A3) {} circle (2pt) 
(90:1) node (A2) {} circle (2pt) (-30:1) node (A4) {} circle (2pt);
\begin{scope}[lightgray]
\draw[fill] (30:1) node (A8) {} circle (2pt);
\draw[->] (A2) to[out=0,in=120] (A8) node[above right] {$8$};
\draw[->] (A8) to[out=-60,in=60] (A4);
\end{scope}
\draw (A1) node[below] {$3$} -- (A2);
\draw (A1) -- (A3);
\draw (A1) -- (A4);
\draw[->] (A2) node[above] {$4$} -- (A3);
\draw[->] (A3) node[left] {$5$} -- (A4);
\draw[->] (A4) node[right] {$7$} -- (A2);
\end{scope}
\draw (6,0) node {$\succ$};
\end{tikzpicture}
\]
This yields a contradiction so this case cannot occur.

\smallskip
\noindent \textbf{Case (ii)} Suppose that $b_{3,8}=b_{3,5}=0$ but $b_{3,4}, 
b_{3,7}\neq 0$. Then 
$b_{3,4}=b_{3,7}=\pm1$ by the mutation-finiteness of $\quiver|_{\{3,4,7\}}$ and 
so the restriction $\quiver|_{\{3,4,5,7,8\}}$ is reduced to $\quiver(2\exdynA_1)$ as 
seen in \eqref{equation:square with one diagonal} and~\eqref{equation:square 
with two diagonals}, which yields a contradiction so this case cannot happen.

\smallskip
\noindent \textbf{Case (iii)}
Suppose that $b_{3,7}=b_{3,8}=0$. Considering
$\quiver|_{[8]\setminus\{2\}}$, there are four cases as follows:
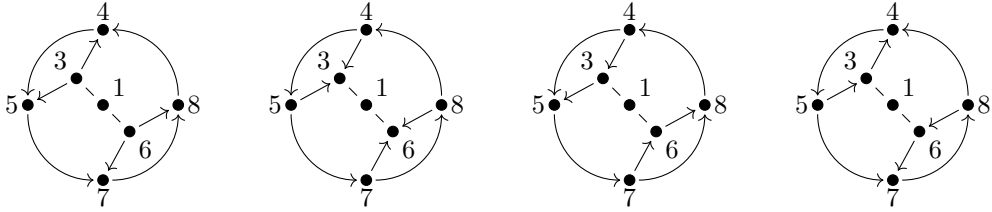
\[
\begin{tikzpicture}[baseline=-.5ex]
\begin{scope}
\draw[fill] (0,0) node (A1) {} circle(2pt)
(135:0.5) node (A3) {} circle (2pt)
(-45:0.5) node (A6) {} circle (2pt)
(0,1) node (A4) {} circle (2pt)
(-1,0) node (A5) {} circle (2pt) 
(0,-1) node (A7) {} circle (2pt) 
(1,0) node (A8) {} circle (2pt);
\draw[->] (A4) to[out=180,in=90] (A5) node[left] {$5$};
\draw[->] (A5) to[out=-90,in=180] (A7) node[below] {$7$};
\draw[->] (A7) to[out=0,in=-90] (A8) node[right] {$8$};
\draw[->] (A8) to[out=90,in=0] (A4) node[above] {$4$};
\draw (A1) node [above right] {$1$} -- (A3);
\draw (A1) -- (A6);
\draw[->] (A3) node[above left] {$3$} -- (A4);
\draw[->] (A3) -- (A5);
\draw[->] (A6) node[below right] {$6$} -- (A7);
\draw[->] (A6) -- (A8);
\end{scope}
\begin{scope}[xshift=3.5cm]
\draw[fill] (0,0) node (A1) {} circle(2pt)
(135:0.5) node (A3) {} circle (2pt)
(-45:0.5) node (A6) {} circle (2pt)
(0,1) node (A4) {} circle (2pt)
(-1,0) node (A5) {} circle (2pt) 
(0,-1) node (A7) {} circle (2pt) 
(1,0) node (A8) {} circle (2pt);
\draw[->] (A4) to[out=180,in=90] (A5) node[left] {$5$};
\draw[->] (A5) to[out=-90,in=180] (A7) node[below] {$7$};
\draw[->] (A7) to[out=0,in=-90] (A8) node[right] {$8$};
\draw[->] (A8) to[out=90,in=0] (A4) node[above] {$4$};
\draw (A1) node [above right] {$1$} -- (A3);
\draw (A1) -- (A6);
\draw[->] (A4) -- (A3) node[above left] {$3$};
\draw[->] (A5) -- (A3);
\draw[->] (A7) -- (A6) node[below right] {$6$};
\draw[->] (A8) -- (A6);
\end{scope}
\begin{scope}[xshift=7cm]
\draw[fill] (0,0) node (A1) {} circle(2pt)
(135:0.5) node (A3) {} circle (2pt)
(-45:0.5) node (A6) {} circle (2pt)
(0,1) node (A4) {} circle (2pt)
(-1,0) node (A5) {} circle (2pt) 
(0,-1) node (A7) {} circle (2pt) 
(1,0) node (A8) {} circle (2pt);
\draw[->] (A4) to[out=180,in=90] (A5) node[left] {$5$};
\draw[->] (A5) to[out=-90,in=180] (A7) node[below] {$7$};
\draw[->] (A7) to[out=0,in=-90] (A8) node[right] {$8$};
\draw[->] (A8) to[out=90,in=0] (A4) node[above] {$4$};
\draw (A1) node [above right] {$1$} -- (A3);
\draw (A1) -- (A6);
\draw[->] (A4) -- (A3) node[above left] {$3$};
\draw[->] (A3) -- (A5);
\draw[->] (A7) -- (A6) node[below right] {$6$};
\draw[->] (A6) -- (A8);
\end{scope}
\begin{scope}[xshift=10.5cm]
\draw[fill] (0,0) node (A1) {} circle(2pt)
(135:0.5) node (A3) {} circle (2pt)
(-45:0.5) node (A6) {} circle (2pt)
(0,1) node (A4) {} circle (2pt)
(-1,0) node (A5) {} circle (2pt) 
(0,-1) node (A7) {} circle (2pt) 
(1,0) node (A8) {} circle (2pt);
\draw[->] (A4) to[out=180,in=90] (A5) node[left] {$5$};
\draw[->] (A5) to[out=-90,in=180] (A7) node[below] {$7$};
\draw[->] (A7) to[out=0,in=-90] (A8) node[right] {$8$};
\draw[->] (A8) to[out=90,in=0] (A4) node[above] {$4$};
\draw (A1) node [above right] {$1$} -- (A3);
\draw (A1) -- (A6);
\draw[->] (A3) node[above left] {$3$} -- (A4);
\draw[->] (A5) -- (A3);
\draw[->] (A6) node[below right] {$6$} -- (A7);
\draw[->] (A8) -- (A6);
\end{scope}
\end{tikzpicture}
\]
One can check easily that the first two are reduced to $\quiver(2\exdynA_1)$ as follows:
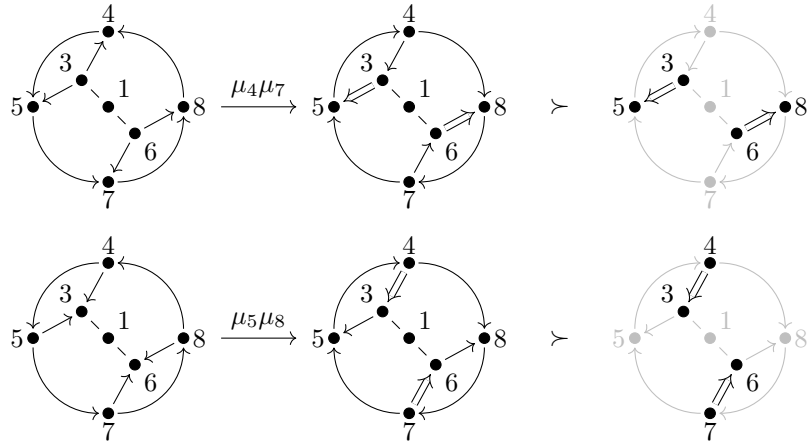
\begin{align*}
&\begin{tikzpicture}[baseline=-.5ex]
\begin{scope}
\draw[fill] (0,0) node (A1) {} circle(2pt)
(135:0.5) node (A3) {} circle (2pt)
(-45:0.5) node (A6) {} circle (2pt)
(0,1) node (A4) {} circle (2pt)
(-1,0) node (A5) {} circle (2pt) 
(0,-1) node (A7) {} circle (2pt) 
(1,0) node (A8) {} circle (2pt);
\draw[->] (A4) to[out=180,in=90] (A5) node[left] {$5$};
\draw[->] (A5) to[out=-90,in=180] (A7) node[below] {$7$};
\draw[->] (A7) to[out=0,in=-90] (A8) node[right] {$8$};
\draw[->] (A8) to[out=90,in=0] (A4) node[above] {$4$};
\draw (A1) node [above right] {$1$} -- (A3);
\draw (A1) -- (A6);
\draw[->] (A3) node[above left] {$3$} -- (A4);
\draw[->] (A3) -- (A5);
\draw[->] (A6) node[below right] {$6$} -- (A7);
\draw[->] (A6) -- (A8);
\end{scope}
\begin{scope}[xshift=4cm]
\draw[fill] (0,0) node (A1) {} circle(2pt)
(135:0.5) node (A3) {} circle (2pt)
(-45:0.5) node (A6) {} circle (2pt)
(0,1) node (A4) {} circle (2pt)
(-1,0) node (A5) {} circle (2pt) 
(0,-1) node (A7) {} circle (2pt) 
(1,0) node (A8) {} circle (2pt);
\draw[<-] (A4) to[out=180,in=90] (A5) node[left] {$5$};
\draw[<-] (A5) to[out=-90,in=180] (A7) node[below] {$7$};
\draw[<-] (A7) to[out=0,in=-90] (A8) node[right] {$8$};
\draw[<-] (A8) to[out=90,in=0] (A4) node[above] {$4$};
\draw (A1) node [above right] {$1$} -- (A3);
\draw (A1) -- (A6);
\draw[->] (A4) -- (A3) node[above left] {$3$};
\draw[-Implies,double distance=2pt] (A3) -- (A5);
\draw[->] (A7) -- (A6) node[below right] {$6$};
\draw[-Implies,double distance=2pt] (A6) -- (A8);
\end{scope}
\begin{scope}[xshift=8cm]
\draw[fill] (135:0.5) node (A3) {} circle (2pt)
(-1,0) node (A5) {} circle (2pt) 
(-45:0.5) node (A6) {} circle (2pt)
(1,0) node (A8) {} circle (2pt);
\begin{scope}[lightgray]
\draw[fill] (0,0) node (A1) {} circle(2pt)
(0,1) node (A4) {} circle (2pt)
(0,-1) node (A7) {} circle (2pt);
\draw[<-] (A4) to[out=180,in=90] (A5);
\draw[<-] (A5) to[out=-90,in=180] (A7) node[below] {$7$};
\draw[<-] (A7) to[out=0,in=-90] (A8);
\draw[<-] (A8) to[out=90,in=0] (A4) node[above] {$4$};
\draw (A1) node [above right] {$1$} -- (A3);
\draw (A1) -- (A6);
\draw[->] (A4) -- (A3);
\draw[->] (A7) -- (A6);
\end{scope}
\draw[-Implies,double distance=2pt] (A3) node[above left] {$3$} -- (A5) 
node[left] {$5$};
\draw[-Implies,double distance=2pt] (A6) node[below right] {$6$} -- (A8) 
node[right] {$8$};
\end{scope}
\draw[->] (1.5,0) -- (2.5,0) node[midway, above] {$\mutation_4\mutation_7$};
\draw (6,0) node {$\succ$};
\end{tikzpicture}\\
&\begin{tikzpicture}[baseline=-.5ex]
\begin{scope}
\draw[fill] (0,0) node (A1) {} circle(2pt)
(135:0.5) node (A3) {} circle (2pt)
(-45:0.5) node (A6) {} circle (2pt)
(0,1) node (A4) {} circle (2pt)
(-1,0) node (A5) {} circle (2pt) 
(0,-1) node (A7) {} circle (2pt) 
(1,0) node (A8) {} circle (2pt);
\draw[->] (A4) to[out=180,in=90] (A5) node[left] {$5$};
\draw[->] (A5) to[out=-90,in=180] (A7) node[below] {$7$};
\draw[->] (A7) to[out=0,in=-90] (A8) node[right] {$8$};
\draw[->] (A8) to[out=90,in=0] (A4) node[above] {$4$};
\draw (A1) node [above right] {$1$} -- (A3);
\draw (A1) -- (A6);
\draw[->] (A4) -- (A3) node[above left] {$3$};
\draw[->] (A5) -- (A3);
\draw[->] (A7) -- (A6) node[below right] {$6$};
\draw[->] (A8) -- (A6);
\end{scope}
\begin{scope}[xshift=4cm]
\draw[fill] (0,0) node (A1) {} circle(2pt)
(135:0.5) node (A3) {} circle (2pt)
(-45:0.5) node (A6) {} circle (2pt)
(0,1) node (A4) {} circle (2pt)
(-1,0) node (A5) {} circle (2pt) 
(0,-1) node (A7) {} circle (2pt) 
(1,0) node (A8) {} circle (2pt);
\draw[<-] (A4) to[out=180,in=90] (A5) node[left] {$5$};
\draw[<-] (A5) to[out=-90,in=180] (A7) node[below] {$7$};
\draw[<-] (A7) to[out=0,in=-90] (A8) node[right] {$8$};
\draw[<-] (A8) to[out=90,in=0] (A4) node[above] {$4$};
\draw (A1) node [above right] {$1$} -- (A3);
\draw (A1) -- (A6);
\draw[-Implies,double distance=2pt] (A4) -- (A3) node[above left] {$3$};
\draw[->] (A3) -- (A5);
\draw[-Implies,double distance=2pt] (A7) -- (A6) node[below right] {$6$};
\draw[->] (A6) -- (A8);
\end{scope}
\begin{scope}[xshift=8cm]
\draw[fill] (135:0.5) node (A3) {} circle (2pt)
(-45:0.5) node (A6) {} circle (2pt)
(0,1) node (A4) {} circle (2pt)
(0,-1) node (A7) {} circle (2pt);
\begin{scope}[lightgray]
\draw[fill] (0,0) node (A1) {} circle(2pt)
(-1,0) node (A5) {} circle (2pt) 
(1,0) node (A8) {} circle (2pt);
\draw[<-] (A4) to[out=180,in=90] (A5) node[left] {$5$};
\draw[<-] (A5) to[out=-90,in=180] (A7);
\draw[<-] (A7) to[out=0,in=-90] (A8) node[right] {$8$};
\draw[<-] (A8) to[out=90,in=0] (A4);
\draw (A1) node [above right] {$1$} -- (A3);
\draw (A1) -- (A6);
\draw[->] (A3) -- (A5);
\draw[->] (A6) -- (A8);
\end{scope}
\draw[-Implies,double distance=2pt] (A4) node[above] {$4$} -- (A3) node[above 
left] {$3$};
\draw[-Implies,double distance=2pt] (A7) node[below] {$7$} -- (A6) node[below 
right] {$6$};
\end{scope}
\draw[->] (1.5,0) -- (2.5,0) node[midway, above] {$\mutation_5\mutation_8$};
\draw (6,0) node {$\succ$};
\end{tikzpicture}
\end{align*}

For the third case, we further reduce it to the quiver 
$\quiver|_{\{3,4,5,7\}}$, which will be reduced to a linear quiver of 
type~$(2,1)$.
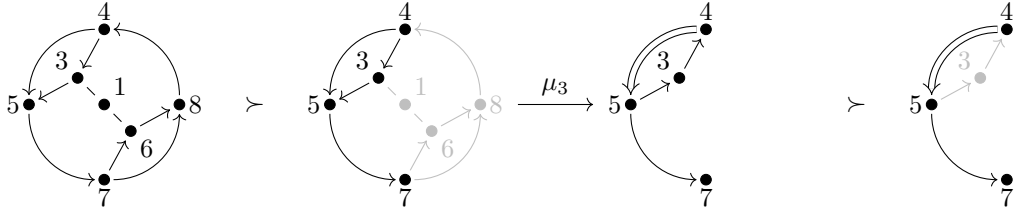
\[
\begin{tikzpicture}[baseline=-.5ex]
\begin{scope}
\draw[fill] (0,0) node (A1) {} circle(2pt)
(135:0.5) node (A3) {} circle (2pt)
(-45:0.5) node (A6) {} circle (2pt)
(0,1) node (A4) {} circle (2pt)
(-1,0) node (A5) {} circle (2pt) 
(0,-1) node (A7) {} circle (2pt) 
(1,0) node (A8) {} circle (2pt);
\draw[->] (A4) to[out=180,in=90] (A5) node[left] {$5$};
\draw[->] (A5) to[out=-90,in=180] (A7) node[below] {$7$};
\draw[->] (A7) to[out=0,in=-90] (A8) node[right] {$8$};
\draw[->] (A8) to[out=90,in=0] (A4) node[above] {$4$};
\draw (A1) node [above right] {$1$} -- (A3);
\draw (A1) -- (A6);
\draw[->] (A4) -- (A3) node[above left] {$3$};
\draw[->] (A3) -- (A5);
\draw[->] (A7) -- (A6) node[below right] {$6$};
\draw[->] (A6) -- (A8);
\end{scope}
\begin{scope}[xshift=4cm]
\draw[fill] (135:0.5) node (A3) {} circle (2pt)
(0,1) node (A4) {} circle (2pt)
(-1,0) node (A5) {} circle (2pt) 
(0,-1) node (A7) {} circle (2pt);
\begin{scope}[lightgray]
\draw[fill] (0,0) node (A1) {} circle(2pt)
(-45:0.5) node (A6) {} circle (2pt)
(1,0) node (A8) {} circle (2pt);
\draw[->] (A7) to[out=0,in=-90] (A8) node[right] {$8$};
\draw[->] (A8) to[out=90,in=0] (A4);
\draw (A1) node [above right] {$1$} -- (A3);
\draw (A1) -- (A6);
\draw[->] (A7) -- (A6) node[below right] {$6$};
\draw[->] (A6) -- (A8);
\end{scope}
\draw[->] (A4) node[above] {$4$} to[out=180,in=90] (A5) node[left] {$5$};
\draw[->] (A5) to[out=-90,in=180] (A7) node[below] {$7$};
\draw[->] (A4) -- (A3) node[above left] {$3$};
\draw[->] (A3) -- (A5);
\end{scope}
\begin{scope}[xshift=8cm]
\draw[fill] (135:0.5) node (A3) {} circle (2pt)
(0,1) node (A4) {} circle (2pt)
(-1,0) node (A5) {} circle (2pt) 
(0,-1) node (A7) {} circle (2pt);
\draw[-Implies,double distance=2pt] (A4) node[above] {$4$} to[out=180,in=90] 
(A5) node[left] {$5$};
\draw[->] (A5) to[out=-90,in=180] (A7) node[below] {$7$};
\draw[<-] (A4) -- (A3) node[above left] {$3$};
\draw[<-] (A3) -- (A5);
\end{scope}
\begin{scope}[xshift=12cm]
\draw[fill] (0,1) node (A4) {} circle (2pt)
(-1,0) node (A5) {} circle (2pt) 
(0,-1) node (A7) {} circle (2pt);
\begin{scope}[lightgray]
\draw[fill] (135:0.5) node (A3) {} circle (2pt);
\draw[<-] (A4) -- (A3) node[above left] {$3$};
\draw[<-] (A3) -- (A5);
\end{scope}
\draw[-Implies,double distance=2pt] (A4) node[above] {$4$} to[out=180,in=90] 
(A5) node[left] {$5$};
\draw[->] (A5) to[out=-90,in=180] (A7) node[below] {$7$};
\end{scope}
\draw (2,0) node {$\succ$};
\draw[->] (5.5,0) -- (6.5,0) node[midway, above] {$\mutation_3$};
\draw (10,0) node {$\succ$};
\end{tikzpicture}
\]

The fourth quiver is mutation equivalent to the quiver $\quiver'$
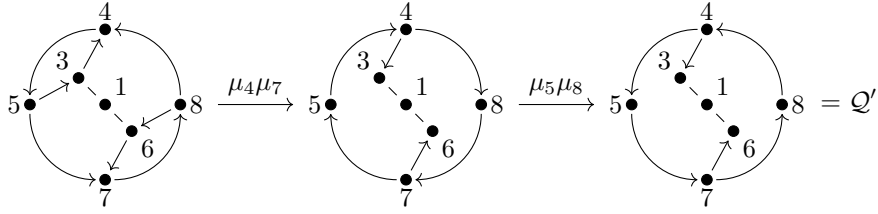
\[
\begin{tikzpicture}[baseline=-.5ex]
\begin{scope}
\draw[fill] (0,0) node (A1) {} circle(2pt)
(135:0.5) node (A3) {} circle (2pt)
(-45:0.5) node (A6) {} circle (2pt)
(0,1) node (A4) {} circle (2pt)
(-1,0) node (A5) {} circle (2pt) 
(0,-1) node (A7) {} circle (2pt) 
(1,0) node (A8) {} circle (2pt);
\draw[->] (A4) to[out=180,in=90] (A5) node[left] {$5$};
\draw[->] (A5) to[out=-90,in=180] (A7) node[below] {$7$};
\draw[->] (A7) to[out=0,in=-90] (A8) node[right] {$8$};
\draw[->] (A8) to[out=90,in=0] (A4) node[above] {$4$};
\draw (A1) node [above right] {$1$} -- (A3);
\draw (A1) -- (A6);
\draw[->] (A3) node[above left] {$3$} -- (A4);
\draw[->] (A5) -- (A3);
\draw[->] (A6) node[below right] {$6$} -- (A7);
\draw[->] (A8) -- (A6);
\end{scope}
\begin{scope}[xshift=4cm]
\draw[fill] (0,0) node (A1) {} circle(2pt)
(135:0.5) node (A3) {} circle (2pt)
(-45:0.5) node (A6) {} circle (2pt)
(0,1) node (A4) {} circle (2pt)
(-1,0) node (A5) {} circle (2pt) 
(0,-1) node (A7) {} circle (2pt) 
(1,0) node (A8) {} circle (2pt);
\draw[<-] (A4) to[out=180,in=90] (A5) node[left] {$5$};
\draw[<-] (A5) to[out=-90,in=180] (A7) node[below] {$7$};
\draw[<-] (A7) to[out=0,in=-90] (A8) node[right] {$8$};
\draw[<-] (A8) to[out=90,in=0] (A4) node[above] {$4$};
\draw (A1) node [above right] {$1$} -- (A3);
\draw (A1) -- (A6);
\draw[<-] (A3) node[above left] {$3$} -- (A4);
\draw[<-] (A6) node[below right] {$6$} -- (A7);
\end{scope}
\begin{scope}[xshift=8cm]
\draw[fill] (0,0) node (A1) {} circle(2pt)
(135:0.5) node (A3) {} circle (2pt)
(-45:0.5) node (A6) {} circle (2pt)
(0,1) node (A4) {} circle (2pt)
(-1,0) node (A5) {} circle (2pt) 
(0,-1) node (A7) {} circle (2pt) 
(1,0) node (A8) {} circle (2pt);
\draw[->] (A4) to[out=180,in=90] (A5) node[left] {$5$};
\draw[->] (A5) to[out=-90,in=180] (A7) node[below] {$7$};
\draw[->] (A7) to[out=0,in=-90] (A8) node[right] {$8$};
\draw[->] (A8) to[out=90,in=0] (A4) node[above] {$4$};
\draw (A1) node [above right] {$1$} -- (A3);
\draw (A1) -- (A6);
\draw[<-] (A3) node[above left] {$3$} -- (A4);
\draw[<-] (A6) node[below right] {$6$} -- (A7);
\end{scope}
\draw[->] (1.5,0) -- (2.5,0) node[midway, above] {$\mutation_4\mutation_7$};
\draw[->] (5.5,0) -- (6.5,0) node[midway, above] {$\mutation_5\mutation_8$};
\end{tikzpicture}=\quiver'
\]
and it reduces to the next case.

\smallskip
\noindent \textbf{Case (iv)} Suppose that $b_{3,4}\neq 0$ but 
$b_{3,5}=b_{3,7}=b_{3,8}=0$. 
Then it looks like the last quiver in the previous case. Then up to mutation 
$\mutation_1$, we may assume that $b_{1,3}b_{3,4}>0$. Namely, 
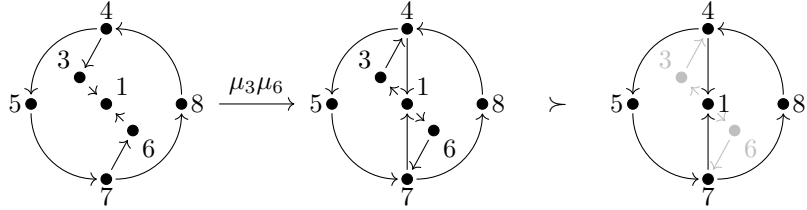
\[
\begin{tikzpicture}[baseline=-.5ex]
\begin{scope}
\draw[fill] (0,0) node (A1) {} circle(2pt)
(135:0.5) node (A3) {} circle (2pt)
(-45:0.5) node (A6) {} circle (2pt)
(0,1) node (A4) {} circle (2pt)
(-1,0) node (A5) {} circle (2pt) 
(0,-1) node (A7) {} circle (2pt) 
(1,0) node (A8) {} circle (2pt);
\draw[->] (A4) to[out=180,in=90] (A5) node[left] {$5$};
\draw[->] (A5) to[out=-90,in=180] (A7) node[below] {$7$};
\draw[->] (A7) to[out=0,in=-90] (A8) node[right] {$8$};
\draw[->] (A8) to[out=90,in=0] (A4) node[above] {$4$};
\draw[<-] (A1) node [above right] {$1$} -- (A3);
\draw[<-] (A1) -- (A6);
\draw[<-] (A3) node[above left] {$3$} -- (A4);
\draw[<-] (A6) node[below right] {$6$} -- (A7);
\end{scope}
\begin{scope}[xshift=4cm]
\draw[fill] (0,0) node (A1) {} circle(2pt)
(135:0.5) node (A3) {} circle (2pt)
(-45:0.5) node (A6) {} circle (2pt)
(0,1) node (A4) {} circle (2pt)
(-1,0) node (A5) {} circle (2pt) 
(0,-1) node (A7) {} circle (2pt) 
(1,0) node (A8) {} circle (2pt);
\draw[->] (A4) to[out=180,in=90] (A5) node[left] {$5$};
\draw[->] (A5) to[out=-90,in=180] (A7) node[below] {$7$};
\draw[->] (A7) to[out=0,in=-90] (A8) node[right] {$8$};
\draw[->] (A8) to[out=90,in=0] (A4) node[above] {$4$};
\draw[<-] (A1) -- (A4);
\draw[<-] (A1) -- (A7);
\draw[->] (A1) node [above right] {$1$} -- (A3);
\draw[->] (A1) -- (A6);
\draw[->] (A3) node[above left] {$3$} -- (A4);
\draw[->] (A6) node[below right] {$6$} -- (A7);
\end{scope}
\begin{scope}[xshift=8cm]
\draw[fill] (0,0) node (A1) {} circle(2pt)
(0,1) node (A4) {} circle (2pt)
(-1,0) node (A5) {} circle (2pt) 
(0,-1) node (A7) {} circle (2pt) 
(1,0) node (A8) {} circle (2pt);
\begin{scope}[lightgray]
\draw[fill] (135:0.5) node (A3) {} circle (2pt)
(-45:0.5) node (A6) {} circle (2pt);
\draw[->] (A1) -- (A3);
\draw[->] (A1) -- (A6);
\draw[->] (A3) node[above left] {$3$} -- (A4);
\draw[->] (A6) node[below right] {$6$} -- (A7);
\end{scope}
\draw[->] (A4) to[out=180,in=90] (A5) node[left] {$5$};
\draw[->] (A5) to[out=-90,in=180] (A7) node[below] {$7$};
\draw[->] (A7) to[out=0,in=-90] (A8) node[right] {$8$};
\draw[->] (A8) to[out=90,in=0] (A4) node[above] {$4$};
\draw[<-] (A1) -- (A4);
\draw[<-] (A1) node[right] {$1$} -- (A7);
\end{scope}
\draw[->] (1.5,0) -- (2.5,0) node[midway, above] {$\mutation_3\mutation_6$};
\draw (6,0) node {$\succ$};
\end{tikzpicture}
\]
Since the last quiver is reduced to $\quiver(2\exdynA_1)$ by \eqref{equation:square with 
one diagonal} and \eqref{equation:square with two diagonals}, this is a 
contradiction, which completes the proof.
\end{proof}

\begin{proof}[Proof of Theorem~\ref{thm_invaraint_implies_admissible} for 
$\dynX = \exdynE$]
For a $G$-invariant quiver $\quiver$ of type~$\exdynE_6$ or $\exdynE_7$ with 
$G=\Z/2\Z$ or $\Z/3\Z$, the condition~\ref{def_admissible_3} in 
Definition~\ref{definition:admissible quiver}(2) follows from 
Lemmas~\ref{lemma:condition_3_holds_of_order_2} and 
\ref{lemma:no monogons in affine E6 with Z/3Z-action}. 
The condition~\ref{def_admissible_4} follows from 
Lemmas~\ref{lemma:no bigons in affine E6 with Z/3Z-action}, 
\ref{lemma:no bigons in affine E6 with Z/2Z-action} and 
\ref{lemma:no bigons in affine E7 with Z/2Z-action}. Therefore 
$\quiver$ is $G$-admissible as claimed.
\end{proof}

\subsection{\texorpdfstring{Admissibility of quivers of 
type~$\exdynD$}{Admissibility of quivers of affine D}}
Throughout this section, we denote $G = \Z/2\Z, \Z/3\Z$, or $(\Z/2\Z)^2$ and 
$\dynY = \exdynD_n^G$. 
Let $\quiver$ be a $G$-invariant quiver on $[n+1]$ of 
type~$\exdynD_{n}$ and $\qbasis=\qbasis(\quiver)$.  
\begin{lemma}\label{lemma:no monogons in affine D}
For each $i\in[n+1]$ and $g\in G$, we obtain
\[
b_{i,g(i)}=0.
\]
\end{lemma}
\begin{proof}
Suppose that $G=\Z/2\Z$ or $\Z/2\Z\times\Z/2\Z$. Then we prove the claim by 
Lemma~\ref{lemma:condition_3_holds_of_order_2}.
Let $G=\Z/3\Z$ which acts on $\exdynD_4$ and identifies $\{3,4,5\}$ 
(see~\ref{fig_D4_D4_3} in 
Appendix~\ref{appendix_actions_on_Dynkin_diagrams}). Since 
$\quiver$ is connected, $b_{i,3}=b_{i,4}=b_{i,5}=b\neq 0$ for some $i\le 2$.
Then every pair of distinct vertices in $\quiver|_{\{i,3,4,5\}}$ is connected
and it contains a cyclic triangle $\quiver|_{\{3,4,5\}}$ of type~$(b,b,b)$, 
which is mutation-infinite by Corollary~\ref{corollary:complete 4-graph}. 
This yields a contradiction and we are done.
\end{proof}
Hence, to prove the admissibility, it is sufficient to show that the 
condition~\ref{def_admissible_4} holds, that is,  
$b_{i,j}b_{i,\tau(j)}\ge 0$ for all $i,j\in[n+1]$.

\subsubsection{\texorpdfstring{$(\dynX,G,\dynY)=(\exdynD_4, \Z/2\Z\times\Z/2\Z, \dynA_2^{(2)})$}{(X, G, Y)=(affine D4, Z/2Z x Z/2Z, affine A2 2)} 
or \texorpdfstring{$(\exdynD_4, \Z/3\Z, \dynD_4^{(3)})$}{(affine D4, Z/3Z, affine D4 3)}
}
\label{section_D4_D43}

\begin{lemma}\label{lemma:no bigons in affine D4}
For each $i,j\in[5]$, we have
\[
b_{i,j}b_{i,\tau(j)}\ge 0.
\]
\end{lemma}
\begin{proof}
Suppose that $G=\Z/2\Z\times \Z/2\Z$ and $\exdynD_4^G = \dynA_2^{(2)}$. Then we 
only need to prove that 
\[
b_{1,j}b_{1,\tau(j)}\ge 0
\]
for $2\le j$, which is obvious since 
$b_{1,j}=b_{\tau(1),\tau(j)}=b_{1,\tau(j)}$.

For $G=\Z/3\Z$, we have $\exdynD_4^G = \dynD_4^{(3)}$. If $i\le 2$ or $j\le 2$, 
then there is nothing to prove. Otherwise, $b_{i,j}=0$ by Lemma~\ref{lemma:no 
monogons in affine D} and we are done.
\end{proof}

\subsubsection{\texorpdfstring{$(\dynX,G,\dynY)=(\exdynD_n, \Z/2\Z, \exdynC_{n-2})$}{(X, G, Y)=(affine Dn, Z/2Z, affine C n-2)}}
\label{section_Dn_Cn-2}

The action of the generator $\tau\in\Z/2\Z$ is as follows:
\begin{align*}
\tau(i) =\begin{cases}
i & \text{ if }i=1, 4,\dots, n-1;\\
3 &\text{ if } i=2;\\
2 & \text{ if }i=3;\\
n+1 &\text{ if } i=n;\\
n & \text{ if }i=n+1.
\end{cases}
\end{align*}
See~\ref{fig_Dn_Cn-2} in Appendix~\ref{appendix_actions_on_Dynkin_diagrams}. 

\begin{lemma}\label{lemma:no bigons in affine Dn yielding twisted affine Cn-2}
For each $i,j\in[n+1]$, we have
\[
b_{i,j}b_{i,\tau(j)}\ge 0.
\]
\end{lemma}
\begin{proof}
If $i\not\in\{2,3,n,n+1\}$ or $j\not\in\{2,3,n,n+1\}$, then there is nothing to 
prove since $\tau(i)=i$ or $\tau(j)=j$ and  
\[
b_{i,j}=b_{\tau(i),\tau(j)} = 
\begin{cases}
b_{i,\tau(j)} & \text{ if } i\not\in\{2,3,n,n+1\};\\
b_{\tau(i),j} = b_{\tau^2(i),\tau(j)} = b_{i,\tau(j)} & \text{ if } j\not\in\{2,3,n,n+1\}.
\end{cases}
\]

If $i,j$ are in the same $\Z/2\Z$-orbit, namely, either $\{i,j\}=\{2,3\}$ or 
$\{i,j\}=\{n,n+1\}$, then we are done since $b_{i,j}=0$ by  Lemma~\ref{lemma:no 
monogons in affine D}.

Finally, suppose that $i\in\{2,3\}$, $j\in\{n,n+1\}$ and 
$b_{i,j}b_{i,\tau(j)}<0$. Then we may assume that $i=2$, $j=n$ and
\[
b_{3,n+1}=b_{2,n}<b_{2,n+1}=b_{3,n}
\]
and therefore the restriction $\quiver|_{\{2,3,n,n+1\}}$ is a directed cycle of 
length $4$. On the other hand, since $\quiver$ is connected, there exists 
$\ell\in[n+1]\setminus\{2,3,n,n+1\}$ such that 
\[
b_{\ell,2}=b_{\ell,3}\neq 0\quad\text{ or }\quad
b_{\ell,n}=b_{\ell,n+1}\neq 0.
\]
Then up to $\mutation_\ell$, the restriction $\quiver|_{\{2,3,n,n+1,\ell\}}$ 
looks like one of three quivers depicted in \eqref{equation:squares}. Indeed, 
as seen in \eqref{equation:square with one diagonal}, \eqref{equation:square 
with two diagonals} and \eqref{equation:square with two diagonals 2}, the 
quiver $\quiver|_{\{2,3,n,n+1,\ell\}}$ eventually reduces to a 
mutation-infinite quiver or the quiver $2\exdynA_1$. This contradicts to 
Corollary~\ref{corollary:seed restriction}.
\end{proof}

\subsubsection{\texorpdfstring{$(\dynX,G,\dynY)=(\exdynD_n, \Z/2\Z, \dynA_{2(n-1)-1}^{(2)})$}{(X, G, Y)=(affine Dn, Z/2Z, affine A 2(n-1)-1 2)}}
\label{section_Dn_A2n-1_2}

The action of the generator $\tau\in\Z/2\Z$ is as follows:
\begin{align*}
\tau(i) =\begin{cases}
i & \text{ if } i\le n-1;\\
n+1 & \text{ if } i=n;\\
n & \text{ if } i=n+1.
\end{cases}
\end{align*}
See~\ref{fig_Dn_A2n-1_2} in Appendix~\ref{appendix_actions_on_Dynkin_diagrams}.

\begin{lemma}\label{lemma:no bigons in affine Dn yielding twisted affine A2n-1}
For each $i,j\in[n+1]$, we have
\[
b_{i,j}b_{i,\tau(j)}\ge 0.
\]
\end{lemma}
\begin{proof}
If $i\le n-1$ or $j\le n-1$, then there is nothing to prove by the same 
argument as Lemma~\ref{lemma:no bigons in affine Dn yielding twisted affine 
Cn-2}. 
Otherwise, $b_{i,j}=0$ by Lemma~\ref{lemma:no monogons in affine D} and we are 
done.
\end{proof}

\subsubsection{\texorpdfstring{$(\dynX,G,\dynY)=(\exdynD_{2n}, \Z/2\Z, \exdynB_n)$}{(X, G, Y)=(affine D2n, Z/2Z, affine Bn)}}
\label{section_D2n_Bn}

The action of the generator $\tau\in\Z/2\Z$ is as follows:
\begin{align*}
\tau(i) =\begin{cases}
2n-1 & \text{ if } i=1;\\
2n+1 & \text{ if }i=2;\\
2n & \text{ if }i=3;\\
2n+2-i & \text{ if }4\le i\le 2n-2;\\
1 & \text{ if }i = 2n-1; \\
3 & \text{ if }i=2n;\\
2 & \text{ if }i=2n+1.
\end{cases}
\end{align*} 
See~\ref{fig_D2n_Bn} in Appendix~\ref{appendix_actions_on_Dynkin_diagrams}. 

\begin{lemma}\label{lemma:no bigons in affine D2n yielding twisted affine Bn}
For each $i,j\in[2n+1]$, we have
\[
b_{i,j}b_{i,\tau(j)}\ge 0.
\]
\end{lemma}
\begin{proof}
If $i=n+1$ or $j=n+1$, then there is nothing to prove since
\[
b_{n+1,j} = b_{\tau(n+1),\tau(j)} = b_{n+1, \tau(j)}\quad\text{ and }\quad
b_{i,n+1} = b_{i, \tau(n+1)}.
\]
Suppose that $b_{i,j}b_{i,\tau(j)}<0$ for $i,j\in[2n+1]\setminus\{n+1\}$. Then 
we may assume that $i,j<n+1$ and 
\[
b_{\tau(i),j}=b_{i,\tau(j)}<0<b_{i,j}=b_{\tau(i),\tau(j)}.
\]
Therefore the restriction $\quiver|_{\{i,j,\tau(i),\tau(j)\}}$ is a directed 
cycle of length 4. We furthermore assume that this cycle is the closest one to 
the vertex $n+1$ with respect to the length of undirected edge-path. In other 
words, for any vertex $p$ closer than $i$ and $j$ from $n+1$, we have 
$b_{p,q}b_{p,\tau(q)}\ge0$ for all $q\in[2n+1]$.

On the other hand, there is a sequence of mutations
\[
\quiver = 
(\mutation_{j_L}^{\exdynD_{2n}}\cdots\mutation_{j_1}^{\exdynD_{2n}})(\quiver(\exdynD_{2n})),
\]
where the sequence $j_1,\dots, j_L$ misses at least one vertex $\ell\in[2n+1]$ 
and so $\tau(\ell)$ as well by Lemma~\ref{lemma_on_facets}. 

We consider subquivers separated by $\ell,\tau(\ell)$ and observe 
that in $\quiver$, there are no edges between pieces separated by 
$\{\ell,\tau(\ell)\}$. Suppose that $\ell=n+1$. Then either $i,j$ or 
$i,\tau(j)$ are contained in different subquivers separated by $n+1$. However 
since we are assuming $b_{i,j}b_{i,\tau(j)}<0$, this is a contradiction and so 
we may assume that $\ell\neq n+1$. Then among separated quivers, there exists a 
central piece $\quiver'$ containing $n+1$, which is $\Z/2\Z$-invariant. 
Moreover, $\quiver'$ is of type~$\dynA_m$ for some $m<2n$.

\smallskip
\noindent\textbf{Claim (i)} We claim that $\ell\in\{i,j,\tau(i),\tau(j)\}$. 
Assume on the contrary that $\ell\not\in\{i,j,\tau(i), \tau(j)\}$.
Then, by the observation above, four vertices $i,j,\tau(i), \tau(j)$ are 
contained in the central piece $\quiver'$. This yields a contradiction since 
any $\Z/2\Z$-invariant quiver of type~$\dynA_m$ is $\Z/2\Z$-admissible. 
Therefore, we get $\ell\in\{i,j,\tau(i),\tau(j)\}$.
\smallskip

\noindent \textbf{Claim (ii)} We claim that $\ell\in\{2,3,2n,2n+1\}$. 
Assume on the contrary that $\ell\not\in\{2,3,2n,2n+1\}$. 
Then the restriction 
$\quiver''$ containing vertices of $\quiver'$ and $\{\ell,\tau(\ell)\}$ is 
$\Z/2\Z$-admissible of type~$\dynA_{2m+2}$. Then by the observation above, 
$i,j,\tau(i),\tau(j)$ are contained in $\quiver''$. This yields a contradiction 
again and therefore we obtain $\ell\in\{2,3,2n,2n+1\}$.

\smallskip

Because of the above two claims (i) and (ii), we may assume that 
\[
\ell\in\{i,j,\tau(i),\tau(j)\}\cap\{2,3,2n,2n+1\}.
\]
Let $i_0=n+1,i_1,\dots,i_{M+1}$ be a sequence of vertices which gives us a 
shortest (undirected) path from $n+1$ to the cycle $\quiver|_{\{i,j,\tau(i),\tau(j)\}}$. 
That is, $i_{M+1}\in\{i,j,\tau(i),\tau(j)\}$. Then we also have another 
shortest (undirected) path given by a sequence of vertices $\tau(i_{M+1}),\dots,\tau(i_1), 
n+1$ from the cycle $\quiver|_{\{i,j,\tau(i),\tau(j)\}}$ to the vertex $n+1$.

Notice that the set 
\[
R=\{\tau(i_M),\dots, \tau(i_1), n+1, i_1,\dots, i_M\}\subset[2n+1]
\]
misses $\{i,j,\tau(i),\tau(j)\}$.

We claim that the restriction $\quiver|_{R}$ is a $\Z/2\Z$-admissible quiver of 
type $\dynA_{2M+1}$. Indeed, $\quiver|_R$ is a restriction of 
$\quiver|_{[2n+1]\setminus\{\ell,\tau(\ell)\}}$, which is of 
type~$\dynA_{2n-1}$. Since any connected subquiver of a quiver mutation 
equivalent 
to $\dynA$ is again of type~$\dynA$, we proved the claim.

Because we take the shortest undirected path connecting $n+1$ and $i$, the restriction $\quiver|_{R}$ is an undirected path having $2M+1$ vertices, that is, the underlying graph of $\quiver|_{R}$ is isomorphic to the Dynkin diagram of type $\dynA_{2M+1}$.

Now consider the quiver $\quiver|_{R\cup\{i,j,\tau(i),\tau(j)\}}$
which is  $\Z/2\Z$-invariant and looks like the left picture 
below. 
Finally, by a 
sequence of orbit mutations, the vertex $n+1$ can be directly connected with both $i$ and $\tau(i)$, so $\quiver|_{R\cup\{i,j,\tau(i),\tau(j)\}}$ can be reduced to one of the quivers in~\eqref{equation:squares} 
as displayed in the right picture below.
This contradiction completes the proof.
\[
\begin{tikzpicture}[baseline=-.5ex,scale=2]
\begin{scope}
\draw[fill] (0,0) node (A1) {} circle(1pt)
(135:0.5) node (A3) {} circle (1pt)
(-45:0.5) node (A6) {} circle (1pt)
(0,1) node (A4) {} circle (1pt)
(-1,0) node (A5) {} circle (1pt) 
(0,-1) node (A7) {} circle (1pt) 
(1,0) node (A8) {} circle (1pt);
\draw[->] (A4) to[out=180,in=90] (A5) node[left] {$j$};
\draw[->] (A5) to[out=-90,in=180] (A7) node[below] {$\tau(i)$};
\draw[->] (A7) to[out=0,in=-90] (A8) node[right] {$\tau(j)$};
\draw[->] (A8) to[out=90,in=0] (A4) node[above] {$i$};
\draw[dashed] (A1) node [above right] {$n+1$} -- (A3);
\draw[dashed] (A1) -- (A6);
\draw[-] (A3) node[above left] {$i_M$} -- (A4);
\draw[-] (A6) node[below right] {$\tau(i_M)$} -- (A7);
\end{scope}
\begin{scope}[xshift=4cm]
\draw[fill, lightgray]
(135:0.5) node (A3) {} circle (1pt)
(-45:0.5) node (A6) {} circle (1pt);
\draw[fill] (0,0) node (A1) {} circle(1pt)
(0,1) node (A4) {} circle (1pt)
(-1,0) node (A5) {} circle (1pt) 
(0,-1) node (A7) {} circle (1pt) 
(1,0) node (A8) {} circle (1pt);
\draw[->] (A4) to[out=180,in=90] (A5) node[left] {$j$};
\draw[->] (A5) to[out=-90,in=180] (A7) node[below] {$\tau(i)$};
\draw[->] (A7) to[out=0,in=-90] (A8) node[right] {$\tau(j)$};
\draw[->] (A8) to[out=90,in=0] (A4) node[above] {$i$};
\draw[dotted, lightgray] (A3) node[above left] {$i_M$};
\draw[dotted, lightgray] (A6) node[below right] {$\tau(i_M)$};
\draw[->] (A1) node [above right] {$n+1$} -- (A4);
\draw[->] (A1) -- (A7);
\end{scope}
\end{tikzpicture}\qedhere
\]
\end{proof}

\subsubsection{\texorpdfstring{$(\dynX,G,\dynY)=(\exdynD_{2n}, \Z/2\Z\times\Z/2\Z, \dynA_{2n-2}^{(2)})$}{(X, G, Y)=(affine D2n, Z/2Z x Z/2Z, affine A 2n-2 2)}}
\label{section_D2n_A2n-2_2}

The actions of two generators $\tau_1, \tau_2\in\Z/2\Z$ are as follows:
\begin{align*}
\tau_1(i) &=\begin{cases}
i & \text{ if }i\le n-1;\\
n+1 & \text{ if }i=n;\\
n & \text{ if }i=n+1,
\end{cases}&
\tau_2(i) &=\begin{cases}
i & \text{ if }i=1,4,\dots, 2n-1;\\
3 & \text{ if }i=2;\\
2 & \text{ if }i=3;\\
2n+1 & \text{ if }i=2n;\\
2n & \text{ if }i=2n+1.
\end{cases}
\end{align*}
See~\ref{fig_D2n_A2n-2_2} in 
Appendix~\ref{appendix_actions_on_Dynkin_diagrams}.

\begin{lemma}\label{lemma:no bigons in affine D2n yielding twisted affine A2n-2}
For each $i,j\in[2n+1]$ and $g\in\Z/2\Z\times\Z/2\Z$, we have
\[
b_{i,j}b_{i,g(j)}\ge 0.
\]
\end{lemma}
\begin{proof}
Since $\quiver$ is already $\Z/2\Z=\langle \tau_1\rangle$-invariant and 
$\tau_1$ and $\tau_1\tau_2$ generate isomorphic actions on $\exdynD_{2n}$, the 
only thing to check is for $g=\tau_2$ by Lemma~\ref{lemma:no bigons in affine 
D2n yielding twisted affine Bn}.

If $i\not\in\{2,3,2n,2n+1\}$ or $j\not\in\{2,3,2n,2n+1\}$, then there is 
nothing to prove since 
\[
b_{i,j}=b_{\tau_2(i),\tau_2(j)} = 
\begin{cases}
b_{i,\tau_2(j)} & \text{ if }i\not\in\{2,3,2n,2n+1\};\\
b_{\tau_2(i),j} = b_{\tau_2^2(i),\tau_2(j)} = b_{i,\tau_2(j)} & 
\text{ if }j\not\in\{2,3,2n,2n+1\}.
\end{cases}
\]
Otherwise, if $i,j\in\{2,3,2n,2n+1\}$, then $b_{i,j}=0$ by Lemma~\ref{lemma:no 
monogons in affine D} and we are done.
\end{proof}

\begin{proof}[Proof of Theorem~\ref{thm_invaraint_implies_admissible} for 
$\dynX = \exdynD$]
For a $G$-invariant quiver $\quiver$ of type~$\exdynD_n$ with $G=\Z/2\Z,\Z/3\Z$ 
or $\Z/2\Z\times\Z/2\Z$, the condition \ref{def_admissible_3} in 
Definition~\ref{definition:admissible quiver}(2) follows from 
Lemma~\ref{lemma:no monogons in affine D} and the condition 
\ref{def_admissible_4} follows from 
Lemmas~\ref{lemma:no bigons in affine D4},
\ref{lemma:no bigons in affine Dn yielding twisted affine Cn-2},
\ref{lemma:no bigons in affine Dn yielding twisted affine A2n-1},
\ref{lemma:no bigons in affine D2n yielding twisted affine Bn}, and
\ref{lemma:no bigons in affine D2n yielding twisted affine A2n-2}.
Therefore $\quiver$ is $G$-admissible as 
claimed.
\end{proof}

\section{Connections with cluster algebras: folded cluster patterns}
\label{section_connections_with_CA}
Under certain conditions, one can \textit{fold} cluster patterns to produce new 
ones. This procedure is used to study cluster algebras of non-simply-laced 
affine type from those of simply-laced affine type (see Table~\ref{table:all 
possible foldings}). 
In this section, we observe the properties of folded cluster patterns of 
non-simply-laced affine type in 
Corollary~\ref{cor_invariant_seeds_form_folded_pattern}.

For a globally foldable quiver $\quiver$ on $[n]$ with respect to $G$-action, we can fold all the seeds in the corresponding 
cluster pattern.
Let $\field^G$ be the field of rational functions in $\#([n]/G)$ independent 
variables and $\psi \colon \field \to \field^G$ be a surjective homomorphism. 
A seed $\seed = (\mathbf{x}, \quiver)$ is called \emph{$(G, 
\psi)$-invariant} (respectively, \emph{$(G,\psi)$-admissible}) if 
\begin{itemize}
\item for any $i \sim i'$, we have $\psi(x_i) = \psi(x_{i'})$;
\item $\quiver$ is $G$-invariant (respectively, $G$-admissible).
\end{itemize}
In this situation, we define a new ``folded'' seed $\seed^G = (\bfx^G, 
(\qbasis(\quiver))^G)$ in $\field^G$ whose exchange matrix is given as before 
and cluster 
variables $\bfx^G = (x_I)$ are indexed by the $G$-orbits and given by $x_I = 
\psi(x_i)$.

\begin{proposition}[{\cite[Corollary~4.4.11]{FWZ_chapter45}}]\label{proposition:folded cluster pattern}
Let $\quiver$ be a quiver on $[n]$ which is globally foldable with respect to a group 
$G$ acting on $[n]$. Let $\initialseed = (\mathbf{x}, 
\quiver)$ be a seed in the field $\field$ of rational functions freely 
generated by a cluster $\mathbf{x} = (x_1,\dots,x_n)$. Define $\psi \colon 
\field  \to \field^G$ so that $\initialseed$ is a $(G, \psi)$-admissible seed. 
Then, for any $G$-orbits $I_1,\dots,I_\ell$, the seed 
$(\mutation_{I_\ell} \dots \mutation_{I_1})(\initialseed)$ is $(G, 
\psi)$-admissible, and moreover, the folded seeds $((\mutation_{I_\ell} \dots 
\mutation_{I_1})(\initialseed))^G$ form a cluster pattern in $\field^G$ with the 
initial seed $\initialseed^G=(\bfx^G, (\qbasis(\quiver))^G)$.
\end{proposition}	

\begin{proposition}\label{prop_admissible_seeds_form_folded_seed_pattern}
Let $\quiver$ be an acyclic quiver on $[n]$ which is globally foldable with respect to a 
group $G$ acting on $[n]$. Let $\initialseed = (\mathbf{x}, 
\quiver)$ be a seed in the field $\field$ of rational functions freely 
generated by a cluster $\mathbf{x} = (x_1,\dots,x_n)$. Define $\psi \colon 
\field  \to \field^G$ so that $\initialseed$ is a $(G, \psi)$-admissible seed. 
Then, the set of $(G,\psi)$-admissible seeds are connected via orbit mutations. 
Indeed, the set of $(G,\psi)$-admissible seeds forms a cluster pattern with the 
initial seed $\initialseed^G$.
\end{proposition}
\begin{proof}
We denote by $\mathcal{S}$ the set of $(G,\psi)$-admissible seeds in the cluster pattern obtained by $\initialseed$. Consider a subset $\mathcal{S}'$ of $\mathcal{S}$ such that each of which element is connected to the initial seed $\initialseed$ via a sequence of orbit mutations. 

To show $\mathcal{S}' = \mathcal{S}$, it is enough to prove that for each seed 
$\seed'$ in the cluster pattern of $\initialseed^G$, there exists only one 
$(G,\psi)$-admissible seed $\seed_t$ in $\mathcal{S}$ such that $\seed_t^G  = 
\seed'$. Take a seed $\seed_{t}$ in $\mathcal{S}'$ such that $\seed_{t}^G 
=\seed'$. We may assume that $\seed_t$ is the initial seed in the cluster pattern, 
that is, $\seed_t$ has cluster variables $\{x_i \mid i \in [n]\}$. Assume on 
the contrary that there exists another $(G,\psi)$-admissible seed $\seed_{s} = 
(\mathbf{x}_{s}, \qbasis_{s})$ satisfies $\seed_{s}^G = \seed'$, then 
\begin{equation}\label{eq_xi_and_xi'}
\{\psi(x_i) \mid i \in [n] \} = \{\psi(x_{i;s}) \mid i \in [n]\}.
\end{equation} 

The Positivity of Laurent phenomenon, which was conjectured 
in Fomin--Zelevinsky~\cite{FZ1_2002} and proved in Gross--Hacking--Keel--Kontsevich~\cite[Corollary~4.4.11]{GHKK18}, states that every non-zero cluster 
variable can be uniquely written by a rational polynomial whose numerator is a 
polynomial with \emph{non-negative} integer coefficients in the initial cluster 
variables $x_1,\dots,x_n$. 
Accordingly, to get~\eqref{eq_xi_and_xi'}, the 
cluster variables $x_{i;s},\dots,x_{n;s}$ should be the initial cluster 
variables $x_1,\dots,x_n$ because the non-negativity of coefficients means no 
cancellation exists. 
Since we are considering a cluster pattern whose initial seed has an acyclic 
quiver, the cluster variables determine a seed by~\cite[Theorem~4.1]{BMRT07} 
(also, see~\cite[Conjecture~4.14]{FZ03_CDM}), so we have $\seed_s = \seed_t$  
and this proves the claim. 
\end{proof}

If a seed $\seed = 
(\mathbf{x}, \quiver)$ is $(G,\psi)$-admissible, then $\seed$ is 
$(G,\psi)$-invariant by Definition~\ref{definition:admissible quiver}.
As a direct corollary of Theorem~\ref{thm_invaraint_implies_admissible} and  
Proposition~\ref{prop_admissible_seeds_form_folded_seed_pattern}, 
we obtain that the converse 
holds when we consider the foldings presented in Table~\ref{table:all possible 
foldings}.
\begin{corollary}\label{cor_invariant_seeds_form_folded_pattern}
Let $(\dynX,G,\dynY)$ be a triple given by a column of Table~\ref{table:all 
possible foldings}.
Let $\initialseed = (\mathbf x, \quiver)$ be a seed. Suppose 
that $\quiver$ is of type~$\dynX$. Define $\psi \colon \field  \to \field^G$ so 
that $\initialseed$ is a $(G, \psi)$-admissible seed. Then, any 
$(G,\psi)$-invariant seed can be reached 
by a sequence of orbit mutations from $\initialseed$. Moreover, the set of $(G,\psi)$-invariant seeds forms the `folded' cluster pattern given by $\initialseed^G$ of $\dynY$ via folding. 
\end{corollary}

\begin{remark}\label{rmk_Dupont}
Let $\quiver$ be an acyclic quiver which is globally foldable with respect to a 
finite group~$G$. Define $\psi \colon \field \to \field^G$ so that 
$\initialseed = (\mathbf x_{t_0}, \quiver)$ is $(G,\psi)$-admissible. Let 
$\seed_t = (\mathbf x_t, \quiver_t)$ be a $(G,\psi)$-invariant seed.
Dupont asked in~\cite[Problem~9.5]{Dupont08} that \textit{can $\seed_t$ be 
reached by sequences of orbit mutations from the initial seed 
$\initialseed$?} Corollary~\ref{cor_invariant_seeds_form_folded_pattern} 
implies that when the quiver $\quiver$ is of type~$\exdynA_{n,n},\exdynD_{n}, 
\exdynE_6$ or $\exdynE_7$ and with the specific choice of~$G$ (as in 
Table~\ref{table:all possible foldings}), we get an affirmative answer to the 
question proposed by Dupont.
\end{remark}

\appendix
\section{Group actions on Dynkin diagrams of affine type}
\label{appendix_actions_on_Dynkin_diagrams}
In the appendix, we provide group actions on Dynkin diagrams of affine type. 
More precisely, for each triple $(\dynX, G, \dynY)$ given by a column of 
Table~\ref{table:all possible foldings}, we describe the $G$-action of the 
Dynkin diagram of $\dynX$. Throughout this section, we denote by $\tau$ the 
generator of each finite group $\Z/2\Z$ or $\Z/3\Z$. 
For $i \in [n]$, we denote by $I_i$ the orbit $G \cdot i$.
We decorate vertices of Dynkin diagram of type~$\dynY$ with orbits $I_i$.
\begin{enumerate}[ref = (\arabic*)]
\item $(\dynX, G, \dynY) = (\exdynA_{2,2}, \Z/2\Z, \exdynA_1)$ \label{fig_A22_A1}
\begin{figure}[H]
\begin{minipage}[b]{.45\textwidth}
\centering
\begin{tikzpicture}[baseline=-.5ex]
\foreach \x in {1,...,4}{
\draw[fill] (90*\x:1) coordinate (A\x) {} circle (2pt);
}

\draw (A1) node[above] {$2$} 
(A2) node[left] {$1$}
(A3) node[below] {$3$}
(A4) node[right] {$4$};

\draw (A1) to (A2);
\draw (A2) to (A3);
\draw (A4) to (A3);
\draw (A1) to (A4);

\node[left = 0.5cm of A2] (L) {};
\node[right = 0.5cm of A4] (R) {};
\draw[dotted, blue, thick] (L) -- (R);
\node[blue] at (L) {$\tau \updownarrow$};

\node[above = 0.5cm of A1] (Above) {};
\node[below = 0.5cm of A3] (Below) {};
\draw[dotted, blue, thick] (Above) -- (Below);
\node[blue] at (Above) {$\leftrightarrow$};
\node[blue, above = 0.1cm of Above] {$\tau$};

\end{tikzpicture}
\end{minipage}
$\rightsquigarrow$
\begin{minipage}[b]{0.45\textwidth}
\centering
\begin{tikzpicture}[baseline=-.5ex]
\tikzstyle{state}=[draw, circle, inner sep = 1.4pt, fill]
\node[state, label = below:$I_1$] (1) {};
\node[state, label = below:$I_2$] (2) [right = of 1] {};

\draw[double line] (1)-- node[pos=0.2]{\scalebox{1}{ $<$}} 
node[pos=0.7]{\scalebox{1}{ $>$}} (2);

\end{tikzpicture}
\end{minipage}
\end{figure}

\item $(\dynX, G, \dynY) = (\exdynA_{n,n}, \Z/2\Z, \dynD_{n+1}^{(2)})$ \label{fig_Ann_Dn+1_2}
\begin{figure}[H]
\begin{minipage}[b]{.45\textwidth}
\centering
\begin{tikzpicture}[baseline=-.5ex]
\draw[fill] 
(-2.3,0) circle(2pt) coordinate (A1) {}
(-1.5,1) circle(2pt) coordinate (A2) {}
(0,1) circle(2pt) coordinate (A3) {}
(1.5,1) circle(2pt) coordinate (An) {}
(-1.5,-1) circle(2pt) coordinate (An+1) {}
(0,-1) circle(2pt) coordinate (An+2) {};
\draw[fill] 
(1.5,-1) circle(2pt) coordinate (A2n-1) {}
(2.3,0) circle(2pt) coordinate (A2n) {};
\draw (A1) node[above] {$1$} -- (A2);
\draw (A1) -- (An+1);
\draw[dashed] (A3) node[above] {$i$} -- (A2) node[above] {$2$};
\draw[dashed] (An+2) node[below] {$n+i-1$} -- (An+1) node[below] {$n+1$};
\draw[dashed] (A3) -- (An) node[above] {$n$};
\draw[dashed] (An+2) -- (A2n-1) node[below] {$2n-1$};
\draw (An) -- (A2n) node[above] {$2n$};
\draw (A2n-1) -- (A2n);

\node[left = 0.5cm of A1] (L) {};
\node[right = 0.5cm of A2n] (R) {};

\draw[dotted, blue, thick] (L) -- (R);
\node[blue] at (L) {$\tau \updownarrow$};
\end{tikzpicture}
\end{minipage} 
$\rightsquigarrow$
\begin{minipage}[b]{0.45\textwidth}
\centering
\begin{tikzpicture}[baseline=-.5ex]
\tikzstyle{state}=[draw, circle, inner sep = 1.4pt, fill]

\node[state, label=below:{$I_2$}] (1) {};
\node[state, label=below:{$I_3$}] (2) [right = of 1] {};
\node[state, label=below:{$I_{n-1}$}] (3) [right = of 2] {};
\node[state, label=below:{$I_n$}] (4) [right =of 3] {};
\node[state, label=below:{$I_{2n}$}] (5) [right =of 4] {};			
\node[state, label=below:{$I_1$}] (6) [left =of 1] {};
\draw (1)--(2)
(3)--(4);
\draw [dotted] (2)--(3);
\draw[double line] (4)-- node{\scalebox{1}{ $>$}} (5);
\draw[double line] (1)-- node{\scalebox{1}{ $<$}} (6);

\end{tikzpicture} 
\end{minipage}
\end{figure}

\item $(\dynX, G, \dynY) = (\exdynD_4, (\Z/2\Z)^2, \dynA_2^{(2)})$
\begin{figure}[H]
\begin{minipage}[b]{0.45\textwidth}
\centering
\begin{tikzpicture}[baseline=-.5ex]
\draw[fill] (0,0) coordinate (A1) {} circle (2pt)
(0,1) coordinate (A2) {} circle (2pt)
(-1,0) coordinate (A3) {} circle (2pt) 
(0,-1) coordinate (A4) {} circle (2pt) 
(1,0) coordinate (A5) {} circle (2pt);
\draw (A1) node[above right] {$1$} -- (A4) node[below] {$4$};
\draw (A1) -- (A2) node[above] {$2$};
\draw (A1) -- (A3) node[left] {$3$};
\draw (A1) -- (A5) node[right] {$5$};

\draw[->, dotted, thick, blue] (5:1) arc (0:85:1);
\draw[->, dotted, thick, blue] (95:1) arc (95:175:1);
\draw[->, dotted, thick, blue] (185:1) arc (185:265:1);
\draw[->, dotted, thick, blue] (275:1) arc (275:355:1);
\end{tikzpicture}
\end{minipage}
$\rightsquigarrow$
\begin{minipage}[b]{0.45\textwidth}
\centering
\begin{tikzpicture}[baseline=-.5ex]
\tikzstyle{state}=[draw, circle, inner sep = 1.4pt, fill]

\node[state, label=below:{$I_1$}] (1) {};
\node[state, label=below:{$I_2$}] (2) [right = of 1] {};

\draw[double distance = 2.7pt] (1)--(2);
\draw[double distance = 0.9pt] (1)-- node{\scalebox{1}{ $<$}} (2);
\end{tikzpicture}
\end{minipage}
\end{figure}
\item $(\dynX, G, \dynY) = (\exdynD_4, \Z/3\Z, 
\dynD_4^{(3)})$
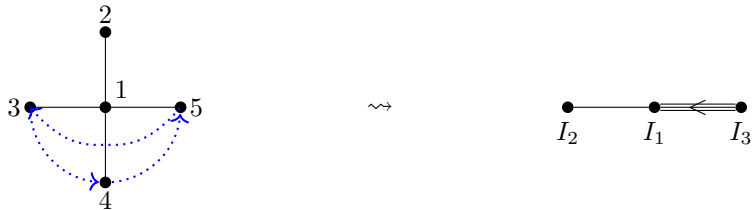
\label{fig_D4_D4_3}
\begin{figure}[H]
\begin{minipage}[b]{0.45\textwidth}
\centering
\begin{tikzpicture}[baseline=-.5ex]
\draw[fill] (0,0) coordinate (A1) {} circle (2pt)
(0,1) coordinate (A2) {} circle (2pt)
(-1,0) coordinate (A3) {} circle (2pt) 
(0,-1) coordinate (A4) {} circle (2pt) 
(1,0) coordinate (A5) {} circle (2pt);
\draw (A1) node[above right] {$1$} -- (A4) node[below] {$4$};
\draw (A1) -- (A2) node[above] {$2$};
\draw (A1) -- (A3) node[left] {$3$};
\draw (A1) -- (A5) node[right] {$5$};

\draw[->, dotted, thick, blue] (185:1) arc (185:265:1);
\draw[->, dotted, thick, blue] (275:1) arc (275:355:1);
\draw[<-,dotted, thick, blue, rounded corners] (A3) 
to[out=-50,in=180] (0,-0.5)
to[out=0,in=-130] (A5);
\end{tikzpicture}
\end{minipage}
$\rightsquigarrow$
\begin{minipage}[b]{0.45\textwidth}
\centering
\begin{tikzpicture}[baseline=-.5ex]
\tikzstyle{state}=[draw, circle, inner sep =1.4pt, fill]

\node[state, label=below:{$I_2$}] (1) {};
\node[state, fill=black, label=below:{$I_1$}] (2) [right = of 1] {};
\node[state, label=below:{$I_3$}] (3) [right=of 2] {};

\draw[triple line] (2)-- node{\scalebox{1}{$<$}} (3);
\draw (1)--(2);
\draw (2)--(3);
\end{tikzpicture}
\end{minipage}
\end{figure}
\item $(\dynX, G, \dynY) = (\exdynD_n, \Z/2\Z, 
\exdynC_{n-2})$
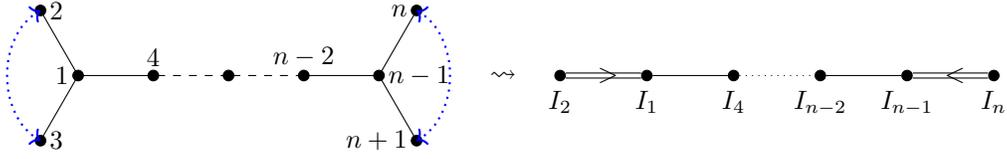
\label{fig_Dn_Cn-2}
\begin{figure}[H]
\begin{minipage}[b]{0.45\textwidth}
\centering
\begin{tikzpicture}[baseline=-.5ex]
\draw[fill]
(0,0) circle(2pt) coordinate (A1) {}
(120:1) circle(2pt) coordinate (A2) {}
(240:1) circle(2pt) coordinate (A3) {}
(4,0) circle (2pt) coordinate (A2n-1) {}
(4,0) ++(60:1) circle(2pt) coordinate (A2n) {}
(4,0) ++(-60:1) circle(2pt) coordinate (A2n+1) {};
\draw[fill]
(1,0) circle(2pt) coordinate (A4) {}
(2,0) circle(2pt) coordinate (An+1) {}
(3,0) circle (2pt) coordinate (A2n-2) {};
\draw (A2) node[right] {$2$} -- 
(A1) node[left] {$1$};
\draw  (A3) node[right] {$3$} -- (A1);
\draw (A4) -- (A1);
\draw[dashed] (An+1) node[below] {} -- (A4) node[above] {$4$};
\draw[dashed] (A2n-2) node[above] {$n-2$} -- (An+1);
\draw (A2n-2) -- (A2n-1) node[right] {$n-1$};
\draw (A2n) node[left] {$n$} -- (A2n-1);
\draw (A2n+1) node[left] {$n+1$} -- (A2n-1);

\draw[<->, dotted, thick, blue] (A2n) to [out = -30, in = 30] (A2n+1);
\draw[<->, dotted, thick, blue] (A2) to [out = 210, in = -210] (A3);
\end{tikzpicture}
\end{minipage}
$\rightsquigarrow$
\begin{minipage}[b]{0.45\textwidth}
\centering
\begin{tikzpicture}[baseline=-.5ex]
\tikzstyle{state}=[draw, circle, inner sep = 1.4pt, fill]

\node[state, label=below:{$I_1$}] (1) {};
\node[state, label=below:{$I_4$}] (2) [right = of 1] {};
\node[state, label=below:{$I_{n-2}$}] (3) [right = of 2] {};
\node[state, label=below:{$I_{n-1}$}] (4) [right =of 3] {};
\node[state, label=below:{$I_n$}] (5) [right =of 4] {};			
\node[state, label=below:{$I_2$}] (6) [left =of 1] {};
\draw (1)--(2)
(3)--(4);
\draw [dotted] (2)--(3);
\draw[double line] (4)-- node{\scalebox{1}{ $<$}} (5);
\draw[double line] (1)-- node{\scalebox{1}{ $>$}} (6);

\end{tikzpicture}
\end{minipage}
\end{figure}
\item $(\dynX, G, \dynY) = (\exdynD_n, \Z/2\Z, 
\dynA_{2(n-1)-1}^{(2)})$\label{fig_Dn_A2n-1_2}
\begin{figure}[H]
\begin{minipage}[b]{0.45\textwidth}
\centering
\begin{tikzpicture}[baseline=-.5ex]
\draw[fill]
(0,0) circle(2pt) coordinate (A1) {}
(120:1) circle(2pt) coordinate (A2) {}
(240:1) circle(2pt) coordinate (A3) {}
(4,0) circle (2pt) coordinate (A2n-1) {}
(4,0) ++(60:1) circle(2pt) coordinate (A2n) {}
(4,0) ++(-60:1) circle(2pt) coordinate (A2n+1) {};
\draw[fill]
(1,0) circle(2pt) coordinate (A4) {}
(2,0) circle(2pt) coordinate (An+1) {}
(3,0) circle (2pt) coordinate (A2n-2) {};
\draw (A2) node[right] {$2$} -- 
(A1) node[left] {$1$};
\draw  (A3) node[right] {$3$} -- (A1);
\draw (A4) -- (A1);
\draw[dashed] (An+1) node[below] {} -- (A4) node[above] {$4$};
\draw[dashed] (A2n-2) node[above] {$n-2$} -- (An+1);
\draw (A2n-2) -- (A2n-1) node[right] {$n-1$};
\draw (A2n) node[left] {$n$} -- (A2n-1);
\draw (A2n+1) node[left] {$n+1$} -- (A2n-1);

\draw[<->, dotted, thick, blue] (A2n) to [out = -30, in = 30] (A2n+1);
\end{tikzpicture}
\end{minipage}
$\rightsquigarrow$
\begin{minipage}[b]{0.45\textwidth}
\centering
\begin{tikzpicture}[baseline=-.5ex]
\tikzstyle{state}=[draw, circle, inner sep = 1.4pt, fill]

\node[state, fill=black,  label=below:{$I_1$}] (1) {};
\node[state, label=below:{$I_4$}] (2) [right = of 1] {};
\node[state, label=below:{$I_{n-2}$}] (3) [right = of 2] {};
\node[state, label=below:{$I_{n-1}$}] (4) [right =of 3] {};
\node[state, label=below:{$I_n$}] (5) [right =of 4] {};
\node[state, label=left:{$I_3$}] (6) [below left = 0.6cm and 0.6cm of 1] {};
\node[state, label=left:{$I_2$}] (7) [above left = 0.6cm and 0.6cm of 1] {};	

\draw (1)--(2)
(3)--(4)
(6)--(1)--(7);
\draw [dotted] (2)--(3);
\draw[double line] (4)-- node{\scalebox{1}{ $<$}} (5);

\end{tikzpicture}
\end{minipage}
\end{figure}
\item $(\dynX, G, \dynY) = (\exdynD_{2n}, \Z/2\Z, \exdynB_n)$\label{fig_D2n_Bn}
\begin{figure}[H]
\begin{minipage}[b]{0.45\textwidth}
\centering
\begin{tikzpicture}[baseline=-.5ex]
\draw[fill]
(0,0) circle(2pt) coordinate (A1) {}
(120:1) circle(2pt) coordinate (A2) {}
(240:1) circle(2pt) coordinate (A3) {}
(4,0) circle (2pt) coordinate (A2n-1) {}
(4,0) ++(60:1) circle(2pt) coordinate (A2n) {}
(4,0) ++(-60:1) circle(2pt) coordinate (A2n+1) {};
\draw[fill]
(1,0) circle(2pt) coordinate (A4) {}
(2,0) circle(2pt) coordinate (An+1) {}
(3,0) circle (2pt) coordinate (A2n-2) {};
\draw (A2) node[right] {$2$} -- (A1) node[left] {$1$};
\draw  (A3) node[right] {$3$} -- (A1);
\draw (A4) -- (A1);
\draw[dashed] (An+1) node[below] {$n+1$} -- (A4) node[above] {$4$};
\draw[dashed] (A2n-2) node[above] {$2n-2$} -- (An+1);
\draw (A2n-2) -- (A2n-1) node[right] {$2n-1$};
\draw (A2n) node[left] {$2n$} -- (A2n-1);
\draw (A2n+1) node[left] {$2n+1$} -- (A2n-1);

\draw[->, dotted, thick, blue] ([shift=(An+1)]30:1) arc (30:150:1) ;
\draw[->, dotted, thick, blue] ([shift=(An+1)]210:1) arc (210:330:1) ;

\end{tikzpicture}
\end{minipage}
$\rightsquigarrow$
\begin{minipage}[b]{0.45\textwidth}
\centering
\begin{tikzpicture}[baseline=-.5ex]
\tikzstyle{state}=[draw, circle, inner sep = 1.4pt, fill]

\node[state, label=below:{$I_1$}] (1) {};
\node[state, label=below:{$I_4$}] (2) [right = of 1] {};
\node[state, label=below:{$I_{n-1}$}] (3) [right = of 2] {};
\node[state, label=below:{$I_n$}] (4) [right =of 3] {};
\node[state, label=below:{$I_{n+1}$}] (5) [right =of 4] {};
\node[state, label=left:{$I_3$}] (6) [below left = 0.6cm and 0.6cm of 1] {};
\node[state, label=left:{$I_2$}] (7) [above left = 0.6cm and 0.6cm of 1] {};

\draw (1)--(2)
(3)--(4)
(6)--(1)--(7);
\draw [dotted] (2)--(3);
\draw[double line] (4)-- node{\scalebox{1}{ $>$}} (5);

\end{tikzpicture}
\end{minipage}
\end{figure}

\item $(\dynX, G, \dynY) = (\exdynD_{2n}, (\Z/2\Z)^2, \dynA_{2n-2}^{(2)})$
\label{fig_D2n_A2n-2_2}
\begin{figure}[H]
\begin{minipage}[b]{0.45\textwidth}
\centering
\begin{tikzpicture}[baseline=-.5ex]
\draw[fill]
(0,0) circle(2pt) coordinate (A1) {}
(120:1) circle(2pt) coordinate (A2) {}
(240:1) circle(2pt) coordinate (A3) {}
(4,0) circle (2pt) coordinate (A2n-1) {}
(4,0) ++(60:1) circle(2pt) coordinate (A2n) {}
(4,0) ++(-60:1) circle(2pt) coordinate (A2n+1) {};
\draw[fill]
(1,0) circle(2pt) coordinate (A4) {}
(2,0) circle(2pt) coordinate (An+1) {}
(3,0) circle (2pt) coordinate (A2n-2) {};
\draw (A2) node[right] {$2$} -- (A1) node[left] {$1$};
\draw (A3) node[right] {$3$} -- (A1);
\draw (A4) -- (A1);
\draw[dashed] (An+1) node[below] {$n+1$} -- (A4) node[above] {$4$};
\draw[dashed] (A2n-2) node[above] {$2n-2$} -- (An+1);
\draw (A2n-2) -- (A2n-1) node[right] {$2n-1$};
\draw (A2n) node[left] {$2n$} -- (A2n-1);
\draw (A2n+1) node[left] {$2n+1$} -- (A2n-1);

\node[above = 1cm of An+1] (Above) {};
\node[below = 1cm of An+1] (Below) {};
\draw[dotted, blue, thick] (Above)--(Below);
\node[blue] at (Above) {$\leftrightarrow$};
\node[blue, above = 0.1cm of Above] {$\tau_1$};

\node[left = 1cm of A1] (L) {};
\node[right = 1cm of A2n-1] (R) {};

\draw[dotted, blue, thick] (L) -- (R);
\node[blue] at (L) {$\tau_2 \updownarrow$};
\end{tikzpicture}
\end{minipage}
$\rightsquigarrow$
\begin{minipage}[b]{0.45\textwidth}
\centering
\begin{tikzpicture}[baseline=-.5ex]
\tikzstyle{state}=[draw, circle, inner sep = 1.4pt, fill]

\node[state, label=below:{$I_1$}, fill=black] (1) {};
\node[state, label=below:{$I_4$}] (2) [right = of 1] {};
\node[state, label=below:{$I_{n-1}$}] (3) [right = of 2] {};
\node[state, label=below:{$I_n$}] (4) [right =of 3] {};
\node[state, label=below:{$I_{n+1}$}] (5) [right =of 4] {};			
\node[state, label=below:{$I_2$}] (6) [left =of 1] {};
\draw (1)--(2)
(3)--(4);
\draw [dotted] (2)--(3);
\draw[double line] (4)-- node{\scalebox{1}{ $>$}} (5);
\draw[double line] (1)-- node{\scalebox{1}{ $>$}} (6);	
\end{tikzpicture}
\end{minipage}
\end{figure}
\item $(\dynX, G, \dynY) = (\exdynE_6, \Z/3\Z, \exdynG_2)$\label{fig_E6_Z3}
\begin{figure}[H]
\begin{minipage}[b]{0.45\textwidth}
\centering
\begin{tikzpicture}[baseline=-.5ex]	
\draw[fill] (0,0) circle (2pt) coordinate (A1) {} 
(60:2) circle (2pt) coordinate (A3) {} 
(180:2) circle (2pt) coordinate (A5) {} 
(300:2) circle (2pt) coordinate (A7) {};
\draw[fill] (60:1) circle (2pt) coordinate (A2) {} 
(180:1) circle (2pt) coordinate (A4) {} 
(300:1) circle (2pt) coordinate (A6) {};
\draw (A1) node[above left] {$1$} -- (A2) node[above left] {$2$};
\draw (A1) -- (A4) node[above left] {$4$};
\draw (A1) -- (A6) node[right] {$6$};
\draw (A3) node[above left] {$3$} -- (A2);
\draw (A5) node[above left] {$5$} -- (A4);
\draw (A7) node[right] {$7$} -- (A6);
\draw[->, dotted, thick, blue] (A1) edge[loop right] (A1);
\draw[->, dotted, thick, blue] (70:1) arc (70:170:1) node[midway, above left] 
{};
\draw[->, dotted, thick, blue] (65:2) arc (65:175:2) node[midway, above left] 
{};
\draw[->, dotted, thick, blue] (190:1) arc (190:290:1) node[midway, below left] 
{};
\draw[->, dotted, thick, blue] (185:2) arc (185:295:2) node[midway, below left] 
{};
\draw[->, dotted, thick, blue] (310:1) arc (310:410:1) node[midway, right] {};
\draw[->, dotted, thick, blue] (305:2) arc (305:415:2) node[midway, right] {};
\end{tikzpicture}
\end{minipage}
$\rightsquigarrow$
\begin{minipage}[b]{0.45\textwidth}
\centering
\begin{tikzpicture}[baseline=-.5ex]
\tikzstyle{state}=[draw, circle, inner sep = 1.4pt, fill]

\node[state, label=below:{$I_1$}] (1){};
\node[state, label=below:{$I_2$}] (2) [right = of 1] {};
\node[state, label=below:{$I_3$}] (3) [right=of 2] {};

\draw[triple line] (1)-- node{\scalebox{1}{$<$}} (2);
\draw (1)--(2);
\draw (2)--(3);
\end{tikzpicture}
\end{minipage}
\end{figure}
\item $(\dynX, G, \dynY) = (\exdynE_6, \Z/2\Z, 
\dynE_6^{(2)})$\label{fig_E6_E6_2}
\begin{figure}[H]
\begin{minipage}[b]{0.45\textwidth}
\centering
\begin{tikzpicture}[baseline=-.5ex]
\draw[fill] (0,0) circle (2pt) coordinate (A1) {} 
(90:2) circle (2pt) coordinate (A3) {} 
(180:2) circle (2pt) coordinate (A5) {} 
(0:2) circle (2pt) coordinate (A7) {};
\draw[fill] (90:1) circle (2pt) coordinate (A2) {} 
(180:1) circle (2pt) coordinate (A4) {}
(0:1) circle (2pt) coordinate (A6) {};
\draw (A1) node[above right] {$1$} -- (A2) node[left] {$2$};
\draw (A1) -- (A4) node[above] {$4$};
\draw (A1) -- (A6) node[above] {$6$};
\draw (A3) node[left] {$3$} -- (A2);
\draw (A5) node[above] {$5$} -- (A4);
\draw (A7) node[above] {$7$} -- (A6);

\node[above = 0.5cm of A3] (Above) {};
\node[below = 0.5cm of A1] (Below) {};

\draw[dotted, blue, thick] (Above)--(Below) {};
\node[blue] at (Above) {$\leftrightarrow$};
\node[blue, above = 0.1cm of Above] {$\tau$};

\end{tikzpicture}
\end{minipage}
$\rightsquigarrow$
\begin{minipage}[b]{0.45\textwidth}
\centering
\begin{tikzpicture}[baseline=-.5ex]
\tikzstyle{state}=[draw, circle, inner sep = 1.4pt, fill]

\node[state,label=below:{$I_3$}] (1) {};
\node[state,label=below:{$I_2$},fill=black] (2) [right = of 1] {};
\node[state,label = below:{$I_1$}] (3) [right = of 2] {};
\node[state,label = below:{$I_4$},fill=black] (4) [right=of 3] {};
\node[state,label = below:{$I_5$}] (5) [right=of 4] {};

\draw (1)--(2)
(2)--(3)
(4)--(5);
\draw[double line] (3)-- node{\scalebox{1}{$<$}} (4);

\end{tikzpicture}
\end{minipage}
\end{figure}
\item $(\dynX, G, \dynY) = (\exdynE_7, \Z/2\Z, \exdynF_4)$\label{fig_E7_F4}
\begin{figure}[H]
\begin{minipage}[b]{0.45\textwidth}
\centering
\begin{tikzpicture}[baseline=-.5ex]
\draw[fill] (0,0) circle (2pt) coordinate (A1) {} 
(180:2) circle (2pt) coordinate (A4) {} 
(0:2) circle (2pt) coordinate (A7) {};
\draw[fill] (90:1) circle (2pt) coordinate (A2) {} 
(180:1) circle (2pt) coordinate (A3) {} 
(180:3) circle (2pt) coordinate (A5) {} 
(0:1) circle (2pt) coordinate (A6) {} 
(0:3) circle (2pt) coordinate (A8) {};
\draw (A1) node[above right] {$1$} -- (A2) node[left] {$2$};
\draw (A1) -- (A3) node[above] {$3$};
\draw (A1) -- (A6) node[above] {$6$};
\draw (A4) node[above] {$4$} -- (A3);
\draw (A4) -- (A5) node[above] {$5$};
\draw (A7) node[above] {$7$} -- (A6);
\draw (A7) -- (A8) node[above] {$8$};

\node[above = 0.5cm of A2] (Above) {};
\node[below = 0.5cm of A1] (Below) {};

\draw[dotted, blue, thick] (Above)--(Below) {};
\node[blue] at (Above) {$\leftrightarrow$};
\node[blue, above = 0.1cm of Above] {$\tau$};

\end{tikzpicture}
\end{minipage}
$\rightsquigarrow$
\begin{minipage}[b]{0.45\textwidth}
\centering
\begin{tikzpicture}[baseline=-.5ex]
\tikzstyle{state}=[draw, circle, inner sep = 1.4pt, fill]

\node[state,fill=black,  label=below:{$I_5$}] (1) {};
\node[state, label=below:{$I_4$}] (2) [right = of 1] {};
\node[state, fill=black, label=below:{$I_3$}] (3) [right = of 2] {};
\node[state, label=below:{$I_1$}] (4) [right =of 3] {};
\node[state, fill=black, label=below:{$I_2$}] (5) [right =of 4] {};

\draw (1)--(2)
(2)--(3)
(4)--(5);
\draw[double line] (3)-- node{\scalebox{1}{$>$}} (4);

\end{tikzpicture}
\end{minipage}
\end{figure}
\end{enumerate}


\end{document}